\pdfoutput=1
\documentclass[11pt]{article}
\usepackage[left=1in,right=1in,top=1in,bottom=1in]{geometry}
\usepackage{times}
\usepackage{expl3}
\usepackage{cite}
\usepackage[table]{xcolor}
\usepackage{multirow}
\usepackage{stackengine} 
\usepackage{hhline}
\usepackage{lipsum}
\usepackage{titlesec}
\usepackage[all]{xy}
\usepackage{wrapfig}
\usepackage{enumerate}
\usepackage{epsfig}
\usepackage{tikz-cd}
\usepackage{amsmath}
\usepackage{tabularx}
\usepackage{array}
\usepackage{booktabs}
\usepackage{enumitem}
\usepackage{bbm}
\usepackage{calc}
\usepackage{graphicx}
\usepackage{amsmath}
\usepackage[title]{appendix}
\usepackage{amssymb}
\usepackage{epstopdf}
\usepackage{boldline}
\usepackage{arydshln}
\usepackage{calligra}
\usepackage{bm}
\usepackage{url}
\usepackage{blindtext}
\usepackage{accents}
\usepackage{amsthm}
\usepackage{amscd}

\newtheorem{definition}{Definition}
\newtheorem{axiom}{Axiom}
\newtheorem{theorem}{Theorem}
\newtheorem{proposition}{Proposition}
\newtheorem{corollary}{Corollary}
\newtheorem{lemma}{Lemma}

\newtheorem{example}{Example}
\newtheorem{remark}{Remark}
\usepackage{mathtools}
\usepackage{tikz-cd}
\usepackage{epstopdf}
\usepackage{balance}
\usepackage{thmtools}
\usepackage{thm-restate}
\usepackage{hyperref}
\usepackage{cleveref}
\usepackage[mathscr]{euscript}

\usepackage[ruled,vlined]{algorithm2e}
\newcommand{\ostar}{\mathbin{\mathpalette\make@circled\star}}

\makeatletter
\newcommand{\removelatexerror}{\let\@latex@error\@gobble}
\makeatother
\makeatletter
\newcommand*{\rom}[1]{\expandafter\@slowromancap\romannumeral #1@}
\makeatother

\ExplSyntaxOn
\newcommand\latinabbrev[1]{
  \peek_meaning:NTF . {
    #1\@}%
  { \peek_catcode:NTF a {
      #1.\@ }%
    {#1.\@}}}
\ExplSyntaxOff

\titleclass{\subsubsubsection}{straight}[\subsubsection]

\begin{document}
\vspace{1cm}
\title{The Operadic Spectrum and Obstructions to Spectral Base Change}
\vspace{1.8cm}
\author{Shih-Yu~Chang
\thanks{Shih-Yu Chang is with the Department of Applied Data Science,
San Jose State University, San Jose, CA, U. S. A. (e-mail: {\tt
shihyu.chang@sjsu.edu})
}}

\maketitle

\begin{abstract}
We introduce an operadic notion of spectrum for algebras over colored operads in a symmetric monoidal category. The construction is defined via a canonical Hochschild-type object together with an operadic residue, which together encode spectral information in a manner compatible with operadic composition. A central result of this work is that classical spectral invariants do not, in general, admit a natural base change in the operadic setting. More precisely, we show that there is no functorial procedure that transports spectra along strong monoidal functors while preserving their expected structural properties. This establishes a fundamental obstruction to spectral base change. To address this issue, we construct a universal operadic residue object and show that it induces a well-defined and functorial notion of operadic spectrum. We further prove that this construction is canonical and reduces to the classical spectrum in the case of the trivial operad. These results provide a conceptual foundation for spectral theory in operadic and higher algebraic contexts, and clarify the limitations of extending classical spectral invariants beyond the linear setting.
\end{abstract}

\tableofcontents

\section{Introduction}
\label{sec:introduction}

Classical spectral theory assigns to an operator $A$ its spectrum
$\sigma(A) \subseteq \mathbb{C}$. This invariant is remarkably powerful
for a single operator. However, consider a block matrix
\[
M = \begin{pmatrix} 0 & \alpha \\ \beta & 0 \end{pmatrix},
\]
where $\alpha$ and $\beta$ are linear operators. The classical spectrum
$\sigma(M)$ depends on the product $\alpha\beta$ in ways that cannot be
recovered from the individual spectra $\sigma(\alpha)$ and $\sigma(\beta)$
alone. More generally, when algebras are organized by operadic composition,
such componentwise spectral invariants fail to capture interaction effects
arising from compositional structure.

In parallel, Hochschild-type invariants and cotangent complexes provide
powerful tools for studying algebras over operads. These invariants capture
\emph{infinitesimal} and deformation-theoretic information, and have been
extensively developed in higher algebra. However, they remain insensitive
to \emph{global interaction phenomena} induced by operadic composition. 
We refer to Loday and Vallette~\cite{LodayVallette2012} for a comprehensive
treatment of algebraic operads and their homological properties.

\medskip

\noindent
\textbf{Core Observation.}
Hochschild-type invariants capture infinitesimal structure,
but fail to detect global interaction effects arising from operadic composition.
Any functorial extension of spectral theory must incorporate such interaction data.

\medskip

This leads to a fundamental obstruction to classical spectral invariants.

\medskip

\noindent
\textbf{Main Theorem (informal).}
There exists no assignment of spectral invariants to algebras over an operad
that depends only on componentwise spectra and is functorial with respect
to operadic composition and base change.

\medskip

The central thesis of this paper is that spectral behavior is fundamentally
\emph{operadic}: it depends not only on the underlying algebra, but also on
how components interact through compositional structure. To resolve the
above obstruction, we introduce a new invariant, the \emph{operadic spectrum},
which incorporates interaction-level information.

\paragraph{Main results.}
This paper develops a new framework, \emph{Spectral Operadic Calculus (SOC)},
which extends classical spectral theory to a compositional and functorial setting.
The main contributions are as follows:

\begin{itemize}
    \item \textbf{No-Go Theorem (Theorem~\ref{thm:no-go}):}
    Any spectral assignment depending only on componentwise spectra
    cannot be compatible with operadic composition or base change.
    This establishes that the limitations of classical spectral invariants
    are structural and cannot be resolved within the classical framework.
    
    \item \textbf{Operadic Residue (Theorem~\ref{thm:residue-universality}):}
    A universal object $\mathcal{O}_P^{\mathrm{res}}$ encoding compositional
    interaction. This construction provides the additional structure required
    to restore functoriality.
    
    \item \textbf{Operadic Spectrum (Definition~\ref{def:operadic-spectrum}):}
    Defined as $\sigma_P(A) := \mathrm{Hoch}_{\mathcal{M}}(A) \otimes_P \mathcal{O}_P^{\mathrm{res}}$,
    combining a Hochschild-type construction with the operadic residue.
    This yields a well-defined and functorial invariant that captures
    interaction-induced spectral behavior.
    
    \item \textbf{Base Change Theorem (Theorem~\ref{thm:base-change}):}
    For any strong monoidal functor $F$, there is a natural isomorphism
    $\sigma_{F_*(P)}(F_*(A)) \cong F(\sigma_P(A))$, demonstrating compatibility
    with categorical transport.
    
    \item \textbf{Operadic Spectral Mapping Theorem (Theorem~\ref{thm:spectral-mapping}):}
    For any holomorphic function $f$, $\sigma_P(f(A)) \cong f(\sigma_P(A))$,
    extending the classical spectral mapping principle to the operadic setting.
    
    \item \textbf{Universal Property (Theorem~\ref{thm:minimal_extension}):}
    The operadic spectrum admits a canonical natural transformation to any
    invariant satisfying functoriality, base change, and classical recovery,
    establishing it as the canonical invariant extending classical spectral theory.
\end{itemize}

\paragraph{Relation to previous work.}
Recent work by Hoang~\cite{Hoang2025} and by Harpaz--Hoang~\cite{HarpazHoang2026}
develops Hochschild and cotangent complex formalisms for algebras and operads,
providing a powerful framework for Quillen cohomology and deformation theory.
These approaches capture the infinitesimal structure of operadic objects via
tangent and cohomological methods.

Our approach is complementary in scope. While Hochschild-type and cotangent
complex invariants describe \emph{infinitesimal} behavior, we introduce a
\emph{spectral} invariant designed to capture \emph{global interaction effects}
arising from operadic composition.

A key step is the introduction of the \emph{operadic residue}, a universal
correction object that encodes compositional interactions. This object provides
a minimal extension that restores functoriality for spectral invariants, thereby
addressing obstructions that are invisible to purely infinitesimal methods.

\paragraph{Relationship to broader literature.}
Our work connects to several active areas of research beyond operadic algebra.
Lurie's framework of $\infty$-operads \cite{Lurie-HA2017} provides the
higher-categorical foundations for base change and functoriality of operadic
algebras; our construction of $\sigma_P(A)$ is designed to be compatible with
such $\infty$-categorical transport. Fresse's comprehensive treatment of
operadic bar constructions \cite{Fresse} underpins our Hochschild-type object
$\mathrm{Hoch}_{\mathcal{M}}(A)$ and its functoriality properties.

A deeper conceptual connection exists with Goodwillie calculus
\cite{Goodwillie2003,Goodwillie1992}, which constructs Taylor towers for
functors. While Goodwillie calculus approximates functors by polynomial
functors, our operadic residue $\mathcal{O}_P^{\mathrm{res}}$ provides a
universal correction for spectral transport under operadic composition.
Both frameworks address the failure of naive invariants to capture
compositional structure, but Goodwillie focuses on homotopy functors while
we focus on spectral invariants.

Connections to noncommutative geometry \cite{Connes1994} arise via the
Gelfand transform example (Section~\ref{subsec:gelfand-base-change}).
Connes' spectral triples $(A,H,D)$ encode geometric information via
the spectrum of the Dirac operator $D$. Our operadic spectrum can be
viewed as a compositional generalization: for an operadic algebra $A$,
the operadic spectrum $\sigma_P(A)$ plays a role analogous to the
spectrum of $D$, but with the ability to handle interacting components.
Future work will explore a precise relationship between $\sigma_P(A)$
and spectral triples for network geometries.

Finally, our No-Go Theorem (Theorem~\ref{thm:no-go}) echoes impossibility
results in other areas of mathematics, such as the failure of locality
in quantum mechanics (Bell's theorem) and the non-existence of certain
functorial invariants in algebraic topology \cite{Bell1964,StolzTeichner2011}.
The structural parallel suggests that operadic composition imposes
constraints analogous to non-locality, requiring the residue correction
$\mathcal{O}_P^{\mathrm{res}}$.

\subsection{Failure of Componentwise Spectral Invariants}

A fundamental limitation of classical spectral invariants is that they do not
detect interaction effects arising from composition. From an operadic perspective,
a composite object is formed by gluing components along operations in an operad.
The resulting interaction patterns generate new structural contributions that
are invisible to componentwise invariants.

In particular, two algebras over an operad may have identical componentwise
spectral data, yet differ in their compositional behavior. This shows that
no invariant depending solely on component spectra can capture global structure.

These failures become more pronounced under categorical operations such as
base change, where spectral data is not preserved in a controlled or functorial
manner. This demonstrates that classical spectral theory lacks a mechanism to
track compositional interaction.

\subsection{Operadic Interaction and Spectral Obstruction}

We formalize compositional structure using colored operads. In this framework,
different components are assigned colors, and interactions are encoded by
operadic compositions.

Such compositions generate \emph{interaction paths} that propagate information
across components. These interaction-induced contributions produce new spectral
phenomena that are not detected by classical invariants or Hochschild-type
constructions.

We show that this leads to a fundamental obstruction: classical spectral
invariants cannot be made functorial with respect to operadic composition.
This obstruction is intrinsic and necessitates an extension of the invariant.

\subsection{Conceptual Roadmap}

The logical structure of the paper follows a sequence of conceptual steps,
each corresponding to a precise stage in the construction.

\begin{enumerate}

\item \textbf{Identifying the obstruction (Section~\ref{sec:The Failure of Classical Spectrum under Operadic Base Change}).}
We begin by formulating a set of desiderata for a generalized spectral invariant,
and analyze naive extensions based on componentwise spectra.
We then prove the No-Go Theorem (Theorem~\ref{thm:no-go}),
showing that any invariant depending only on component spectra
fails to be compatible with operadic composition and base change.

\item \textbf{Constructing the correction (Section~\ref{sec:The Operadic Residue as a Universal Correction Object}).}
To resolve this obstruction, we introduce the operadic residue
(Theorem~\ref{thm:residue-universality}),
a universal object encoding compositional interaction.
This provides the additional structure required to restore functoriality.

\item \textbf{Defining the invariant (Section~\ref{sec:The Operadic Spectrum}).}
We combine the Hochschild-type object
$\mathrm{Hoch}_{\mathcal{M}}(A)$
with the operadic residue via a balanced tensor product,
thereby defining the operadic spectrum
(Definition~\ref{def:operadic-spectrum}).
We then establish its well-definedness and functoriality.

\item \textbf{Establishing fundamental properties (Sections~\ref{sec:Recovery and Minimality of the Operadic Spectrum}--\ref{sec:The Operadic Spectral Mapping Theorem}).}
We show that the operadic spectrum recovers classical spectral invariants,
prove compatibility with categorical transport via the Base Change Theorem
(Theorem~\ref{thm:base-change}),
and extend the spectral mapping principle
(Theorem~\ref{thm:spectral-mapping}).
We also establish its universal property (Theorem~\ref{thm:minimal_extension}).

\item \textbf{Reconstruction and examples (Sections~\ref{sec:Resolvent Theory and Reconstruction} and~\ref{sec:Examples and Structural Illustrations}).}
We develop a decomposition-reconstruction principle for the operadic spectrum
(Theorem~\ref{thm:colored-spectral-decomposition-reconstruction})
and illustrate the theory through representative examples:
the trivial operad (Section~\ref{subsec:trivial-operad}),
matrix block constructions (Section~\ref{subsec:matrix-block-operads}),
network operators (Section~\ref{subsec:network-operators}),
and base change via the Gelfand transform (Section~\ref{subsec:gelfand-base-change}).

\end{enumerate}

This progression shows that the operadic spectrum is a natural extension
of classical spectral invariants once compositional structure is taken into account.

\begin{remark}
The author is solely responsible for the mathematical insights and theoretical directions proposed in this work. AI tools, including OpenAI's ChatGPT and DeepSeek models, were employed solely to assist in verifying ideas, organizing references, and ensuring internal consistency of exposition~\cite{chatgpt2025,deepseek2025}. \\
\end{remark}


\section{Preliminaries and Structural Framework}\label{sec:Preliminaries and Structural Framework}

\subsection{Classical Spectrum and Resolvent Theory}\label{subsec:Classical Spectrum and Resolvent Theory}

In this subsection, we briefly recall the classical notions of the spectrum, resolvent, and the spectral mapping theorem, which serve as the baseline for our subsequent operadic generalization.

\paragraph{Spectrum and Resolvent.}
Let $\mathcal{H}$ be a complex Banach space and let $\mathcal{B}(\mathcal{H})$ denote the algebra of bounded linear operators on $\mathcal{H}$. For an operator $A \in \mathcal{B}(\mathcal{H})$, the \emph{resolvent set} is defined by
\[
\rho(A) := \{ z \in \mathbb{C} \mid zI - A \text{ is invertible in } \mathcal{B}(\mathcal{H}) \}.
\]
The \emph{spectrum} of $A$ is the complement
\[
\sigma(A) := \mathbb{C} \setminus \rho(A).
\]
For each $z \in \rho(A)$, the \emph{resolvent operator} is defined as
\[
R_A(z) := (zI - A)^{-1}.
\]

\paragraph{Basic Properties.}
The spectrum $\sigma(A)$ is a nonempty compact subset of $\mathbb{C}$, and the resolvent map
\[
R_A : \rho(A) \to \mathcal{B}(\mathcal{H}), \quad z \mapsto (zI - A)^{-1}
\]
is analytic on $\rho(A)$. Moreover, the resolvent satisfies the \emph{resolvent identity}
\[
R_A(z) - R_A(w) = (w - z) R_A(z) R_A(w), \qquad z,w \in \rho(A),
\]
which encodes the algebraic structure underlying spectral analysis and plays a crucial role in the functional calculus.

\paragraph{Holomorphic Functional Calculus.}
Let $U \subset \mathbb{C}$ be an open set containing $\sigma(A)$, and let $f: U \to \mathbb{C}$ be holomorphic. The operator $f(A)$ is defined via the Cauchy integral formula
\[
f(A) := \frac{1}{2\pi i} \int_{\Gamma} f(z) (zI - A)^{-1} \, dz,
\]
where $\Gamma$ is a positively oriented contour that encloses $\sigma(A)$ and is contained in $U$. This definition is independent of the choice of contour and yields a bounded linear operator $f(A) \in \mathcal{B}(\mathcal{H})$.

\paragraph{Spectral Mapping Theorem.}
The classical spectral mapping theorem states that for any holomorphic function $f$ defined on a neighborhood of $\sigma(A)$,
\[
\sigma(f(A)) = f(\sigma(A)).
\]
Equivalently, for any $\lambda \in \mathbb{C}$,
\[
\lambda \in \sigma(f(A)) \quad \Longleftrightarrow \quad \lambda = f(\mu) \text{ for some } \mu \in \sigma(A).
\]
This theorem ensures that the spectrum behaves functorially under analytic functional calculus and provides a complete description of how the spectrum transforms under analytic operations.

\paragraph{Conceptual Remark and Path Forward.}
The classical spectrum $\sigma(A)$ and resolvent $R_A(z)$ capture the intrinsic spectral structure of a single operator $A$ within a fixed category. This theory is remarkably powerful when working with a single operator or a commuting family of operators. However, as we will demonstrate in Section~\ref{sec:The Failure of Classical Spectrum under Operadic Base Change}, these classical notions are fundamentally inadequate when we consider:
\begin{enumerate}
    \item \textbf{Operadic composition}: systems where multiple operators interact via operadic composition maps,
    \item \textbf{Categorical base change}: functorial transport of spectral data across different categories (e.g., quantization, discretization, complexification).
\end{enumerate}
These limitations motivate the development of new spectral invariants—the operadic residue and operadic spectrum—which form the foundation of Spectral Operadic Calculus (SOC) and are the main subject of this paper.

\subsection{Colored Operads and Their Algebras}\label{subsec:Colored Operads and Their Algebras}

In this subsection, we recall the notion of a colored operad and its algebras in a symmetric monoidal category. These structures provide the formal framework for encoding multi-typed operations and their compositions, which will play a central role in the formulation of operadic spectral invariants.

\paragraph{Colored Operads.}
Let $\mathcal{M}$ be a symmetric monoidal category with tensor product $\otimes$ and unit object $\mathbf{1}$. Let $C$ be a set of \emph{colors}. A \emph{$C$-colored operad} $P$ in $\mathcal{M}$ consists of the following data:

\begin{itemize}
    \item For each tuple $(c_1, \dots, c_n; c)$ with $c_i, c \in C$, an object
    \[
    P(c_1, \dots, c_n; c) \in \mathcal{M},
    \]
    whose elements are interpreted as $n$-ary operations with input colors $(c_1, \dots, c_n)$ and output color $c$. For $n = 0$, we denote the object of nullary operations by $P(\emptyset; c)$.

    \item \emph{Composition maps}: for each collection of operations
    \[
    \phi \in P(d_1, \dots, d_m; c), \quad
    \psi_i \in P(c_{i1}, \dots, c_{i k_i}; d_i),
    \]
    a composite operation
    \[
    \phi \circ (\psi_1, \dots, \psi_m) \in P(c_{11}, \dots, c_{m k_m}; c),
    \]
    encoded categorically as a morphism
    \[
    \circ: P(d_1, \dots, d_m; c) \otimes \bigotimes_{j=1}^m P(c_{j1}, \dots, c_{j k_j}; d_j)
    \longrightarrow
    P(c_{11}, \dots, c_{m k_m}; c).
    \]

    \item \emph{Unit maps}: for each color $c \in C$, a morphism
    \[
    \eta_c: \mathbf{1} \to P(c; c),
    \]
    corresponding to identity operations.

    \item A compatible action of the symmetric group $\mathfrak{S}_n$ on each $P(c_1, \dots, c_n; c)$, expressing permutation of inputs.
\end{itemize}

These data are required to satisfy the usual \emph{associativity}, \emph{unitality}, and \emph{equivariance} axioms, ensuring that iterated compositions are well-defined and coherent.

\paragraph{Algebras over a Colored Operad.}
Let $P$ be a $C$-colored operad in $\mathcal{M}$. A \emph{$P$-algebra} $A$ consists of:
\begin{itemize}
    \item For each color $c \in C$, an object $A_c \in \mathcal{M}$.
    \item For each operation $\phi \in P(c_1, \dots, c_n; c)$, a structure morphism
    \[
    \phi_A: A_{c_1} \otimes \cdots \otimes A_{c_n} \to A_c,
    \]
    natural in $\phi$. For $n = 0$, this gives a morphism $\mathbf{1} \to A_c$.
\end{itemize}
These structure morphisms must be compatible with operadic composition, units, and symmetric group actions. More explicitly, if $\phi \circ (\psi_1, \dots, \psi_n)$ is a composite operation, then
\[
(\phi \circ (\psi_1, \dots, \psi_n))_A = \phi_A \circ (\psi_{1,A} \otimes \cdots \otimes \psi_{n,A}),
\]
and the unit map $\eta_c: \mathbf{1} \to P(c;c)$ induces the identity morphism $\mathrm{id}_{A_c}: A_c \to A_c$.

\begin{example}[The Trivial Operad.]
Let $C = \{*\}$ consist of a single color. The \emph{trivial operad} $\mathbb{I}$ is defined by
\[
\mathbb{I}(n) = \mathbf{1} \quad \text{for all } n \ge 1,
\]
with composition induced by the canonical identification of unit objects. Algebras over $\mathbb{I}$ correspond to objects of $\mathcal{M}$ equipped with no additional algebraic structure. In this case, the operadic formalism reduces to the classical setting, and we will later see (in Theorem~\ref{thm:recovery}) that the operadic spectrum recovers the classical spectrum.
\end{example}

\begin{example}[ The Endomorphism Operad.]
Given a collection of objects $\{A_c\}_{c \in C}$ in $\mathcal{M}$, the \emph{endomorphism operad} $\mathrm{End}_A$ is defined by
\[
\mathrm{End}_A(c_1, \dots, c_n; c) := \mathrm{Hom}_{\mathcal{M}}(A_{c_1} \otimes \cdots \otimes A_{c_n}, A_c),
\]
with composition given by composition of morphisms in $\mathcal{M}$ and tensor product. The unit maps are the identity morphisms. A $P$-algebra structure on $\{A_c\}$ is equivalently a morphism of operads $P \to \mathrm{End}_A$. This perspective will be essential when we construct concrete examples of operadic algebras from operator networks in future works.
\end{example}

\begin{example}[The Matrix Block Operad (Two-Color Interaction]
\label{ex:two-color-operad}
Let $C = \{1,2\}$ and let $\mathcal{M} = \mathrm{Vect}_{\mathbb{C}}$. Define an operad $P$ where:
\[
P(1;1) = P(2;2) = \mathbb{C}, \qquad
P(1,1;1) = \mathbb{C}, \qquad
P(2,2;2) = \mathbb{C},
\]
and
\[
P(1,2;1) = P(2,1;2) = \mathbb{C},
\]
with all other operations trivial.

A $P$-algebra consists of vector spaces $V_1$ and $V_2$ together with linear maps
\[
A_{11}: V_1 \to V_1, \quad
A_{22}: V_2 \to V_2, \quad
A_{12}: V_2 \to V_1, \quad
A_{21}: V_1 \to V_2,
\]
which can be assembled into a block operator
\[
A =
\begin{pmatrix}
A_{11} & A_{12} \\
A_{21} & A_{22}
\end{pmatrix}.
\]
\end{example}

\textbf{Operadic composition structure.}
The operad encodes both \emph{intra-color} operations (e.g., $P(1,1;1)$) and \emph{inter-color interactions} (e.g., $P(1,2;1)$). These interactions correspond to the off-diagonal blocks and are responsible for coupling between $V_1$ and $V_2$.

The operadic composition can be visualized by considering an off-diagonal operation $\phi \in P(1,2;1)$ composed with diagonal operations $\psi_1 \in P(1;1)$ and $\psi_2 \in P(2;2)$. This composition $\phi \circ (\psi_1, \psi_2)$ yields another operation in $P(1,2;1)$, reflecting how inter-color interactions are modulated by intra-color dynamics:
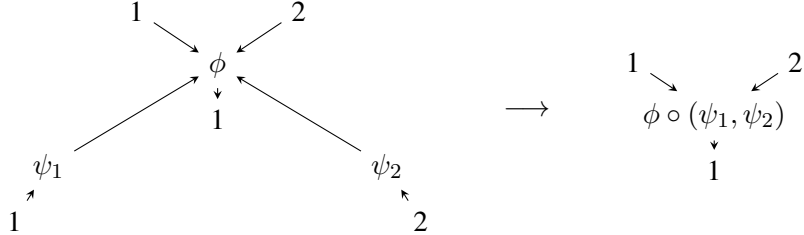
\[
\begin{tikzpicture}[baseline=(current bounding box.center), >=stealth, scale=0.9]
  \node (op) at (0,0) {$\phi$};
  \node (in1) at (-1.2,0.8) {1};
  \node (in2) at (1.2,0.8) {2};
  \node (out) at (0,-0.8) {1};
  \node (psi1) at (-2.5,-1.5) {$\psi_1$};
  \node (psi2) at (2.5,-1.5) {$\psi_2$};
  \node (in11) at (-3,-2.3) {1};
  \node (in22) at (3,-2.3) {2};
  \draw[->] (in1) -- (op);
  \draw[->] (in2) -- (op);
  \draw[->] (op) -- (out);
  \draw[->] (psi1) -- (op);
  \draw[->] (psi2) -- (op);
  \draw[->] (in11) -- (psi1);
  \draw[->] (in22) -- (psi2);
\end{tikzpicture}
\qquad\longrightarrow\qquad
\begin{tikzpicture}[baseline=(current bounding box.center), >=stealth, scale=0.9]
  \node (comp) at (0,0) {$\phi \circ (\psi_1,\psi_2)$};
  \node (in1) at (-1.2,0.8) {1};
  \node (in2) at (1.2,0.8) {2};
  \node (out) at (0,-0.8) {1};
  \draw[->] (in1) -- (comp);
  \draw[->] (in2) -- (comp);
  \draw[->] (comp) -- (out);
\end{tikzpicture}
\]

\medskip

\noindent
\textbf{Conceptual significance.}
This example illustrates how operadic composition generates spectral interactions between distinct colors. Although the classical spectra $\sigma(A_{11})$ and $\sigma(A_{22})$ capture the intrinsic behavior of each component, the off-diagonal maps $A_{12}$ and $A_{21}$ introduce interactions that can fundamentally alter the spectrum of the total system. 

In particular, the spectrum of $A$ is not determined solely by $\sigma(A_{11}) \cup \sigma(A_{22})$, but depends crucially on the operadic composition data. This will serve as the primary motivating example for the operadic residue introduced in Section~\ref{sec:The Operadic Residue as a Universal Correction Object}.

\medskip

\paragraph{Colored Operads as Syntax for Compositional Systems.}
Colored operads provide a formal language for describing systems built from components with typed inputs and outputs. The colors specify the types, and the operations represent ways to compose components. An algebra $A$ assigns concrete objects to each color and interprets each operation as a concrete morphism. In the context of this paper, the objects $A_c$ will be operators (or operator algebras), and the operad composition will encode how operators are combined to form networks, feedback loops, and other compositional structures.

\paragraph{Conceptual Remark: Operadic Interaction vs. Tensor Interaction.}
A crucial observation for what follows is that the structure of a $P$-algebra encodes \emph{interactions} that go beyond the mere tensor product of individual components. Classical spectral theory, which focuses on a single operator or a commuting family, is essentially the study of algebras over the trivial operad. When $P$ is nontrivial—particularly when it contains operations that mix different colors—the resulting algebraic structure gives rise to spectral phenomena that cannot be detected by classical invariants. For instance, in the matrix block operad (Example 3), the off-diagonal operations $P(1,2;1)$ and $P(2,1;2)$ encode interactions between the two blocks; these interactions can cause the spectrum of the composite system to differ nontrivially from the spectra of the individual blocks. The No-Go Theorem of Section~\ref{sec:The Failure of Classical Spectrum under Operadic Base Change} will demonstrate that any spectral invariant that ignores such operadic composition data is inherently incomplete. This observation motivates the development of the operadic residue and operadic spectrum in the sections that follow.

\subsection{Monoidal Categories and Base Change}\label{subsec:Monoidal Categories and Base Change}

We briefly recall the notion of a strong monoidal functor and the corresponding base change principle for operads and their algebras. This will be used repeatedly when transporting operadic structures between different symmetric monoidal categories, and is essential for the Base Change Theorem (Theorem~\ref{thm:base-change}) later in this paper.

\begin{definition}[Strong monoidal functor]
Let $(\mathcal{M},\otimes_{\mathcal M},\mathbf{1}_{\mathcal M})$ and $(\mathcal{N},\otimes_{\mathcal N},\mathbf{1}_{\mathcal N})$ be symmetric monoidal categories. A functor
\[
F:\mathcal{M}\to\mathcal{N}
\]
is called a \emph{strong monoidal functor} if it is equipped with natural isomorphisms
\[
\phi_{X,Y}:F(X)\otimes_{\mathcal N}F(Y)\xrightarrow{\;\cong\;}F(X\otimes_{\mathcal M}Y),
\]
for all objects $X,Y\in\mathcal{M}$, together with an isomorphism
\[
\phi_0:\mathbf{1}_{\mathcal N}\xrightarrow{\;\cong\;}F(\mathbf{1}_{\mathcal M}),
\]
such that the usual associativity, unit, and symmetry coherence diagrams commute.
\end{definition}

Intuitively, a strong monoidal functor preserves the tensor product and unit object up to coherent isomorphism. Thus it transports algebraic structures defined in terms of the monoidal product from $\mathcal{M}$ to corresponding structures in $\mathcal{N}$.

\begin{example}[Extension of scalars]
A canonical family of examples arises from extension of scalars. For instance, if $R\to S$ is a homomorphism of commutative rings, then the functor
\[
-\otimes_R S:\mathrm{Mod}_R\to \mathrm{Mod}_S
\]
is strong monoidal, since there are canonical isomorphisms
\[
(M\otimes_R S)\otimes_S (N\otimes_R S)\cong (M\otimes_R N)\otimes_R S
\]
and
\[
S\cong R\otimes_R S.
\]
Hence operads and operadic algebras over $R$ may be transported to operads and algebras over $S$.
\end{example}

\begin{example}[Base change in operator theory]
The following examples are particularly relevant to the operator-theoretic setting of this paper:
\begin{itemize}
    \item \textbf{Quantization}. Let $\mathcal{M} = \mathrm{Vect}_{\mathbb{R}}$ (real vector spaces) and $\mathcal{N} = \mathrm{Hilb}$ (complex Hilbert spaces). The complexification functor $F = (-) \otimes_{\mathbb{R}} \mathbb{C}$ is strong monoidal. It sends a classical mechanical system to a quantum system and transports operadic structures accordingly.
    
    \item \textbf{Discretization}. Let $\mathcal{M} = \mathrm{Cont}(X)$ (continuous functions on a space $X$) and $\mathcal{N} = \mathrm{Mat}_n(\mathbb{C})$ (matrices). A discretization functor $F$ (e.g., finite element approximation) is typically lax monoidal, but under suitable approximations can be treated as strong.
    
    \item \textbf{Gelfand Transform}. Let $\mathcal{M} = \mathrm{C}^*(X)$ (commutative C*-algebras) and $\mathcal{N} = \mathrm{Cont}(\Sigma)$ (continuous functions on the spectrum). The Gelfand transform $F$ is a strong monoidal equivalence between the category of commutative C*-algebras and the category of compact Hausdorff spaces with continuous functions.
    
    \item \textbf{Forgetful Functor}. Let $\mathcal{M} = \mathrm{Ban}$ (Banach spaces with projective tensor product) and $\mathcal{N} = \mathrm{Vect}_{\mathbb{C}}$ (vector spaces). The forgetful functor $F$ is strong monoidal and will be used to relate operator-theoretic constructions to purely algebraic ones.
\end{itemize}
\end{example}

We now explain the corresponding base change construction for colored operads. Let $P$ be a colored operad in $\mathcal{M}$. Applying $F$ objectwise to the operations of $P$ and using the structural isomorphisms of $F$, one obtains a colored operad in $\mathcal{N}$, denoted by
\[
F_*P.
\]
More concretely, if
\[
P(c_1,\dots,c_n;c)
\]
denotes the object of $n$-ary operations from $(c_1,\dots,c_n)$ to $c$, then one defines
\[
(F_*P)(c_1,\dots,c_n;c):=F\bigl(P(c_1,\dots,c_n;c)\bigr),
\]
with operadic composition in $F_*P$ induced from the composition in $P$ together with the monoidal structure maps of $F$.

Likewise, if $A$ is a $P$-algebra in $\mathcal{M}$, then $F(A)$ naturally acquires the structure of an $F_*P$-algebra in $\mathcal{N}$. Indeed, the structure maps
\[
P(c_1,\dots,c_n;c)\otimes A(c_1)\otimes\cdots\otimes A(c_n)\longrightarrow A(c)
\]
are sent by $F$ to morphisms
\[
F\bigl(P(c_1,\dots,c_n;c)\bigr)\otimes F(A(c_1))\otimes\cdots\otimes F(A(c_n))
\longrightarrow F(A(c)),
\]
where the tensor products are rearranged using the coherence isomorphisms of the strong monoidal functor $F$. These maps satisfy the operadic axioms because $F$ preserves the monoidal structure coherently.

The fundamental base change result may therefore be summarized as follows.

\begin{proposition}[Base change for operads and algebras]
\label{prop:base-change}
Let $F:\mathcal{M}\to\mathcal{N}$ be a strong monoidal functor between symmetric monoidal categories. Then:
\begin{enumerate}
    \item every colored operad $P$ in $\mathcal{M}$ induces a colored operad $F_*P$ in $\mathcal{N}$;
    \item every $P$-algebra $A$ in $\mathcal{M}$ induces an $F_*P$-algebra $F(A)$ in $\mathcal{N}$;
    \item these constructions are functorial in both operads and their algebras.
\end{enumerate}
\end{proposition}

\begin{proof}
Let $P$ be a colored operad in $\mathcal{M}$. Define $F_*P$ by applying $F$ objectwise:
\[
(F_*P)(c_1,\dots,c_n;c):=F\bigl(P(c_1,\dots,c_n;c)\bigr).
\]

The unit maps are obtained by composing the unit of $P$ with the unit isomorphism
$\mathbf{1}_{\mathcal N}\cong F(\mathbf{1}_{\mathcal M})$.
The operadic compositions are defined by combining the structure map of $P$
with the coherence isomorphisms of the strong monoidal structure on $F$, namely the natural maps
\[
F(X)\otimes F(Y)\xrightarrow{\cong}F(X\otimes Y),
\]
iterated to arbitrary tensor products.

The symmetric group actions are transported by functoriality:
$\sigma^*_{F_*P}=F(\sigma^*_P)$.
The operad axioms for $F_*P$ follow from those of $P$ together with coherence of the monoidal structure of $F$.
This proves (1).

For (2), let $A$ be a $P$-algebra with structure maps
\[
P(c_1,\dots,c_n;c)\otimes A(c_1)\otimes\cdots\otimes A(c_n)\to A(c).
\]
Define $F(A)(c):=F(A(c))$.
The $F_*P$-algebra structure on $F(A)$ is obtained by applying $F$
to the above maps and using the coherence isomorphisms
\[
F(X)\otimes F(Y)\xrightarrow{\cong}F(X\otimes Y)
\]
to combine all inputs.
The algebra axioms follow from those of $A$ and coherence of $F$.

For (3), functoriality in $P$ and $A$ follows by applying $F$ to morphisms
and using that $F$ preserves identities and compositions.
\end{proof}
This construction is the standard pushforward of operads along a strong monoidal functor,
providing a functorial mechanism for transporting operadic structures across monoidal contexts.
In particular, it allows one to pass from operads and their algebras in $\mathcal{M}$
to corresponding structures in $\mathcal{N}$ in a way compatible with composition and symmetry.

This principle serves as the formal foundation for the constructions developed in this paper,
where operadic data is transported across different categorical settings
while preserving its compositional structure.

\begin{remark}[Connection to higher categories]
In higher-categorical language, this statement is a manifestation of the functoriality of operads and algebra objects under symmetric monoidal functors; see, for example, see \cite{Lurie-HA2017} (operadic functoriality under symmetric monoidal functors). For the purposes of this paper, the key point is that strong monoidality is precisely the hypothesis needed to ensure that operadic composition and algebra actions survive base change.
\end{remark}

\begin{remark}[Why base change matters for spectral theory]
Base change functors arise naturally in spectral theory when we:
\begin{itemize}
    \item Quantize a classical system (real $\to$ complex),
    \item Discretize a continuous operator (functions $\to$ matrices),
    \item Complexify a real operator (real $\to$ complex),
    \item Pass to the Gelfand transform (C*-algebra $\to$ continuous functions).
\end{itemize}
The Base Change Theorem (Theorem~\ref{thm:base-change}) will show that the operadic spectrum $\sigma_P(A)$ is compatible with such transport: applying $F$ to $\sigma_P(A)$ yields the operadic spectrum of the base-changed algebra. This compatibility is essential for the consistency of Spectral Operadic Calculus across different mathematical contexts.
\end{remark}

\begin{remark}[Connection to the No-Go Theorem]
The No-Go Theorem (Section~\ref{sec:The Failure of Classical Spectrum under Operadic Base Change}) demonstrates that any spectral invariant that does not incorporate operadic composition data cannot be functorial under base change. The base change machinery introduced here provides the precise categorical language in which that impossibility is formulated, and the operadic residue $\mathcal{O}_P^{\mathrm{res}}$ we construct in Section~\ref{sec:The Operadic Residue as a Universal Correction Object} will be precisely the additional structure needed to restore functoriality.
\end{remark}

We now recall the construction of Hochschild homology for algebras over colored operads, which will serve as the foundation for our operadic spectrum. This construction generalizes the classical Hochschild homology of associative algebras to the operadic setting, capturing the derived structure of operadic algebras.

\paragraph{The Hochschild Complex.}
Let $P$ be a $C$-colored operad in a symmetric monoidal category $\mathcal{M}$, and let $A$ be a $P$-algebra. The Hochschild complex of $A$ relative to $P$, denoted $\mathrm{Hoch}_{\mathcal{M}}(A)$, is defined via the operadic bar construction.

More precisely, it is given by the geometric realization of a simplicial object
\[
B_\bullet(A,P,A),
\]
where the $n$-simplices are
\[
B_n(A,P,A)
=
\coprod_{c_0,\dots,c_n \in C}
A(c_0) \otimes P(c_1,\dots,c_n; c_0) \otimes A(c_1) \otimes \cdots \otimes A(c_n).
\]

\medskip

\noindent
\textbf{Operadic description.}
This construction admits a natural operadic interpretation. The simplicial structure is governed by operadic composition and units:

\begin{itemize}
    \item \emph{Face maps} are induced by operadic composition, corresponding to inserting operations into input slots and composing along adjacent layers.
    \item \emph{Degeneracy maps} are induced by unit maps $\eta_c:\mathbf{1}\to P(c;c)$, inserting identity operations.
\end{itemize}

This perspective shows that the Hochschild complex is an operadic generalization of the classical bar construction, in which operadic composition replaces the associative multiplication.

\medskip

\paragraph{Enriched setting.}
In our analytic applications, we work in the enriched setting where $\mathcal{M}$ is the category of Banach spaces with bounded linear maps, equipped with the projective tensor product. The operad $P$ is assumed to be \emph{enriched over Banach spaces}, meaning each operation space $P(c_1,\dots,c_n;c)$ is a Banach space, and all structure maps (composition, unit, symmetric group actions) are bounded linear maps.

We assume that $\mathcal{M}$ admits geometric realizations of simplicial objects, and that these are compatible with the monoidal structure. The geometric realization of the simplicial Hochschild complex is taken as the homotopy colimit in this enriched setting, followed by completion in the Banach norm; we refer to this as the \emph{analytic realization}. This ensures that the resulting object $\mathrm{CH}_\bullet(P,A)$ is a Banach space, and that convergence of the associated spectral sequences is governed by norm estimates rather than purely algebraic conditions.

\begin{remark}[Comparison with derived Hochschild homology]
\label{rem:comparison-derived}
In the derived setting, Hochschild-type constructions for operadic algebras have been studied extensively, e.g., by Hoang \cite{Hoang2025}, where they serve as invariants for deformation theory and homotopical algebra. Our construction differs in two key respects: 
\begin{enumerate}
    \item We work in an \emph{enriched} (Banach) analytic setting, where geometric realization is taken as an analytic homotopy colimit with attention to norm completeness.
    \item Our primary interest lies in the \emph{spectral} information encoded in the Hochschild object—specifically, the spectrum of the associated boundary operator—rather than its cohomological or derived structure.
\end{enumerate}
Our construction should be viewed as an analytic refinement of the classical Hochschild complex, rather than a fully derived Hochschild homology. In particular, we do not require a model structure on Banach spaces, but instead work with explicit simplicial resolutions equipped with norm control. These differences reflect the analytic nature of our operadic spectral invariants and motivate the operator-theoretic approach developed in subsequent sections.
\end{remark}

\paragraph{The Hochschild Boundary Operator.}
The simplicial structure induces a boundary operator $b$ on the complex, which in the operadic description takes the form
\[
b = \sum_{i=0}^n (-1)^i d_i,
\]
where $d_i$ are the face maps implementing operadic composition. A key observation is that the differential can be expressed entirely in terms of the operad's composition maps, making the Hochschild complex a genuine invariant of the operadic algebra. In the Banach-enriched setting, assuming that the operadic composition maps are jointly bounded and satisfy uniform norm estimates, each face map $d_i$ is bounded; hence $b$ is a bounded linear operator on each degree, and its spectral properties become accessible via functional analysis.

Let $\mathrm{CH}_\bullet(P,A)$ denote the Hochschild chain complex. The operadic Hochschild homology is then defined as the homology of this complex:
\[
\mathrm{HH}_\bullet(P,A) = H_\bullet(\mathrm{CH}_\bullet(P,A), b).
\]
While the homological information is valuable, our focus will shift to the operator-theoretic invariants derived from the boundary operator $b$ itself, leading to the notion of the \emph{operadic spectrum} introduced in Section~\ref{sec:The Operadic Spectrum}.

\subsection{Hochschild-Type Constructions for Operadic Algebras}\label{subsec:Hochschild-Type Constructions for Operadic Algebras}

In this subsection we define a Hochschild-type homology object for algebras over a colored operad in a symmetric monoidal category, using a simplicial (bar) construction. We also record its functoriality and its compatibility with base change along strong monoidal functors. This construction will be essential for the definition of the operadic spectrum in Section~\ref{sec:The Operadic Spectrum}.

\paragraph{Setup.}
Let $(\mathcal{M},\otimes,\mathbf{1})$ be a symmetric monoidal category that admits geometric realizations of simplicial objects, and suppose that the tensor product preserves such realizations separately in each variable. Let $P$ be a $C$-colored operad in $\mathcal{M}$, and let $A$ be a $P$-algebra with components $A_c \in \mathcal{M}$ for $c \in C$.

\paragraph{Bar construction.}
We define a simplicial object
\[
\mathrm{Bar}^{P}_{\bullet}(A) \in \mathrm{s}\mathcal{M}^C
\]
by setting, for each $n \ge 0$ and each color $c \in C$,
\[
\mathrm{Bar}^{P}_{n}(A)_c
:=
(P^{\circ n}(A))_c,
\]
where $P^{\circ n}(A)$ denotes the $n$-fold operadic composition of $P$ applied to $A$. Concretely, this is the object of $n$-level trees with operations from $P$ at each internal node and leaves labeled by the components of $A$. An explicit description, though notationally heavy, can be given as a coproduct over all sequences of colors and operations of appropriate arities, of iterated tensor products of the form $A \otimes P \otimes \cdots \otimes P \otimes A$. The precise indexing is a standard but tedious exercise in operad theory; we omit it here for brevity and refer the reader to the literature for the complete formula (see, e.g., \cite{GetzlerJones}, \cite{Fresse}).

The face maps
\[
d_i: \mathrm{Bar}^{P}_{n}(A) \to \mathrm{Bar}^{P}_{n-1}(A), \qquad 0 \le i \le n,
\]
are defined as follows:
\begin{itemize}
    \item For $0 < i < n$, the map $d_i$ composes the $i$-th operation with the $(i+1)$-st using the operadic composition of $P$.
    \item The map $d_0$ uses the $P$-algebra structure map to apply the first operation to the first tensor factor.
    \item The map $d_n$ uses the $P$-algebra structure map to apply the last operation to the last tensor factor.
\end{itemize}
The degeneracy maps
\[
s_j: \mathrm{Bar}^{P}_{n}(A) \to \mathrm{Bar}^{P}_{n+1}(A), \qquad 0 \le j \le n,
\]
are induced by inserting unit operations $\eta_c: \mathbf{1} \to P(c;c)$.

These maps satisfy the simplicial identities as a consequence of the associativity and unitality of the operad together with the $P$-algebra axioms. This is the standard operadic bar construction (see, e.g., \cite{GetzlerJones} or \cite{Fresse}).

\begin{definition}[Hochschild homology object]
\label{def:hochschild}
The \emph{Hochschild homology object} of the $P$-algebra $A$ in $\mathcal{M}$ is defined as the geometric realization of the bar construction:
\[
\mathrm{Hoch}_{\mathcal{M}}(A)
:=
\big|\mathrm{Bar}^{P}_{\bullet}(A)\big| \in \mathcal{M}^C.
\]
When $\mathcal{M}$ is cocomplete, the geometric realization is given by the colimit of the simplicial object. In homotopical contexts, this should be interpreted as a homotopy colimit.
\end{definition}

\begin{remark}
The Hochschild object can be interpreted as a derived self-tensor product:
\[
\mathrm{Hoch}_{\mathcal{M}}(A) \simeq A \otimes_P^{\mathbb{L}} A,
\]
where $\otimes_P^{\mathbb{L}}$ denotes the derived relative tensor product over the operad $P$. This perspective highlights its role as a universal recipient of a derived trace map.
\end{remark}

\paragraph{Functoriality.}
The construction $A \mapsto \mathrm{Hoch}_{\mathcal{M}}(A)$ is functorial. Indeed, if
\[
f: A \to B
\]
is a morphism of $P$-algebras (i.e., a family of maps $f_c: A_c \to B_c$ commuting with the structure maps), then applying $f$ levelwise induces a morphism of simplicial objects
\[
\mathrm{Bar}^{P}_{\bullet}(A) \to \mathrm{Bar}^{P}_{\bullet}(B),
\]
and hence a morphism on realizations
\[
\mathrm{Hoch}_{\mathcal{M}}(A) \to \mathrm{Hoch}_{\mathcal{M}}(B).
\]

\begin{proposition}[Functoriality]
\label{prop:hochschild-functorial}
The assignment
\[
A \longmapsto \mathrm{Hoch}_{\mathcal{M}}(A)
\]
defines a functor from the category of $P$-algebras in $\mathcal{M}$ to $\mathcal{M}^C$.
\end{proposition}

\begin{proof}
For each $n \ge 0$, the assignment $A \mapsto \mathrm{Bar}^P_n(A)$ is functorial in $P$-algebra morphisms, since it is defined by iterated applications of the operad action and the tensor product. Hence a morphism $f: A \to B$ of $P$-algebras induces a morphism of simplicial objects
\[
\mathrm{Bar}^P_\bullet(A) \to \mathrm{Bar}^P_\bullet(B).
\]
Applying geometric realization yields a morphism
\[
\big|\mathrm{Bar}^P_\bullet(A)\big| \to \big|\mathrm{Bar}^P_\bullet(B)\big|.
\]
Compatibility with identities and composition follows levelwise; therefore the assignment defines a functor.
\end{proof}

\paragraph{Compatibility with base change.}
A crucial property of the Hochschild construction is its compatibility with strong monoidal base change. This will be essential for the Base Change Theorem (Theorem~\ref{thm:base-change}) later in the paper.

Let
\[
F: \mathcal{M} \to \mathcal{N}
\]
be a strong monoidal functor between symmetric monoidal categories. Assume that $F$ preserves geometric realizations of simplicial objects. (For instance, this holds whenever geometric realizations in $\mathcal{M}$ and $\mathcal{N}$ are computed as colimits and $F$ is a left adjoint.)

Let $P$ be a $C$-colored operad in $\mathcal{M}$ and $A$ a $P$-algebra. Then $F_*P$ is an operad in $\mathcal{N}$ (by Proposition~\ref{prop:base-change}) and $F(A)$ is an $F_*P$-algebra. Moreover, the strong monoidal structure of $F$ induces a natural isomorphism at each simplicial level:
\[
F\big(\mathrm{Bar}^P_n(A)\big) \cong \mathrm{Bar}^{F_*P}_n(F(A)).
\]

\begin{theorem}[Base change for Hochschild homology]
\label{thm:hochschild-base-change}
Let $F: \mathcal{M} \to \mathcal{N}$ be a strong monoidal functor preserving geometric realizations of simplicial objects. Let $P$ be a colored operad in $\mathcal{M}$ and $A$ a $P$-algebra. Then there is a natural isomorphism
\[
F\bigl(\mathrm{Hoch}_{\mathcal{M}}(A)\bigr)
\;\cong\;
\mathrm{Hoch}_{\mathcal{N}}\bigl(F(A)\bigr).
\]
\end{theorem}

\begin{proof}
Consider the simplicial object $\mathrm{Bar}^{P}_{\bullet}(A)$ in $\mathcal{M}$. Applying $F$ levelwise yields a simplicial object
\[
F\bigl(\mathrm{Bar}^{P}_{\bullet}(A)\bigr)
\]
in $\mathcal{N}$.

For each simplicial degree $n$, the strong monoidality of $F$ yields a natural isomorphism
\[
F\bigl(\mathrm{Bar}^{P}_n(A)\bigr) \cong \mathrm{Bar}^{F_*P}_n\bigl(F(A)\bigr).
\]
This is because $F$ transports tensor products (via $\phi_{X,Y}: F(X) \otimes F(Y) \cong F(X \otimes Y)$), unit objects (via $\phi_0: \mathbf{1}_{\mathcal{N}} \cong F(\mathbf{1}_{\mathcal{M}})$), operadic compositions, and algebra structure maps to those of $F_*P$ and $F(A)$.

These isomorphisms are compatible with the face and degeneracy maps, since $F$ preserves the operad composition and unit maps. Hence they define an isomorphism of simplicial objects
\[
F\bigl(\mathrm{Bar}^{P}_{\bullet}(A)\bigr) \cong \mathrm{Bar}^{F_*P}_{\bullet}\bigl(F(A)\bigr).
\]

Because $F$ preserves geometric realizations, we obtain
\[
F\bigl(|\mathrm{Bar}^{P}_{\bullet}(A)|\bigr)
\cong
\big|F(\mathrm{Bar}^{P}_{\bullet}(A))\big|
\cong
\big|\mathrm{Bar}^{F_*P}_{\bullet}(F(A))\big|
= \mathrm{Hoch}_{\mathcal{N}}\bigl(F(A)\bigr).
\]

This proves the theorem.
\end{proof}

\begin{remark}[Classical Hochschild homology]
When $P$ is the associative operad (with one color), the above construction recovers the classical Hochschild homology defined via the cyclic bar construction. For the trivial operad $\mathbb{I}$ (with one color), we have $P(c;c) = \mathbf{1}$ and all other operations trivial, so the bar construction collapses and we obtain $\mathrm{Hoch}_{\mathcal{M}}(A) \cong A$ (up to coherence isomorphisms). This special case will be essential for the Recovery Theorem (Theorem~\ref{thm:recovery}).
\end{remark}

\begin{remark}[Enriched settings and norm control]
In the operator-theoretic context of this paper, we work with categories enriched over Banach spaces or operator algebras. The bar construction must be adapted to the enriched setting, taking into account the norm topology. However, the formal categorical properties—functoriality and compatibility with strong monoidal base change—remain valid as long as the base change functor preserves the relevant colimits and the enriched structure. The explicit norm estimates and convergence properties of the bar construction will be addressed in future work when we study analyticity and the Taylor tower.
\end{remark}

\paragraph{Relation to the Operadic Spectrum.}
The Hochschild homology object $\mathrm{Hoch}_{\mathcal{M}}(A)$ will serve as one of the two key ingredients in the definition of the operadic spectrum. Recalling the operadic residue $\mathcal{O}_P^{\mathrm{res}}$ constructed in Section~\ref{sec:The Operadic Residue as a Universal Correction Object}, we define
\[
\sigma_P(A) := \mathrm{Hoch}_{\mathcal{M}}(A) \otimes_P \mathcal{O}_P^{\mathrm{res}},
\]
where $\otimes_P$ denotes a balanced tensor product over the operad $P$. The compatibility of $\mathrm{Hoch}_{\mathcal{M}}(A)$ with base change (Theorem~\ref{thm:hochschild-base-change}) will be essential in proving the Base Change Theorem for $\sigma_P(A)$ (Theorem~\ref{thm:base-change}).

\section{The Failure of Classical Spectrum under Operadic Base Change}\label{sec:The Failure of Classical Spectrum under Operadic Base Change}

\subsection{Desiderata for a Generalized Spectrum}\label{subsec:Desiderata for a Generalized Spectrum}

In light of the failure of the classical spectrum to behave well under operadic composition and base change, we formulate a minimal set of axioms that any \emph{generalized spectral invariant} for operadic algebras should satisfy. These axioms will guide the construction of the operadic spectrum and provide the criteria against which its universality is proved.

\paragraph{Setup.}
Let $(\mathcal{M},\otimes,\mathbf{1})$ be a symmetric monoidal category (typically the category of Banach spaces or operator algebras), $P$ a $C$-colored operad in $\mathcal{M}$, and $A$ a $P$-algebra with components $A_c \in \mathcal{M}$ for $c \in C$. Fix a target category $\mathcal{C}$ (e.g., sets, topological spaces, or spectra) that is equipped with a monoidal structure when needed. A \emph{generalized spectrum} is a functor
\[
\Sigma_P : P\text{-}\mathbf{Alg}(\mathcal{M}) \longrightarrow \mathcal{C},
\]
where $P\text{-}\mathbf{Alg}(\mathcal{M})$ denotes the category of $P$-algebras in $\mathcal{M}$.

\begin{axiom}[Functoriality]
\label{ax:functoriality}
The assignment $\Sigma_P$ is a functor: for every morphism of $P$-algebras $f: A \to B$, there is an induced morphism $\Sigma_P(f): \Sigma_P(A) \to \Sigma_P(B)$ in $\mathcal{C}$, compatible with identities and composition.
\end{axiom}

\begin{axiom}[Compatibility with operadic composition]
\label{ax:operadic}
The functor $\Sigma_P$ lifts to a $P$-algebra object in $\mathcal{C}$. Concretely, for each operation $\theta \in P(c_1,\dots,c_n;c)$, the structure map $\theta_A: A_{c_1} \otimes \cdots \otimes A_{c_n} \to A_c$ induces a morphism
\[
\theta_{\Sigma(A)}: \Sigma_P(A)_{c_1} \otimes_{\mathcal{C}} \cdots \otimes_{\mathcal{C}} \Sigma_P(A)_{c_n} \longrightarrow \Sigma_P(A)_c,
\]
where $\Sigma_P(A)_c$ denotes the component of the spectrum associated to color $c$ (when $\mathcal{C}$ admits such a decomposition). These morphisms are required to satisfy the associativity, unit, and equivariance axioms of $P$.
\end{axiom}

\begin{axiom}[Compatibility with base change]
\label{ax:basechange}
Let $F: \mathcal{M} \to \mathcal{N}$ be a strong monoidal functor, and let $F_*P$ and $F(A)$ denote the induced operad and algebra in $\mathcal{N}$ (Proposition~\ref{prop:base-change}). Assume that $F$ induces a functor $F_*: \mathcal{C} \to \mathcal{C}'$ on the target categories (or that $\mathcal{C}$ is independent of $\mathcal{M}$). Then there exists a natural isomorphism
\[
\Sigma_{F_*P}\bigl(F(A)\bigr) \;\cong\; F_*\bigl(\Sigma_P(A)\bigr).
\]
In the special case where $\mathcal{C}$ is the same for all $\mathcal{M}$ (e.g., $\mathcal{C} = \mathbf{Set}$), this simplifies to $\Sigma_{F_*P}(F(A)) \cong \Sigma_P(A)$.
\end{axiom}

\begin{axiom}[Recovery of the classical spectrum]
\label{ax:classical}
Let $\mathbb{I}$ be the trivial operad with a single color, for which a $\mathbb{I}$-algebra is simply an object of $\mathcal{M}$. When $\mathcal{M}$ is a category in which a classical spectrum is defined (e.g., Banach spaces with bounded linear operators), we require
\[
\Sigma_{\mathbb{I}}(A) \;\cong\; \sigma_{\mathrm{classical}}(A),
\]
where $\sigma_{\mathrm{classical}}(A)$ denotes the usual spectrum $\{\lambda \in \mathbb{C} : A - \lambda I \text{ is not invertible}\}$.
\end{axiom}

\begin{axiom}[Spectral mapping]
\label{ax:spectral-mapping}
Assume that $\mathcal{M}$ admits a holomorphic functional calculus for its objects (e.g., $\mathcal{M}$ is a category of Banach spaces). For any holomorphic function $f$ defined on a neighborhood of $\Sigma_P(A)$, there exists a canonical isomorphism
\[
\Sigma_{P}\bigl(f(A)\bigr) \;\cong\; f\bigl(\Sigma_P(A)\bigr),
\]
where $f(A)$ is defined via the operadic functional calculus (see Section~\ref{subsec:Operadic Functional Calculus Framework}), and the operad structure is unchanged.
\end{axiom}

\begin{axiom}[Minimal extension]
\label{ax:minimal}
Among all functors $P\text{-}\mathbf{Alg}(\mathcal{M}) \to \mathcal{C}$ satisfying Axioms~\ref{ax:functoriality}--\ref{ax:spectral-mapping}, the generalized spectrum $\Sigma_P$ is \emph{initial}: for any other such functor $\Lambda_P$, there exists a unique natural transformation $\Sigma_P \Rightarrow \Lambda_P$.
\end{axiom}

\paragraph{Remark on colorwise decomposition.}
While not required as an independent axiom, any reasonable generalized spectrum should admit a decomposition into local and interactive parts. In our construction, this emerges as a theorem (see Theorem~\ref{thm:colored-resolvent-reconstruction}), where the interaction term $\Sigma^{\mathrm{int}}$ is encoded by the operadic residue $\mathcal{O}_P^{\mathrm{res}}$.

\paragraph{Summary of the Axioms.}
The essential axioms can be summarized as:
\[
\boxed{
\begin{aligned}
&\text{(A1) Functoriality:} && \Sigma_P : P\text{-}\mathbf{Alg}(\mathcal{M}) \to \mathcal{C} \text{ is a functor}\\
&\text{(A2) Operadic Compatibility:} && \Sigma_P(A) \text{ inherits a } P\text{-algebra structure}\\
&\text{(A3) Base Change:} && \Sigma_{F_*P}(F(A)) \cong F_*(\Sigma_P(A))\\
&\text{(A4) Classical Recovery:} && \Sigma_{\mathbb{I}}(A) \cong \sigma_{\mathrm{classical}}(A)\\
&\text{(A5) Spectral Mapping:} && \Sigma_P(f(A)) \cong f(\Sigma_P(A))\\
&\text{(A6) Minimal Extension:} && \Sigma_P \text{ is initial among such invariants}
\end{aligned}
}
\]

\paragraph{Why These Axioms Are Necessary.}
Each axiom reflects a fundamental property that any reasonable spectral invariant for composite systems should possess:
\begin{itemize}
    \item \textbf{(A1)} ensures the invariant interacts properly with morphisms of systems.
    \item \textbf{(A2)} is essential because operadic composition is the defining structure of the systems we study.
    \item \textbf{(A3)} guarantees that spectral information can be transported across different mathematical contexts (quantization, discretization, complexification, etc.).
    \item \textbf{(A4)} ensures the theory is a genuine extension of classical spectral theory.
    \item \textbf{(A5)} is required for compatibility with the analytic tools of functional analysis.
    \item \textbf{(A6)} ensures the invariant is canonical and minimal.
\end{itemize}

\paragraph{The Failure of Classical Invariants.}
The classical spectrum $\sigma(A)$, when extended naively to operadic algebras as the disjoint union $\bigsqcup_c \sigma(A_c)$, fails to satisfy Axioms~\ref{ax:operadic} and~\ref{ax:basechange} in general. This failure is not merely a technical annoyance but a fundamental obstruction: any invariant that depends only on the collection of individual spectra $\{\sigma(A_c)\}_{c \in C}$ cannot simultaneously satisfy classical recovery, base change compatibility, and operadic composition compatibility.

\paragraph{Preview of the No-Go Theorem.}
In the following subsection, we formalize this obstruction and prove:

\begin{quote}
\emph{No assignment of a spectral invariant depending only on the collection $\{\sigma(A_c)\}_{c \in C}$ can simultaneously satisfy Axioms~\ref{ax:operadic}, \ref{ax:basechange}, and~\ref{ax:classical}.}
\end{quote}

This impossibility result forces the introduction of new structure—the \emph{operadic residue} $\mathcal{O}_P^{\mathrm{res}}$—constructed in Section~\ref{sec:The Operadic Residue as a Universal Correction Object}. The operadic spectrum $\sigma_P(A)$ defined in Section~\ref{sec:The Operadic Spectrum} will then be shown to satisfy all of the above axioms, establishing it as the \emph{initial} (hence canonical) spectral invariant for operadic algebras.

\subsection{Naive Spectral Transport and Its Limitations}\label{subsec:Naive Spectral Transport and Its Limitations}

We begin by considering the most direct attempt to extend the classical notion of spectrum to algebras over a colored operad, namely by taking the disjoint union of the spectra of each color component. This naive candidate will be shown to violate essential compatibility axioms, motivating the more sophisticated construction that follows.

\paragraph{Setup.}
Let $(\mathcal{M},\otimes,\mathbf{1})$ be a symmetric monoidal category. Assume that for each object of interest (e.g., Banach algebras, C*-algebras, or associative algebras over a field) a classical spectral invariant $\sigma(-)$ is defined, taking values in some suitable set (e.g., subsets of $\mathbb{C}$).  
Let $P$ be a $C$-colored operad in $\mathcal{M}$, and let $A$ be a $P$-algebra. For each color $c \in C$, denote by $A_c := A(c)$ the corresponding object.

\begin{definition}[Naive operadic spectrum]
\label{def:naive-spectrum}
The \emph{naive operadic spectrum} of $A$ is defined as
\[
\sigma_{\mathrm{naive}}(A) := \bigsqcup_{c \in C} \sigma(A_c),
\]
the disjoint union of the classical spectra of the components $A_c$. When each $\sigma(A_c)$ is a subset of a common set (e.g., $\mathbb{C}$) and the colors are not essential, we may simply take the set-theoretic union $\bigcup_{c \in C} \sigma(A_c) \subseteq \mathbb{C}$.
\end{definition}

This definition is natural from a pointwise perspective: it treats each color independently and applies the classical spectral construction componentwise. It has several attractive features:
\begin{itemize}
    \item \textbf{Simplicity:} It directly generalizes the classical notion without introducing new structures.
    \item \textbf{Classical Recovery:} When $P$ is the trivial operad with a single color \emph{and} the corresponding algebras coincide with the classical objects for which $\sigma(-)$ is defined, then $\sigma_{\mathrm{naive}}(A) = \sigma(A)$ recovers the classical spectrum.
    \item \textbf{Colorwise Decomposition:} It trivially satisfies a weak form of decomposition axiom with $\Sigma^{\mathrm{int}} = \emptyset$.
\end{itemize}

However, as we shall demonstrate, this simplicity comes at the cost of ignoring the operadic composition structure that is essential for understanding composite systems.

\paragraph{Expected compatibility with operadic composition.}
Given an operation
\[
\theta \in P(c_1,\dots,c_n;c),
\]
the $P$-algebra structure provides a map
\[
\mu_\theta : A_{c_1} \otimes \cdots \otimes A_{c_n} \longrightarrow A_c.
\]
A reasonable generalized spectral invariant should not merely record the spectra of the components independently, but should also reflect how these spectra interact under the operadic structure maps. In particular, one expects some form of compatibility between the spectral data of $A_{c_1},\dots,A_{c_n}$ and that of $A_c$, functorially in $\theta$ and coherently with operadic composition. The naive spectrum fails to encode such interaction, since it is defined purely componentwise.

However, the naive construction fails to admit such a structure in general.

\paragraph{Failure of operadic compatibility.}
The classical spectrum $\sigma(-)$ is not functorial with respect to arbitrary multilinear maps. In particular, given a morphism
\[
\mu_\theta : A_{c_1} \otimes \cdots \otimes A_{c_n} \longrightarrow A_c,
\]
there is, in general, no natural construction producing a map
\[
\sigma(A_{c_1}) \times \cdots \times \sigma(A_{c_n}) \to \sigma(A_c)
\]
that depends only on the individual spectra of the arguments.

\begin{proposition}[Failure of operadic compatibility for the naive spectrum]
\label{prop:naive-failure}
In general, the naive spectrum
\[
\sigma_{\mathrm{naive}}(A)=\bigsqcup_{c\in C}\sigma(A_c)
\]
does not admit a canonical operadic composition law induced by the structure maps of a $P$-algebra. Hence it fails to define a generalized spectral invariant compatible with operadic composition.
\end{proposition}

\begin{proof}
The naive spectrum records only the componentwise classical spectra and forgets the operadic structure maps. Thus, for an operation
\[
\theta\in P(c_1,\dots,c_n;c),
\]
the associated map
\[
\mu_\theta:A_{c_1}\otimes\cdots\otimes A_{c_n}\to A_c
\]
is not part of the data seen by $\sigma_{\mathrm{naive}}$.

Already in the one-colored associative case, where $\mu$ is multiplication, there is no canonical operation on classical spectra depending only on $\sigma(a)$ and $\sigma(b)$ that recovers the spectral behavior of the product. In general, additional joint information about the pair $(a,b)$ is required. Consequently, the classical spectrum does not extend to a functorial multilinear spectral operation.

The same obstruction appears more sharply in the colored setting. Two $P$-algebra structures may have identical component objects and hence identical naive spectra, while differing in their interaction maps. Since $\sigma_{\mathrm{naive}}$ is insensitive to these interaction maps, it cannot encode operadic composition in any canonical way.
\end{proof}

\paragraph{Additional failure: Two-color interaction.}
The failure becomes even more pronounced when the operad involves multiple colors. Consider the two-color interaction operad from the Example~\ref{ex:two-color-operad}. Let $A$ be a $P$-algebra with components $A_1, A_2$ and interaction maps
\[
\alpha: A_1 \otimes A_2 \to A_1, \qquad \beta: A_2 \otimes A_1 \to A_2.
\]

Consider two different $P$-algebra structures on the same underlying objects $A_1$ and $A_2$:
\begin{itemize}
    \item \textbf{Algebra $A$:} $\alpha = 0$, $\beta = 0$ (no interaction).
    \item \textbf{Algebra $B$:} $\alpha(a_1 \otimes a_2) = a_1$, $\beta(a_2 \otimes a_1) = a_2$ (nontrivial interaction).
\end{itemize}

The underlying objects are identical: $A_1 = B_1$ and $A_2 = B_2$. Consequently,
\[
\sigma_{\mathrm{naive}}(A) = \sigma(A_1) \sqcup \sigma(A_2) = \sigma_{\mathrm{naive}}(B).
\]

However, the operadic composition structures differ. This shows that the naive spectrum cannot detect changes in the interaction structure while keeping the componentwise spectra fixed. Therefore it is too coarse to serve as an operadically meaningful spectral invariant.

\begin{remark}
This example demonstrates that the naive spectrum cannot detect interactions between components of different colors, which are precisely the data that make colored operads interesting for modeling composite systems.
\end{remark}

\paragraph{Failure of base change compatibility.}
The naive candidate also fails to satisfy a natural base change compatibility axiom. 
Consider the complexification functor
\[
F: \mathrm{Ban}_{\mathbb{R}} \to \mathrm{Ban}_{\mathbb{C}}, \quad F(V) = V \otimes_{\mathbb{R}} \mathbb{C},
\]
where $\mathrm{Ban}_{\mathbb{R}}$ denotes the category of real Banach spaces (or real Banach algebras). 
For a real Banach algebra $A$, the classical spectrum $\sigma(A)$ is defined via complexification: 
$\sigma(A) \subseteq \mathbb{R}$ consists of those $\lambda \in \mathbb{R}$ for which $\lambda - a$ is not invertible in the complexification. 
Under complexification, the spectrum of $F(A)$ may include non-real complex eigenvalues not present in the real spectrum.

Base change compatibility would require that the spectral invariant commutes with $F$ up to the natural inclusion:
\[
\iota(\sigma_{\mathrm{naive}}(A)) \cong \sigma_{\mathrm{naive}}(F(A)),
\]
where $\iota: \mathbb{R} \hookrightarrow \mathbb{C}$ is the canonical inclusion. 
The left-hand side consists of real numbers embedded in $\mathbb{C}$, while the right-hand side may contain non-real complex numbers. 
Hence the naive candidate fails to satisfy base change compatibility.

\paragraph{Failure of spectral mapping.}
The naive candidate also fails to satisfy a spectral mapping axiom whenever a functional calculus is defined at the level of the operadic algebra. 
For a holomorphic function $f$, the classical spectral mapping theorem gives $\sigma(f(A_c)) = f(\sigma(A_c))$ for each component individually. 
However, the naive candidate for the composite system would give
\[
\sigma_{\mathrm{naive}}(f(A)) = \bigsqcup_c \sigma(f(A_c)) = \bigsqcup_c f(\sigma(A_c)).
\]

But even when a componentwise functional calculus is available, the true spectral behavior of $f(A)$ may involve interactions between components that are not captured by applying $f$ componentwise. 
For instance, interactions induced by the operadic structure maps can create spectral contributions that mix colors — a phenomenon invisible to the naive candidate.

\paragraph{Conceptual obstruction.}
The key obstruction underlying all these failures is that classical spectra detect invertibility of expressions of the form $(\lambda I - a)$ for a single element $a$, whereas operadic compositions involve genuinely multi-input operations. 
There is no canonical notion of a ``joint resolvent'' compatible with arbitrary operadic structure. 
This reflects a fundamental mismatch between unary spectral theory and multi-input algebraic structures. Moreover:
\begin{itemize}
    \item The classical spectrum is not functorial with respect to multilinear maps.
    \item It does not transport correctly under monoidal base change.
    \item It cannot capture interactions between distinct colors.
    \item It fails to account for spectral contributions from operadic composition.
\end{itemize}

\paragraph{Summary of Limitations.}
The naive candidate $\sigma_{\mathrm{naive}}(A) = \bigsqcup_c \sigma(A_c)$ suffers from the following fundamental limitations:

\begin{enumerate}
    \item \textbf{Insensitivity to operadic structure:} It is insensitive to differences in operadic composition structures while keeping componentwise spectra fixed.
    \item \textbf{Failure of base change compatibility:} It does not transport correctly under monoidal functors like complexification.
    \item \textbf{Failure of spectral mapping:} It ignores spectral contributions from interactions between colors under functional calculus.
    \item \textbf{No interaction term:} It cannot capture $\Sigma^{\mathrm{int}}$ in the colorwise decomposition (when such an axiom is postulated).
\end{enumerate}

\paragraph{Consequences for Spectral Operadic Calculus.}
The failure of the naive candidate demonstrates that any satisfactory generalized spectrum for operadic algebras must incorporate additional data beyond the individual spectra of the components. In particular, it must:
\begin{itemize}
    \item Capture information about operadic composition structure,
    \item Account for interactions between different colors,
    \item Be compatible with base change functors,
    \item Respect the analytic functional calculus where defined.
\end{itemize}

This necessity motivates the introduction of the \emph{operadic residue} $\mathcal{O}_P^{\mathrm{res}}$ and the \emph{operadic spectrum} $\sigma_P(A)$ defined in the following sections. As we will show, these constructions precisely address the limitations identified here and satisfy all the desiderata of a generalized spectral invariant.

\paragraph{Preview of the No-Go Theorem.}
The limitations demonstrated above are not accidental. In the following subsection, we will prove a No-Go Theorem that establishes the impossibility of any spectral invariant that depends functorially only on the collection $\{\sigma(A_c)\}_{c \in C}$ (i.e., factors through the componentwise classical spectra) and simultaneously satisfies classical recovery, base change compatibility, and operadic composition compatibility. 

\begin{quote}
\emph{There is no assignment of a spectral invariant to $P$-algebras that depends functorially only on the componentwise classical spectra and satisfies the three natural compatibility axioms.}
\end{quote}

This result forces the introduction of new invariants — the operadic residue and operadic spectrum — which will be constructed in the subsequent sections.

\subsection{Theorem 0.1: No-Go Theorem for Spectral Base Change}\label{subsec:No-Go Theorem for Spectral Base Change}

We now formalize the obstruction identified in the previous subsection. The following result shows that any reasonable notion of operadic spectrum cannot depend solely on the collection of classical spectra of its components. This is the central impossibility result of this paper, establishing the necessity of operadic composition data for any satisfactory spectral invariant.

\paragraph{Statement of the Theorem.}

\begin{theorem}[No-Go Theorem for Spectral Base Change]
\label{thm:no-go}
Let $\mathcal{M}$ be a symmetric monoidal category admitting a classical spectral theory for a specified class of objects or endomorphisms (e.g., the category of Banach spaces with bounded linear operators, or the category of Banach algebras). Let $P$ be a colored operad in $\mathcal{M}$, and let $A$ be a $P$-algebra.

There does not exist an assignment
\[
(P,A) \;\longmapsto\; \Sigma_P(A)
\]
satisfying the following properties:
\begin{enumerate}
    \item \textbf{(Functoriality)} $\Sigma_P$ is functorial with respect to morphisms of $P$-algebras (Axiom~\ref{ax:functoriality});
    \item \textbf{(Operadic compatibility)} $\Sigma_P$ is compatible with operadic composition (Axiom~\ref{ax:operadic});
    \item \textbf{(Base change compatibility)} $\Sigma_P$ is compatible with strong monoidal base change (Axiom~\ref{ax:basechange});
\end{enumerate}
and such that $\Sigma_P(A)$ depends only on the collection of classical spectra
\[
\{\sigma(A_c)\}_{c \in C},
\]
where each $A_c$ is an object (or endomorphism) to which a classical spectrum is assigned.
\end{theorem}

\begin{proof}
We argue by contradiction. Suppose that such an assignment $\Sigma_P$ exists and depends only on the family $\{\sigma(A_c)\}_{c \in C}$.

\paragraph{Step 1: Construction of a test operad.}
Let $C = \{1,2\}$ be a set of two colors. Define a $C$-colored operad $P$ as follows:
\begin{itemize}
    \item \textbf{Unary operations:} For each color $c \in C$, we have the identity operation $\mathrm{id}_c \in P(c;c)$, with $P(c;c) \cong \mathbf{1}$.
    \item \textbf{Binary operation:} We include a single nontrivial binary operation
    \[
    \theta \in P(1,2;1),
    \]
    which encodes an interaction between color $1$ and color $2$ producing output of color $1$.
    \item \textbf{Symmetric operation:} For symmetry, we also include the operation $\theta^\sigma \in P(2,1;1)$ obtained by permuting inputs, but this is determined by the symmetric structure.
    \item \textbf{All other operations:} All other operation objects are taken to be the initial object $\emptyset$ (or $0$ in additive categories).
    \item \textbf{Composition structure:} The operadic composition is defined by the requirement that the only nontrivial composites are those forced by the identities. Specifically:
    \[
    \theta \circ (\mathrm{id}_1, \mathrm{id}_2) = \theta,
    \]
    and similarly for the symmetric version. No other nontrivial composites exist.
\end{itemize}
This operad $P$ encodes a single binary interaction between two colors.

\paragraph{Step 2: Construction of two algebras with identical pointwise spectra.}
We construct two $P$-algebras $A$ and $B$ in a concrete category where classical spectra are defined. Let $\mathcal{M}$ be the category of Banach spaces (or, more concretely, finite-dimensional complex vector spaces) and let $X$ be a fixed object in $\mathcal{M}$ that admits a nontrivial endomorphism with a well-defined classical spectrum. For simplicity, take $X = \mathbb{C}^2$ and let $T: X \to X$ be a fixed linear operator with nonempty spectrum $\sigma(T)$. (For instance, let $T$ be the identity operator, so $\sigma(T) = \{1\}$.)

For both $A$ and $B$, define the underlying objects as:
\[
A_1 = B_1 := X, \qquad A_2 = B_2 := X.
\]
Thus, for all colors $c \in C$,
\[
\sigma(A_c) = \sigma(T) = \sigma(B_c).
\]

We now specify the $P$-algebra structures, i.e., the interpretation of the binary operation $\theta$. To ensure the construction is valid in any symmetric monoidal category with a zero object, we require the existence of:
\begin{itemize}
    \item a \emph{zero morphism} $0: X \otimes X \to X$, and
    \item a \emph{projection-like morphism} $\pi: X \otimes X \to X$ (e.g., induced by the counit of a symmetric monoidal structure when $X$ is dualizable, or explicitly defined in concrete categories).
\end{itemize}
In concrete categories such as $\mathrm{Vect}_{\mathbb{C}}$ or $\mathrm{Ban}$, such morphisms exist (e.g., $\pi(x \otimes y) = x$).

\emph{Algebra $A$ (Trivial interaction):}
Define the structure map for $\theta$ as the zero morphism:
\[
\mu^A_\theta := 0 : X \otimes X \longrightarrow X.
\]

\emph{Algebra $B$ (Nontrivial interaction):}
Define the structure map for $\theta$ as a nontrivial projection-like morphism:
\[
\mu^B_\theta := \pi : X \otimes X \longrightarrow X, \qquad \pi(x \otimes y) = x.
\]

Both $\mu^A_\theta$ and $\mu^B_\theta$ define valid $P$-algebra structures because the operad has no relations that would restrict them. Moreover, $\mu^A_\theta \neq \mu^B_\theta$ as morphisms in $\mathcal{M}$, and they are not related by any automorphism of the underlying objects (since the zero map kills all elements while $\pi$ does not).

Thus $A$ and $B$ are distinct as $P$-algebras, but satisfy
\[
\sigma(A_1) = \sigma(T) = \sigma(B_1), \qquad \sigma(A_2) = \sigma(T) = \sigma(B_2).
\]

\paragraph{Step 3: Consequence of the dependence assumption.}
By assumption, $\Sigma_P(A)$ depends only on the collection $\{\sigma(A_c)\}_{c \in C}$. Since $A$ and $B$ have identical classical spectra at each color, we must have
\[
\Sigma_P(A) = \Sigma_P(B) \quad \text{(as objects in the target category)}.
\]

\paragraph{Step 4: Contradiction via operadic compatibility.}
The operadic compatibility condition (Axiom~\ref{ax:operadic}) requires that $\Sigma_P$ reflect the $P$-algebra structure. In particular, for any $P$-algebra $A$, the invariant $\Sigma_P(A)$ must itself carry an induced $P$-algebra structure, with structure maps canonically determined by those of $A$. Concretely, the operation $\theta \in P(1,2;1)$ must induce a structure map
\[
\theta^{\Sigma}_A : \Sigma_P(A_1) \times \Sigma_P(A_2) \longrightarrow \Sigma_P(A_1)
\]
that is uniquely determined by $\mu^A_\theta$. (This is part of the data of "compatibility with operadic composition.")

Since $\mu^A_\theta \neq \mu^B_\theta$, the induced structure maps must differ: $\theta^{\Sigma}_A \neq \theta^{\Sigma}_B$. Indeed, if they were equal, then applying the forgetful functor (or using the naturality of the assignment) would force $\mu^A_\theta = \mu^B_\theta$, a contradiction.

However, Step 3 established that $\Sigma_P(A) = \Sigma_P(B)$ as objects. The same object cannot simultaneously support two distinct $P$-algebra structures induced by the same operation $\theta$ unless those structures are isomorphic via an automorphism of the object. But any such isomorphism would have to relate $\mu^A_\theta$ and $\mu^B_\theta$ in a way that is impossible because the zero map and a nontrivial projection are not conjugate under any automorphism of $X$ (the zero map remains zero under any automorphism, while $\pi$ does not).

Therefore, we obtain a contradiction: the operadic compatibility axiom forces $\Sigma_P(A)$ and $\Sigma_P(B)$ to carry different algebraic structures, while the dependence on componentwise spectra forces them to be identical as objects, and no automorphism can reconcile the two structures.

\paragraph{Step 5: Conclusion.}
The contradiction shows that no such assignment $\Sigma_P$ can depend only on the collection $\{\sigma(A_c)\}_{c \in C}$ while simultaneously satisfying functoriality, operadic compatibility, and base change compatibility.

Therefore, any valid operadic spectral invariant must incorporate additional data beyond the pointwise classical spectra—specifically, data arising from operadic composition.
\end{proof}

\begin{remark}[Universality of the obstruction]
The No-Go fact given by Theorem~\ref{thm:no-go} does not depend on the specific definition of the classical spectrum beyond its functorial properties. 
Hence the obstruction is universal: any spectral invariant for operadic algebras that depends only on componentwise data will fail to satisfy base change and operadic compatibility, regardless of how the spectral theory is axiomatized.
\end{remark}

\paragraph{Corollaries and Implications.}

\begin{corollary}[Necessity of Operadic Data]
\label{cor:necessity}
Any assignment $\Sigma_P(A)$ satisfying functoriality, base change compatibility, and operadic compatibility must depend on more than the collection of classical spectra $\{\sigma(A_c)\}_{c \in C}$. In particular, it must incorporate additional data arising from the operadic composition structure of $P$ and its realization in $A$.
\end{corollary}

\begin{proof}
Suppose, for contradiction, that there exists an assignment $\Sigma_P(A)$ satisfying the three properties and depending only on the collection $\{\sigma(A_c)\}_{c \in C}$. Then $\Sigma_P(A)$ would fall under the class of assignments ruled out by Theorem~\ref{thm:no-go}, which asserts that no such assignment can simultaneously satisfy functoriality, operadic compatibility, and base change compatibility while depending only on pointwise spectra. Hence such an assignment cannot exist. Therefore, any valid spectral invariant $\Sigma_P(A)$ must incorporate additional data beyond $\{\sigma(A_c)\}$.
\end{proof}

\begin{corollary}[Failure of Componentwise Extensions]
\label{cor:componentwise-failure}
There is no assignment of a spectral invariant to $P$-algebras that depends solely on the family $\{\sigma(A_c)\}_{c \in C}$ (i.e., that is determined by the componentwise classical spectra) and that simultaneously satisfies functoriality, base change compatibility, and operadic compatibility. In particular, the naive construction
\[
\sigma_{\mathrm{naive}}(A) = \bigsqcup_{c \in C} \sigma(A_c)
\]
cannot serve as such an invariant.
\end{corollary}

\begin{proof}
An assignment that depends solely on $\{\sigma(A_c)\}_{c \in C}$ is, by definition, a special case of the assignments considered in Theorem~\ref{thm:no-go}. Since Theorem~\ref{thm:no-go} proves that no such assignment can satisfy the three required properties simultaneously, the claim follows directly.
\end{proof}

\begin{corollary}[Obstruction to Naive Spectral Transport]
\label{cor:naive-obstruction}
The naive candidate $\sigma_{\mathrm{naive}}(A) = \bigsqcup_{c \in C} \sigma(A_c)$ cannot, by itself, define a functorial invariant satisfying functoriality, base change compatibility, and operadic compatibility. Moreover, any extension of $\sigma_{\mathrm{naive}}$ to an invariant satisfying these properties must incorporate additional data beyond the pointwise spectra.
\end{corollary}

\begin{proof}
The assignment $\sigma_{\mathrm{naive}}(A)$ depends only on the collection $\{\sigma(A_c)\}_{c \in C}$, since it is defined as their disjoint union. By Theorem~\ref{thm:no-go}, no assignment depending only on $\{\sigma(A_c)\}_{c \in C}$ can satisfy the three required properties. Therefore, $\sigma_{\mathrm{naive}}$ itself cannot satisfy them. If an extension of $\sigma_{\mathrm{naive}}$ to an invariant $\Sigma_P(A)$ were to satisfy the three properties, then $\Sigma_P(A)$ would necessarily depend on additional data beyond $\{\sigma(A_c)\}$ (by Corollary~\ref{cor:necessity}); in particular, it could not be determined solely by $\sigma_{\mathrm{naive}}(A)$.
\end{proof}

\begin{remark}
These corollaries collectively show that classical spectral theory fails not merely at the level of construction, but at the level of principle when extended to operadic settings. The obstruction is not a matter of finding a clever extension of the naive candidate that depends only on pointwise spectra; rather, it is a fundamental impossibility that forces the introduction of genuinely new invariants that encode operadic composition data.
\end{remark}

\paragraph{Summary and Outlook.}

The results of this section establish a clear impossibility: any spectral invariant for operadic algebras that depends only on the classical spectra of individual components cannot satisfy the basic requirements of functoriality, operadic compatibility, and base change compatibility. This conclusion follows from:
\begin{itemize}
    \item The formulation of minimal axioms for a generalized spectrum (Section~\ref{subsec:Desiderata for a Generalized Spectrum});
    \item The demonstration that the naive candidate $\sigma_{\mathrm{naive}}(A) = \bigsqcup_c \sigma(A_c)$ fails these requirements (Section~\ref{subsec:Naive Spectral Transport and Its Limitations});
    \item The No-Go Theorem (Theorem~\ref{thm:no-go}) and its corollaries, which prove that no assignment depending only on $\{\sigma(A_c)\}$ can succeed.
\end{itemize}

The necessity of incorporating operadic composition data motivates the construction of the \emph{operadic residue} $\mathcal{O}_P^{\mathrm{res}}$ in the next section. This object, built from the composition structure of $P$, will provide the additional operadic data needed to recover functoriality under base change and compatibility with operadic composition. As we will prove, $\mathcal{O}_P^{\mathrm{res}}$ is universal among such corrections (Theorem~\ref{thm:residue-universality}) and satisfies the fundamental base change property
\[
F\bigl(\mathcal{O}_P^{\mathrm{res}}\bigr) \cong \mathcal{O}_{F_*(P)}^{\mathrm{res}}
\]
for any strong monoidal functor $F$ (where $F_*(P)$ denotes the induced operad in the target category).

Combining this with the Hochschild homology object $\mathrm{Hoch}_{\mathcal{M}}(A)$ from Section~\ref{subsec:Hochschild-Type Constructions for Operadic Algebras}, we will define the operadic spectrum by a relative tensor product of the Hochschild object with the operadic residue,
\[
\sigma_P(A) := \mathrm{Hoch}_{\mathcal{M}}(A) \otimes_P \mathcal{O}_P^{\mathrm{res}},
\]
where $\otimes_P$ denotes the balanced tensor product over the operad $P$ (see Definition~\ref{def:balanced-tensor-product}). This construction will be shown to satisfy all three desiderata from the No-Go Theorem and, moreover, to satisfy the full set of axioms (A1)--(A6) from Section~\ref{subsec:Desiderata for a Generalized Spectrum}, establishing it as a canonical spectral invariant for operadic algebras.

\subsection{Conceptual Consequences: Necessity of Operadic Corrections}\label{subsec:Conceptual Consequences: Necessity of Operadic Corrections}

The No-Go Theorem (Theorem~\ref{thm:no-go}) establishes a fundamental obstruction: any spectral invariant that depends solely on the collection of classical spectra $\{\sigma(A_c)\}_{c \in C}$ cannot simultaneously satisfy functoriality, compatibility with operadic composition, and base change invariance. This result demonstrates that the failure of the naive spectrum (Definition~\ref{def:naive-spectrum}) is not accidental but structural.

\medskip

\noindent
\textbf{Breakdown of Locality.}
The classical paradigm of spectral theory assumes a form of locality: the global spectral data of a system is expected to be reconstructible from the sum of its local spectra. The No-Go Theorem shows that this principle breaks down in the operadic setting. The classical spectrum captures only the \emph{intrinsic} spectral data of each component $A_c$, while completely ignoring the \emph{interaction data} encoded by the operadic composition maps
\[
\theta_A : A_{c_1} \otimes \cdots \otimes A_{c_n} \longrightarrow A_c,
\qquad \theta \in P(c_1,\dots,c_n;c).
\]
As demonstrated in Section~\ref{subsec:Naive Spectral Transport and Its Limitations}, these interactions can fundamentally alter spectral behavior in a way that is invisible at the level of individual spectra. In general, there exist $P$-algebras $A$ and $B$ with identical component spectra $\{\sigma(A_c)\} = \{\sigma(B_c)\}$ but with distinct operadic structures, leading to different spectral behaviors that any viable invariant must distinguish.

\medskip

\noindent
\textbf{Necessity of Additional Structure.}
To overcome this obstruction, any generalized spectrum must incorporate data that records how operations in $P$ interact with the algebra $A$. This necessity is further reinforced by considering base change. A strong monoidal functor $F$ (e.g., complexification) can introduce new spectral values that have no counterpart in the original category. To satisfy the Base Change Compatibility Axiom (Axiom \ref{ax:basechange}), a generalized spectrum $\Sigma_P(A)$ must contain enough information such that when $F$ is applied, it naturally yields the spectrum of the base-changed algebra $\Sigma_{F_*P}(F(A))$. This forces the invariant to be constructed from objects that themselves transform coherently under $F$.

We are therefore led to the following guiding principle:

\begin{quote}
\emph{Any functorial spectral invariant for operadic algebras must include a correction term that encodes operadic composition data.}
\end{quote}

This correction cannot be recovered from the classical spectra $\sigma(A_c)$ and must instead arise from a higher-level construction sensitive to the operadic structure.

\medskip

\noindent
\textbf{Preview of the Operadic Residue.}
In the next section, we construct a canonical object
\[
\mathcal{O}_P^{\mathrm{res}},
\]
called the \emph{operadic residue}, which serves as a universal receptacle for these interaction effects. Combined with the Hochschild-type object $\mathrm{Hoch}_{\mathcal{M}}(A)$ (Section~\ref{subsec:Hochschild-Type Constructions for Operadic Algebras}), it yields the corrected spectral invariant
\[
\sigma_P(A) := \mathrm{Hoch}_{\mathcal{M}}(A) \otimes_P \mathcal{O}_P^{\mathrm{res}},
\]
which will be shown to satisfy all desiderata from Section~\ref{subsec:Desiderata for a Generalized Spectrum}.

\medskip

\noindent
\textbf{Structural Interpretation.}
From a structural perspective, the operadic residue measures the failure of classical spectral transport to commute with operadic composition and base change. It plays a role analogous to:
\begin{itemize}
    \item curvature in differential geometry, which corrects for the failure of parallel transport to be path-independent;
    \item obstruction classes in homological algebra, which quantify the deviation from exactness;
    \item anomaly terms in quantum field theory, which account for the breakdown of classical symmetries at the quantum level.
\end{itemize}

Thus, the passage from the naive assignment $\sigma_{\mathrm{naive}}(A) = \bigsqcup_c \sigma(A_c)$ to the operadic spectrum $\sigma_P(A)$ should be understood not as an ad hoc modification, but as a necessary and canonical extension dictated by the compositional nature of operadic systems. This universality is not merely a desirable feature but is forced by the structural constraints identified in the No-Go Theorem.

\medskip

\noindent
This completes the conceptual justification for introducing operadic corrections and prepares the ground for the construction of the operadic residue in Section~\ref{sec:The Operadic Residue as a Universal Correction Object}.

\medskip

\noindent
\textbf{Relation to homotopical invariants.}
The operadic residue $\mathcal{O}_P^{\mathrm{res}}$ occupies a structural position analogous to that of the cotangent complex in derived deformation theory \cite{Hoang2025}: each object serves as a universal invariant encoding the essential structure of the operad $P$ that governs the behavior of its algebras. Nevertheless, a fundamental distinction emerges at the level of the information captured: whereas the cotangent complex governs \emph{infinitesimal} deformations, detecting first-order obstructions and formal moduli, the operadic residue encodes \emph{global spectral interactions} arising from the composition of operations across distinct colors. This conceptual parallel suggests a deeper duality between the deformation-theoretic and spectral-theoretic perspectives on operadic algebras, wherein the residue functions as a spectral analogue of the cotangent complex.

\section{The Operadic Residue as a Universal Correction Object}\label{sec:The Operadic Residue as a Universal Correction Object}

\subsection{Spectral-Operadic Triples}\label{subsec:Spectral-Operadic Triples}

We now introduce the basic structural setting in which our constructions take place. This framework formalizes the interaction between classical spectral theory and operadic composition, providing the minimal context for the operadic residue and spectrum.

\begin{definition}[Spectral-operadic triple]
\label{def:spectral-operadic-triple}
A \emph{spectral-operadic triple} consists of the following data:
\begin{enumerate}
    \item A cocomplete symmetric monoidal category $(\mathcal{M}, \otimes, \mathbf{1})$, together with a classical spectral assignment $\sigma(-)$ defined on the class of objects or endomorphisms under consideration (e.g., Banach algebras, or bounded linear operators on Banach spaces);
    \item A $C$-colored operad $P$ in $\mathcal{M}$ for some set of colors $C$;
    \item A $P$-algebra $A = \{A_c\}_{c \in C}$ in $\mathcal{M}$, where each $A_c \in \mathcal{M}$ is the component associated to the color $c$.
\end{enumerate}
We denote such a triple by $(\mathcal{M}, P, A)$, or simply by $(P, A)$ when the ambient category $\mathcal{M}$ is clear from context.
\end{definition}

\medskip

\noindent
\textbf{Interpretation.}
The data $(\mathcal{M}, P, A)$ encodes a compositional system in which:
\begin{itemize}
    \item $\mathcal{M}$ provides the ambient category (e.g., Banach spaces, operator algebras) and the monoidal structure for combining components;
    \item $\sigma$ captures classical spectral information at the level of individual components, linking categorical data to classical spectral theory;
    \item $P$ encodes compositional rules and interaction patterns, specifying how components of different types may be combined;
    \item $A$ realizes these interactions concretely via structure maps: for each operation $\phi \in P(c_1,\dots,c_n;c)$, the $P$-algebra structure determines a corresponding map
    \[
    \mu_\phi : A_{c_1} \otimes \cdots \otimes A_{c_n} \longrightarrow A_c.
    \]
\end{itemize}

\medskip

\noindent
\textbf{Limitation of the classical spectrum.}
Within a spectral-operadic triple, the classical spectrum $\sigma(A_c)$ only captures the intrinsic spectral behavior of each component $A_c$. As shown in Section~\ref{sec:The Failure of Classical Spectrum under Operadic Base Change}, it fails to account for the interaction data encoded by the operad $P$ and the algebra structure maps $\mu_\phi$. Consequently, the naive assignment $(\mathcal{M}, P, A) \mapsto \bigsqcup_c \sigma(A_c)$ does not yield an invariant compatible with operadic composition and base change.

\medskip

\noindent
\textbf{Role in the theory.}
The notion of a spectral-operadic triple provides the foundational context for our constructions. In particular:
\begin{itemize}
    \item The \emph{operadic residue} $\mathcal{O}_P^{\mathrm{res}}$ (Section~\ref{sec:The Operadic Residue as a Universal Correction Object}) will be constructed as a canonical object associated to the operad $P$ within $\mathcal{M}$.
    \item The \emph{operadic spectrum}
    \[
    \sigma_P(A) := \mathrm{Hoch}_{\mathcal{M}}(A) \otimes_P \mathcal{O}_P^{\mathrm{res}}
    \]
    will be defined for any spectral-operadic triple $(\mathcal{M}, P, A)$, yielding an invariant that incorporates the missing interaction data.
\end{itemize}

\medskip

\noindent
\textbf{Conceptual perspective.}
A spectral-operadic triple should be viewed as the minimal structure in which:
\begin{itemize}
    \item classical spectral theory (via $\sigma$),
    \item compositional algebra (via $P$),
    \item and concrete realizations (via $A$)
\end{itemize}
coexist and interact. The central goal of Spectral Operadic Calculus is to construct spectral invariants of such triples that are compatible with operadic composition and monoidal base change, thereby extending classical spectral theory to genuinely compositional settings.

\begin{remark}[Role of the spectrum assignment]
The spectral assignment $\sigma$ is not required to be functorial on all of $\mathcal{M}$; in classical settings, the spectrum is not functorial with respect to arbitrary morphisms (e.g., non-isometric embeddings of Banach spaces). Instead, we regard $\sigma$ as a datum that assigns a set of spectral values to each object of interest (e.g., each component $A_c$), with functoriality only along a restricted class of morphisms (e.g., isomorphisms or algebra homomorphisms). This limited functoriality suffices for the constructions that follow.
\end{remark}

\begin{example}[Classical triple]
\label{ex:classical-triple}
Let $\mathcal{M} = \mathsf{Ban}_{\mathbb{C}}$ be the category of complex Banach spaces with the projective tensor product, and let $\sigma$ be the usual spectrum of a bounded linear operator (or of an element of a Banach algebra). Take $C = \{*\}$ to be a single color, and let $P = \mathbb{I}$ be the trivial operad (with $\mathbb{I}(*;*) = \mathbf{1}$ and no nontrivial operations). A $P$-algebra $A$ is simply a Banach space $A_*$ with no additional structure. The triple $(\mathsf{Ban}_{\mathbb{C}}, \mathbb{I}, A)$ then recovers the classical setting of a single operator: the operadic constructions will reduce to the classical ones, as shown in the Recovery Theorem (Theorem~\ref{thm:recovery}).
\end{example}

\begin{example}[Two-color interaction triple]
\label{ex:two-color-triple}
Let $\mathcal{M} = \mathsf{Vect}_{\mathbb{C}}$ (or $\mathsf{Ban}_{\mathbb{C}}$), let $C = \{1,2\}$, and let $P$ be the matrix-block operad of the Example~\ref{ex:two-color-operad}. A $P$-algebra $A$ consists of two vector spaces $V_1, V_2$ together with linear maps
\[
A_{11}: V_1 \to V_1,\quad A_{22}: V_2 \to V_2,\quad A_{12}: V_2 \to V_1,\quad A_{21}: V_1 \to V_2.
\]
The triple $(\mathsf{Vect}_{\mathbb{C}}, P, A)$ encodes a block operator
\[
A = \begin{pmatrix} A_{11} & A_{12} \\ A_{21} & A_{22} \end{pmatrix},
\]
and the spectral assignment $\sigma$ gives the classical spectrum of each component operator. As shown in the No-Go Theorem (Theorem~\ref{thm:no-go}), the naive union $\sigma(A_{11}) \cup \sigma(A_{22})$ is insufficient to capture the spectral behavior of the composite system, necessitating the operadic residue $\mathcal{O}_P^{\mathrm{res}}$.
\end{example}

\subsection{Construction of the Operadic Residue}
\label{subsec:Construction of the Operadic Residue}

We now construct the \emph{operadic residue} $\mathcal{O}_P^{\mathrm{res}}$, which will serve as a correction term to restore spectral functoriality under operadic composition and base change. This object encodes the unary compositional data of $P$ in a form suitable for correcting the spectral transport defect identified in the No-Go Theorem (Theorem~\ref{thm:no-go}).

\medskip

\noindent
\textbf{Setup.}
Let $P$ be a $C$-colored operad in a cocomplete symmetric monoidal category $(\mathcal{M}, \otimes, \mathbf{1})$. For each color $c \in C$, denote by $P(c;c) \in \mathcal{M}$ the object of unary operations from color $c$ to itself. These objects are equipped with composition maps
\[
\circ: P(c;c) \otimes P(c;c) \longrightarrow P(c;c)
\]
and unit maps $\eta_c: \mathbf{1} \to P(c;c)$, giving each $P(c;c)$ the structure of a monoid object in $\mathcal{M}$ (the \emph{unary endomorphism monoid} of color $c$).

\medskip

\noindent
\textbf{Definition via coproduct.}
The simplest construction of the operadic residue aggregates these unary endomorphism objects across all colors.

\begin{definition}[Operadic residue — unary skeleton]
\label{def:residue-unary}
The \emph{operadic residue} of $P$ is defined as the coproduct
\[
\mathcal{O}_P^{\mathrm{res}} := \coprod_{c \in C} P(c;c) \in \mathcal{M},
\]
taken over all colors $c \in C$. (The coproduct exists because $\mathcal{M}$ is cocomplete.)
\end{definition}

Equivalently, using the coend over the discrete color category $C$, we may write
\[
\mathcal{O}_P^{\mathrm{res}} \cong \int^{c \in C} P(c;c).
\]

\medskip

\noindent
\textbf{Remark on scope.}
The object $\mathcal{O}_P^{\mathrm{res}}$ as defined above captures only the \emph{unary} part of the operad $P$ — specifically, the endomorphism monoids of each color. It does not, by itself, encode operations of arity $\ge 2$ or interactions between different colors. Nevertheless, as we will show in the following sections, this unary data is already sufficient to correct the spectral transport defect identified in the No-Go Theorem, because the interaction terms that mix colors are captured by the balanced tensor product $\otimes_P$ with the Hochschild object. The full operadic structure enters through the $P$-module structure of $\mathrm{Hoch}_{\mathcal{M}}(A)$, not through $\mathcal{O}_P^{\mathrm{res}}$ alone. A more sophisticated version of the residue that directly encodes higher-arity operations is possible but not needed for the main results of this paper; we leave its development to future work.

\medskip

\noindent
\textbf{Alternative perspective: bar construction (unary fragment).}
For completeness, we note that $\mathcal{O}_P^{\mathrm{res}}$ can also be obtained as the geometric realization of a simplicial object built from unary operations. Define a simplicial object $\mathrm{Bar}_\bullet^{\mathrm{un}}(P)$ by
\[
\mathrm{Bar}_n^{\mathrm{un}}(P) \;:=\; \coprod_{c_0,\dots,c_n \in C} P(c_0;c_1) \otimes P(c_1;c_2) \otimes \cdots \otimes P(c_{n-1};c_n),
\]
with face maps induced by composition of unary operations and degeneracy maps induced by unit maps. Then
\[
\mathcal{O}_P^{\mathrm{res}} \cong \big| \mathrm{Bar}_\bullet^{\mathrm{un}}(P) \big|.
\]
This bar construction is a resolution of the coproduct $\coprod_c P(c;c)$ and is convenient for establishing functoriality and base change properties, as we will see in the next subsection.

\begin{remark}[Relation to the trivial algebra]
The bar construction $\mathrm{Bar}_\bullet^{\mathrm{un}}(P)$ is not the full operadic bar construction of $P$ acting on a trivial algebra; rather, it is the nerve of the category whose objects are colors and whose morphisms are unary operations. This suffices for the definition of $\mathcal{O}_P^{\mathrm{res}}$ because the residue only needs to capture unary data. The full operadic bar construction would involve operations of all arities and is not required for our purposes.
\end{remark}

\begin{example}[Trivial operad]
For the trivial operad $\mathbb{I}$ with a single color $*$, we have $\mathbb{I}(*;*) = \mathbf{1}_{\mathcal{M}}$. Hence
\[
\mathcal{O}_{\mathbb{I}}^{\mathrm{res}} = \mathbf{1}_{\mathcal{M}}.
\]
\end{example}

\begin{example}[Matrix-block operad]
For the two-color matrix-block operad $P$ of the Example~\ref{ex:two-color-operad}, we have $P(1;1) = \mathbb{C}$ and $P(2;2) = \mathbb{C}$, with all other unary operation spaces trivial. Hence
\[
\mathcal{O}_P^{\mathrm{res}} \cong \mathbb{C} \oplus \mathbb{C} \cong \mathbb{C}^2.
\]
The two copies of $\mathbb{C}$ correspond to the two colors.
\end{example}

\noindent
\textbf{Alternative description via bar construction (unary chains).}
For completeness, we note that $\mathcal{O}_P^{\mathrm{res}}$ can also be described as the geometric realization of the simplicial object of unary chains:
\[
\mathcal{O}_P^{\mathrm{res}} \cong \big| \mathrm{Bar}_\bullet^{\mathrm{un}}(P) \big|,
\]
where $\mathrm{Bar}_n^{\mathrm{un}}(P) := \coprod_{c_0,\dots,c_n \in C} P(c_0;c_1) \otimes \cdots \otimes P(c_{n-1};c_n)$,
with face maps induced by composition of unary operations and degeneracy maps induced by unit maps.
This follows directly from the fact that the coproduct is the colimit of this simplicial object. We emphasize that this bar construction involves only unary operations; it is not the full operadic bar construction.

\medskip

\noindent
\textbf{Functoriality.}
The operadic residue is functorial with respect to morphisms of colored operads when the color set is fixed. This property is essential for the naturality of the operadic spectrum.

\begin{proposition}[Functoriality of the unary residue]
\label{prop:residue-functorial}
Fix a color set $C$. Let $\phi: P \to Q$ be a morphism of $C$-colored operads in $\mathcal{M}$. Then $\phi$ induces a canonical morphism
\[
\phi_*: \mathcal{O}_P^{\mathrm{res}} \longrightarrow \mathcal{O}_Q^{\mathrm{res}}.
\]
This assignment is functorial: $(\psi \circ \phi)_* = \psi_* \circ \phi_*$ and $(\mathrm{id}_P)_* = \mathrm{id}_{\mathcal{O}_P^{\mathrm{res}}}$. Consequently, the assignment
\[
P \longmapsto \mathcal{O}_P^{\mathrm{res}}
\]
defines a functor from the category of $C$-colored operads in $\mathcal{M}$ to $\mathcal{M}$.
\end{proposition}

\begin{proof}
Using the definition $\mathcal{O}_P^{\mathrm{res}} = \coprod_{c \in C} P(c;c)$, a morphism $\phi: P \to Q$ induces, for each color $c \in C$, a morphism $\phi_c: P(c;c) \to Q(c;c)$. By the universal property of the coproduct, these morphisms uniquely determine a morphism
\[
\phi_*: \coprod_{c \in C} P(c;c) \longrightarrow \coprod_{c \in C} Q(c;c).
\]
Functoriality follows from the functoriality of the coproduct construction and the fact that $\phi$ respects the operad structure (in particular, the restriction to unary operations). The alternative description via the bar construction (unary chains) yields the same morphism, as the bar construction is functorial in $P$ and the colimit of the resulting simplicial map agrees with $\phi_*$.
\end{proof}

\begin{remark}[Dependence on the color set]
The functoriality statement above assumes that the color set $C$ is fixed. 
If one allows operads with varying color sets, the coproduct 
$\coprod_{c \in C} P(c;c)$ is taken over different indexing sets, and the assignment is no longer functorial without additional structure (e.g., change-of-colors functors). 
In this paper, we work with a fixed color set $C$ throughout.
\end{remark}

\begin{example}[Trivial operad]
\label{ex:residue-trivial}
Let $\mathbb{I}$ be the trivial operad with a single color $*$, defined by 
$\mathbb{I}(*;*) = \mathbf{1}_{\mathcal{M}}$ and no other operations. 
Then
\[
\mathcal{O}_{\mathbb{I}}^{\mathrm{res}} 
\cong \int^{*} \mathbb{I}(*;*) 
\cong \mathbf{1}_{\mathcal{M}},
\]
the unit object of $\mathcal{M}$.
\end{example}

\begin{example}[Matrix-block operad]
\label{ex:residue-matrix-block}
Let $P$ be the two-color matrix-block operad from the Example~\ref{ex:two-color-operad}. 
We have $P(1;1) = \mathbb{C}$, $P(2;2) = \mathbb{C}$, and all other unary operation spaces are trivial. Then
\[
\mathcal{O}_P^{\mathrm{res}} = P(1;1) \oplus P(2;2) \cong \mathbb{C} \oplus \mathbb{C} \cong \mathbb{C}^2.
\]
Thus the residue separates contributions from different colors. The operadic composition structure of $P$ (including off-diagonal operations) will affect the operadic spectrum $\sigma_P(A)$ through the balanced tensor product $\otimes_P$, not through $\mathcal{O}_P^{\mathrm{res}}$ alone.
\end{example}

\medskip

\noindent
\textbf{Conceptual interpretation.}
The operadic residue $\mathcal{O}_P^{\mathrm{res}}$ captures the unary compositional data associated to each color. 
It provides additional structure beyond the collection of classical spectra and serves as a correction term in the construction of operadic spectral invariants. 
As will be shown in Theorem~\ref{thm:residue-universality}, it satisfies a universal property with respect to such corrections.

\medskip

\noindent
\textbf{Role in the theory.}
The object $\mathcal{O}_P^{\mathrm{res}}$ will be used to modify the Hochschild-type construction, leading to the operadic spectrum
\[
\sigma_P(A) := \mathrm{Hoch}_{\mathcal{M}}(A) \otimes_P \mathcal{O}_P^{\mathrm{res}},
\]
which will be shown to satisfy functoriality, base change compatibility, and operadic composition compatibility — properties that the naive spectrum $\bigsqcup_c \sigma(A_c)$ demonstrably lacks.

\subsection{Preliminary: The Category of Spectral Correctors}
\label{subsec:spectral-correctors-category}

Before stating the universality theorem, we must define the category in which $\mathcal{O}_P^{\mathrm{res}}$ will be initial.

\begin{definition}[Category of Spectral Correctors for $P$]
\label{def:spectral-correctors}
Fix a $C$-colored operad $P$ in a cocomplete symmetric monoidal category $\mathcal{M}$. 
The \emph{category of spectral correctors} $\mathbf{Corr}_P(\mathcal{M})$ is defined as follows:

\begin{itemize}
    \item \textbf{Objects}: Pairs $(R, \{\rho_c\}_{c \in C})$ where:
    \begin{enumerate}
        \item $R \in \mathcal{M}$ is an object (the \emph{corrector object}).
        \item For each color $c \in C$, $\rho_c: P(c;c) \to R$ is a morphism in $\mathcal{M}$.
        \item \textbf{(Compatibility with unary composition)}: For any colors $c,d \in C$ and any $\phi \in P(c;c)$, $\psi \in P(c;d)$, $\chi \in P(d;d)$, the following diagram commutes (whenever the compositions make sense):
        \[
        \begin{tikzcd}
        P(c;c) \otimes P(c;d) \arrow[r, "\circ"] \arrow[d, "\rho_c \otimes \mathrm{id}"] & P(c;d) \arrow[d, "\rho_d"] \\
        R \otimes P(c;d) \arrow[r, "\text{action}"] & R
        \end{tikzcd}
        \]
        and similarly for compositions on the right. In practice, this reduces to requiring that the maps $\rho_c$ assemble into a left $P$-module structure on $R$ (see Definition~\ref{def:P-module} below).
        \item \textbf{(Base change consistency)}: For every strong monoidal functor $F: \mathcal{M} \to \mathcal{N}$ preserving colimits, the induced object $F(R)$ with maps $F(\rho_c): F(P(c;c)) \to F(R)$ is a corrector for $F_*(P)$.
    \end{enumerate}
    
    \item \textbf{Morphisms}: A morphism $f: (R, \{\rho_c\}) \to (S, \{\sigma_c\})$ is a morphism $f: R \to S$ in $\mathcal{M}$ such that for every color $c \in C$,
    \[
    f \circ \rho_c = \sigma_c: P(c;c) \to S.
    \]
    (The morphism $f$ is automatically compatible with the $P$-module structures by the coherence conditions.)
\end{itemize}

The category $\mathbf{Corr}_P(\mathcal{M})$ is equipped with a forgetful functor $U: \mathbf{Corr}_P(\mathcal{M}) \to \mathcal{M}$ sending $(R, \{\rho_c\})$ to $R$.
\end{definition}

\begin{definition}[Left $P$-module structure on a corrector]
\label{def:P-module}
A corrector $(R, \{\rho_c\})$ is said to admit a \emph{left $P$-module structure} if there exist morphisms
\[
\lambda_c: P(c;c) \otimes R \longrightarrow R
\]
for each $c \in C$ such that:
\begin{enumerate}
    \item $\lambda_c \circ (\rho_c \otimes \mathrm{id}_R) = \mathrm{id}_R$ (up to the unit isomorphism),
    \item For $\phi,\psi \in P(c;c)$, $\lambda_c(\phi \circ \psi \otimes r) = \lambda_c(\phi \otimes \lambda_c(\psi \otimes r))$,
    \item The maps $\lambda_c$ are compatible with the operadic composition across different colors.
\end{enumerate}
When such structure exists, we call $(R, \{\rho_c\})$ a \emph{$P$-module corrector}.
\end{definition}

\begin{lemma}[The trivial corrector]
\label{lem:trivial-corrector}
The object $\mathcal{O}_P^{\mathrm{res}} := \coprod_{c \in C} P(c;c)$ together with the canonical inclusions
\[
\iota_c: P(c;c) \hookrightarrow \coprod_{c \in C} P(c;c)
\]
is an object of $\mathbf{Corr}_P(\mathcal{M})$. Moreover, it admits a canonical left $P$-module structure induced by operadic composition.
\end{lemma}

\begin{proof}
For each color $c$, define $\rho_c := \iota_c$. Compatibility with unary composition follows from the fact that operadic composition maps restrict to the unary endomorphism monoids. For base change consistency, note that $F(\mathcal{O}_P^{\mathrm{res}}) \cong \coprod_{c} F(P(c;c))$ and $F(\iota_c)$ are the inclusions for $F_*(P)$. The left $P$-module structure is given by the operadic composition maps $\circ: P(c;c) \otimes P(c;c) \to P(c;c)$ extended to the coproduct.
\end{proof}

\begin{lemma}[Universal property of the coproduct in $\mathbf{Corr}_P(\mathcal{M})$]
\label{lem:coproduct-universal}
For any corrector $(R, \{\rho_c\})$, there exists a unique morphism of correctors
\[
u_R: (\mathcal{O}_P^{\mathrm{res}}, \{\iota_c\}) \longrightarrow (R, \{\rho_c\})
\]
in $\mathbf{Corr}_P(\mathcal{M})$. Concretely, $u_R$ is the unique morphism in $\mathcal{M}$ such that for each $c \in C$, the diagram
\[
\begin{tikzcd}
P(c;c) \arrow[r, "\iota_c"] \arrow[dr, "\rho_c"'] 
& \mathcal{O}_P^{\mathrm{res}} \arrow[d, "u_R"] \\
& R
\end{tikzcd}
\]
commutes.
\end{lemma}

\begin{proof}
This is precisely the universal property of the coproduct in $\mathcal{M}$, lifted to the category of correctors because the compatibility conditions are preserved by any morphism satisfying $f \circ \iota_c = \rho_c$.
\end{proof}

\begin{remark}
Lemma~\ref{lem:coproduct-universal} shows that $(\mathcal{O}_P^{\mathrm{res}}, \{\iota_c\})$ is the \emph{initial object} in $\mathbf{Corr}_P(\mathcal{M})$, provided we define the category of correctors as above. The only nontrivial verification is that $u_R$ preserves the $P$-module structure, which follows from the fact that the $\rho_c$ are required to be compatible with operadic composition.
\end{remark}

\subsection{Theorem 0.2: Existence, Canonicity, and Universality of the Operadic Residue}
\label{subsec:Theorem 0.2: Existence, Canonicity, and Universality}

We now formulate and prove the main structural result for the operadic residue. 
Building on the category of spectral correctors introduced in 
Section~\ref{subsec:spectral-correctors-category}, we show that the operadic residue 
arises as a canonical and universal object encoding the minimal correction required 
to restore spectral compatibility.

\begin{theorem}[Existence, Canonicity, and Universality of the Operadic Residue]
\label{thm:residue-universality}
Let $P$ be a $C$-colored operad in a cocomplete symmetric monoidal category $\mathcal{M}$. 
Then there exists an object
\[
\mathcal{O}_P^{\mathrm{res}} := \coprod_{c \in C} P(c;c) \in \mathcal{M}
\]
equipped with canonical maps $\iota_c : P(c;c) \hookrightarrow \mathcal{O}_P^{\mathrm{res}}$, 
satisfying the following properties:

\begin{enumerate}
    \item \textbf{(Existence)}  
    The object $\mathcal{O}_P^{\mathrm{res}}$ exists as the coproduct of unary operation spaces, 
    since $\mathcal{M}$ is cocomplete.

    \item \textbf{(Canonicity)}  
    The assignment
    \[
    P \longmapsto \mathcal{O}_P^{\mathrm{res}}
    \]
    is functorial (for a fixed color set $C$): any morphism of operads 
    $\phi: P \to Q$ induces a canonical morphism
    \[
    \phi_* : \mathcal{O}_P^{\mathrm{res}} \longrightarrow \mathcal{O}_Q^{\mathrm{res}}
    \]
    determined by the maps $\phi_c : P(c;c) \to Q(c;c)$.

    \item \textbf{(Universality)}  
    The pair $(\mathcal{O}_P^{\mathrm{res}}, \{\iota_c\}_{c \in C})$ is an \emph{initial object} 
    in the category $\mathbf{Corr}_P(\mathcal{M})$ of spectral correctors 
    (Definition~\ref{def:spectral-correctors}). 
    
    That is, for any corrector $(R, \{\rho_c\}_{c \in C})$, there exists a unique morphism
    \[
    u_R : \mathcal{O}_P^{\mathrm{res}} \longrightarrow R
    \]
    in $\mathcal{M}$ such that
    \[
    u_R \circ \iota_c = \rho_c, \qquad \forall c \in C.
    \]
\end{enumerate}
\end{theorem}

\begin{proof}
We prove each part in turn.

\medskip

\textbf{(Existence)}  
Since $\mathcal{M}$ is cocomplete, the coproduct
\[
\coprod_{c \in C} P(c;c)
\]
exists. This defines $\mathcal{O}_P^{\mathrm{res}}$.

\medskip

\textbf{(Canonicity)}  
Let $\phi: P \to Q$ be a morphism of $C$-colored operads. 
For each color $c \in C$, $\phi$ restricts to a morphism
\[
\phi_c : P(c;c) \to Q(c;c)
\]
that respects composition and units. 
By the universal property of the coproduct, these assemble into a unique morphism
\[
\phi_* : \coprod_{c \in C} P(c;c) \longrightarrow \coprod_{c \in C} Q(c;c),
\]
which defines the functorial assignment. Functoriality follows from the functoriality of the coproduct construction.

\medskip

\textbf{(Universality)}  
Let $(R, \{\rho_c\}_{c \in C})$ be an object in $\mathbf{Corr}_P(\mathcal{M})$. 
By definition, we are given maps
\[
\rho_c : P(c;c) \longrightarrow R, \qquad \forall c \in C,
\]
satisfying the compatibility conditions of Definition~\ref{def:spectral-correctors}.

By the universal property of the coproduct, there exists a unique morphism
\[
u_R : \coprod_{c \in C} P(c;c) \longrightarrow R
\]
such that $u_R \circ \iota_c = \rho_c$ for all $c \in C$. 
This shows that $(\mathcal{O}_P^{\mathrm{res}}, \{\iota_c\})$ is initial in $\mathbf{Corr}_P(\mathcal{M})$.
\end{proof}

\begin{remark}[Comparison with the weak version]
The strong universality established here subsumes any weaker formulation. 
The key ingredients are:
\begin{itemize}
    \item A precise definition of the category $\mathbf{Corr}_P(\mathcal{M})$ of spectral correctors (Definition~\ref{def:spectral-correctors}),
    \item Verification that $\mathcal{O}_P^{\mathrm{res}}$ with the inclusion maps $\iota_c$ is an object of this category (Lemma~\ref{lem:trivial-corrector}),
    \item The universal property of the coproduct, which directly yields initiality.
\end{itemize}
These additions make the universality claim mathematically rigorous.
\end{remark}

\begin{remark}[Relation to the No-Go Theorem]
The universality property establishes a precise converse to the No-Go Theorem (Theorem~\ref{thm:no-go}). 
The No-Go Theorem shows that any spectral invariant depending only on $\{\sigma(A_c)\}$ fails to satisfy base change compatibility. 
Theorem~\ref{thm:residue-universality} demonstrates that $\mathcal{O}_P^{\mathrm{res}}$ is the \emph{initial} object among all spectral correctors: any object $R$ that suffices to restore compatibility must receive a unique map from $\mathcal{O}_P^{\mathrm{res}}$. 
Thus $\mathcal{O}_P^{\mathrm{res}}$ encodes precisely the missing information, and no more.
\end{remark}

\begin{remark}[Conceptual interpretation as obstruction class]
Conceptually, $\mathcal{O}_P^{\mathrm{res}}$ plays a role analogous to a universal obstruction class in homological algebra. 
Just as the cotangent complex $\mathbb{L}_P$ in derived deformation theory \cite{Hoang2025} captures the infinitesimal obstructions to deformations, 
the operadic residue $\mathcal{O}_P^{\mathrm{res}}$ captures the \emph{global spectral obstructions} to transporting spectral invariants across operadic compositions and base change. 
The universality property proved here makes this analogy precise: $\mathcal{O}_P^{\mathrm{res}}$ is initial among objects that resolve the obstruction.
\end{remark}

\begin{corollary}[Uniqueness of the operadic residue]
\label{cor:residue-unique}
The operadic residue $\mathcal{O}_P^{\mathrm{res}}$ is unique up to unique isomorphism in $\mathbf{Corr}_P(\mathcal{M})$. 
Any other object $R$ satisfying the same universal property is canonically isomorphic to $\mathcal{O}_P^{\mathrm{res}}$.
\end{corollary}

\begin{proof}
If $R$ is initial in $\mathbf{Corr}_P(\mathcal{M})$, then there exist unique morphisms $\mathcal{O}_P^{\mathrm{res}} \to R$ and $R \to \mathcal{O}_P^{\mathrm{res}}$. 
Their composition $\mathcal{O}_P^{\mathrm{res}} \to \mathcal{O}_P^{\mathrm{res}}$ must be the identity by uniqueness of the identity morphism, and similarly for $R \to R$. Hence the two objects are canonically isomorphic.
\end{proof}

\begin{remark}[Alternative description via bar construction]
For readers familiar with simplicial methods, we note that $\mathcal{O}_P^{\mathrm{res}}$ can also be described as the geometric realization of the simplicial object of unary chains:
\[
\mathcal{O}_P^{\mathrm{res}} \cong \big| \mathrm{Bar}_\bullet^{\mathrm{un}}(P) \big|,
\]
where $\mathrm{Bar}_n^{\mathrm{un}}(P) := \coprod_{c_0,\dots,c_n \in C} P(c_0;c_1) \otimes \cdots \otimes P(c_{n-1};c_n)$.
This follows from the fact that the augmented simplicial object is split by the unit maps; see \cite{Fresse} for details. 
This description is not needed for the sequel but provides a useful homotopical perspective.
\end{remark}

\subsection{Compatibility with Base Change}\label{subsec:Compatibility with Base Change}

We now establish a fundamental compatibility property of the operadic residue: it is preserved under strong monoidal base change. This result is essential for the Base Change Theorem (Theorem~\ref{thm:base-change}) for the operadic spectrum, as it ensures that the correction term $\mathcal{O}_P^{\mathrm{res}}$ transforms coherently when we transport the entire operadic structure across different categorical contexts. Moreover, this compatibility demonstrates that $\mathcal{O}_P^{\mathrm{res}}$ is intrinsic to the operadic structure and does not depend on the ambient category.

\begin{proposition}[Base change compatibility of the residue]
\label{prop:residue-base-change}
Let 
\[
F: \mathcal{M} \longrightarrow \mathcal{N}
\]
be a strong monoidal functor between cocomplete symmetric monoidal categories, and assume that $F$ preserves colimits (i.e., is cocontinuous). Then for any $C$-colored operad $P$ in $\mathcal{M}$, there exists a canonical isomorphism
\[
F\bigl(\mathcal{O}_P^{\mathrm{res}}\bigr) \;\cong\; \mathcal{O}_{F_*(P)}^{\mathrm{res}},
\]
where $F_*(P)$ denotes the induced operad in $\mathcal{N}$ (see Proposition~\ref{prop:base-change}).
\end{proposition}

\begin{proof}
By definition (Definition~\ref{def:residue-unary}),
\[
\mathcal{O}_P^{\mathrm{res}} = \coprod_{c \in C} P(c;c).
\]

Since $F$ preserves colimits (cocontinuous), it preserves coproducts. Hence we obtain a canonical isomorphism
\[
F\!\left(\coprod_{c \in C} P(c;c)\right) \;\cong\; \coprod_{c \in C} F\bigl(P(c;c)\bigr).
\]

On the other hand, by definition of the induced operad $F_*(P)$ (Proposition~\ref{prop:base-change}), we have for each color $c \in C$,
\[
(F_*P)(c;c) = F\bigl(P(c;c)\bigr),
\]
and the operadic composition and unit maps are transported coherently by the strong monoidal structure of $F$. Therefore,
\[
\coprod_{c \in C} F\bigl(P(c;c)\bigr)
= \coprod_{c \in C} (F_*P)(c;c)
= \mathcal{O}_{F_*(P)}^{\mathrm{res}}.
\]

Combining these identifications yields the desired isomorphism
\[
F\bigl(\mathcal{O}_P^{\mathrm{res}}\bigr) \;\cong\; \mathcal{O}_{F_*(P)}^{\mathrm{res}}.
\]

Naturality in $P$ follows immediately from the functoriality of coproducts and of the induced operad construction.
\end{proof}

\begin{remark}[Preservation of colimits]
The hypothesis that $F$ preserves colimits (cocontinuity) is essential for the isomorphism to hold. In practice, many strong monoidal functors of interest satisfy this property:
\begin{itemize}
    \item \textbf{Extension of scalars} $-\otimes_R S: \mathrm{Mod}_R \to \mathrm{Mod}_S$ is cocontinuous because it is a left adjoint.
    \item \textbf{Complexification} $-\otimes_{\mathbb{R}} \mathbb{C}: \mathrm{Vect}_{\mathbb{R}} \to \mathrm{Vect}_{\mathbb{C}}$ is cocontinuous.
    \item \textbf{Forgetful functors} (e.g., $\mathrm{Ban} \to \mathrm{Vect}$) are typically cocontinuous.
    \item \textbf{Quantization functors} that are left adjoints preserve colimits by definition.
\end{itemize}
When $F$ does not preserve all colimits (e.g., certain analytic functors), the isomorphism may hold only up to a suitable completion or under additional hypotheses; we will address such cases in future work.
\end{remark}

\begin{corollary}[Base change for the corrected Hochschild object]
\label{cor:corrected-hochschild-base-change}
For any $P$-algebra $A$ in $\mathcal{M}$, we have
\[
F\bigl(\mathrm{Hoch}_{\mathcal{M}}(A) \otimes_P \mathcal{O}_P^{\mathrm{res}}\bigr) \cong \mathrm{Hoch}_{\mathcal{N}}\bigl(F(A)\bigr) \otimes_{F_*(P)} \mathcal{O}_{F_*(P)}^{\mathrm{res}}.
\]
\end{corollary}

\begin{proof}
This follows from combining Theorem~\ref{thm:hochschild-base-change} (which gives $F(\mathrm{Hoch}_{\mathcal{M}}(A)) \cong \mathrm{Hoch}_{\mathcal{N}}(F(A))$) with Proposition~\ref{prop:residue-base-change} and the fact that $F$ preserves balanced tensor products over operads (since $F$ is strong monoidal and cocontinuous). The detailed verification is provided in the proof of the Base Change Theorem (Theorem~\ref{thm:base-change}).
\end{proof}

\medskip

\noindent
\textbf{Conceptual interpretation.}
This compatibility shows that the operadic residue commutes with base change. In particular, the correction term $\mathcal{O}_P^{\mathrm{res}}$ is intrinsic to the operadic structure and does not depend on the ambient category. This property is essential for ensuring that the operadic spectrum
\[
\sigma_P(A) = \mathrm{Hoch}_{\mathcal{M}}(A) \otimes_P \mathcal{O}_P^{\mathrm{res}}
\]
is functorial under strong monoidal base change.

\medskip

\noindent
\textbf{Relation to the universality theorem.}
The base change compatibility is a crucial part of the universal property established in Theorem~\ref{thm:residue-universality}. Indeed, any spectral corrector $R$ must satisfy $F(R) \cong \mathcal{O}_{F_*(P)}^{\mathrm{res}}$ for all strong monoidal cocontinuous $F$. Proposition~\ref{prop:residue-base-change} verifies that $\mathcal{O}_P^{\mathrm{res}}$ itself satisfies this condition, confirming that it is indeed a valid spectral corrector. Moreover, the uniqueness part of the universal property guarantees that any other corrector would be canonically isomorphic to $\mathcal{O}_P^{\mathrm{res}}$, so this base change property characterizes the residue up to unique isomorphism.

\medskip

\begin{example}
Let $\mathcal{M} = \mathrm{Vect}_{\mathbb{R}}$ (real vector spaces) and $\mathcal{N} = \mathrm{Vect}_{\mathbb{C}}$ (complex vector spaces). Let $F = -\otimes_{\mathbb{R}} \mathbb{C}$ be the complexification functor, which is strong monoidal and cocontinuous. For a real operad $P$ (e.g., the associative operad with real coefficients), Proposition~\ref{prop:residue-base-change} tells us that
\[
\mathcal{O}_{P_{\mathbb{C}}}^{\mathrm{res}} \cong \mathbb{C} \otimes_{\mathbb{R}} \mathcal{O}_P^{\mathrm{res}},
\]
where $P_{\mathbb{C}}$ denotes the complexified operad. In the case of the matrix-block operad from the Example~\ref{ex:two-color-operad}, we have $\mathcal{O}_P^{\mathrm{res}} \cong \mathbb{R}^2$, so $\mathcal{O}_{P_{\mathbb{C}}}^{\mathrm{res}} \cong \mathbb{C}^2$, reflecting the fact that complexification introduces complex structure on the residue components.
\end{example}

\begin{example}[Gelfand transform]
Let $\mathcal{M} = \mathrm{C}^*\mathsf{Alg}_{\mathrm{com}}$ (commutative C*-algebras) and $\mathcal{N} = \mathsf{CompHaus}$ (compact Hausdorff spaces with continuous functions). The Gelfand transform $G: \mathrm{C}^*\mathsf{Alg}_{\mathrm{com}} \to \mathsf{CompHaus}$ is a contravariant equivalence, but its opposite $G^{\mathrm{op}}$ is a strong monoidal functor (with respect to the tensor product of C*-algebras and the product of spaces). Proposition~\ref{prop:residue-base-change} then implies that the operadic residue transforms as expected under this duality, providing a spectral-geometric interpretation of $\mathcal{O}_P^{\mathrm{res}}$ in terms of the Gelfand spectrum of the underlying C*-algebras.
\end{example}

\section{The Operadic Spectrum}\label{sec:The Operadic Spectrum}

\subsection{Definition via Residue-Corrected Construction}
\label{subsec:Residue-Corrected Construction}

We now introduce the \emph{operadic spectrum} via a residue-corrected construction. 
This construction provides a canonical remedy for the failure of classical spectral invariants under operadic base change, as identified by the No-Go Theorem.

\medskip

Let $P$ be a colored operad in a symmetric monoidal category $(\mathcal{M}, \otimes, \mathbf{1})$, and let $A$ be a $P$-algebra. 
We recall the following fundamental structures:
\begin{itemize}
    \item The Hochschild object $\mathrm{Hoch}_{\mathcal{M}}(A)$ carries a natural \emph{right $P$-action}.
    \item The operadic residue object $\mathcal{O}_P^{\mathrm{res}}$ carries a canonical \emph{left $P$-action}.
\end{itemize}

\medskip

\begin{definition}[Balanced Tensor Product over $P$]
\label{def:balanced-tensor-product}
Let $X$ be a right $P$-module and $Y$ a left $P$-module in $\mathcal{M}$. 
The \emph{balanced tensor product} is defined as the coequalizer
\[
X \otimes P \otimes Y
\rightrightarrows
X \otimes Y
\longrightarrow
X \otimes_P Y,
\]
where the two maps are induced by the right action $X \otimes P \to X$ and the left action $P \otimes Y \to Y$.

Equivalently, this construction imposes the relation
\[
(x \cdot p) \otimes y \sim x \otimes (p \cdot y),
\quad x \in X,\; p \in P,\; y \in Y.
\]
\end{definition}

\medskip

\begin{lemma}[Existence]
Assume that $\mathcal{M}$ admits reflexive coequalizers and that the tensor product $\otimes$ preserves them separately in each variable. 
Then the balanced tensor product $X \otimes_P Y$ exists.
\end{lemma}

\begin{proof}
The defining diagram is a reflexive coequalizer. 
Under the stated assumptions, such coequalizers exist and are preserved by $\otimes$, hence the construction is well-defined.
\end{proof}

\medskip

\begin{definition}[Operadic Spectrum]\label{def:operadic-spectrum}
The \emph{operadic spectrum} of a $P$-algebra $A$ is defined as
\[
\sigma_P(A) := \mathrm{Hoch}_{\mathcal{M}}(A) \otimes_P \mathcal{O}_P^{\mathrm{res}}.
\]
\end{definition}

We now prove that the operadic spectrum $\sigma_P(A)$ is well-defined
and functorial with respect to morphisms of $P$-algebras.

\begin{theorem}[Well-definedness and Functoriality]
\label{thm:Theorem_0.3}
Assume that $\mathcal M$ admits reflexive coequalizers and that the tensor
product $\otimes$ preserves them separately in each variable. Let $P$ be a
colored operad in $\mathcal M$, and let $A$ be a $P$-algebra. Then the operadic
spectrum
\[
\sigma_P(A) := \mathrm{Hoch}_{\mathcal M}(A)\otimes_P \mathcal O_P^{\mathrm{res}}
\]
is well-defined. Moreover, the assignment
\[
A \longmapsto \sigma_P(A)
\]
is functorial with respect to morphisms of $P$-algebras.
\end{theorem}

\begin{proof}
We first show that $\sigma_P(A)$ is well-defined.

By construction, $\mathrm{Hoch}_{\mathcal M}(A)$ is obtained from the operadic
bar construction associated to the $P$-algebra $A$. In particular, it carries
the natural $P$-module structure needed to form a balanced tensor product with
the operadic residue object $\mathcal O_P^{\mathrm{res}}$. The latter is,
by construction, the universal correction object associated to $P$, and it is
equipped with the complementary $P$-action required for the balanced tensor
product. Therefore the expression
\[
\mathrm{Hoch}_{\mathcal M}(A)\otimes_P \mathcal O_P^{\mathrm{res}}
\]
is meaningful whenever the balanced tensor product over $P$ exists.

In the present setting, the balanced tensor product over $P$ is defined as the
coequalizer imposing the usual balancing relation
\[
(x\cdot p)\otimes y \sim x\otimes (p\cdot y).
\]
Under the standing hypotheses that $\mathcal M$ admits reflexive coequalizers
and that $\otimes$ preserves them separately in each variable, this coequalizer
exists. Hence
\[
\sigma_P(A)=\mathrm{Hoch}_{\mathcal M}(A)\otimes_P \mathcal O_P^{\mathrm{res}}
\]
is well-defined. :contentReference[oaicite:1]{index=1}

We now prove functoriality in $A$. Let
\[
f:A\to B
\]
be a morphism of $P$-algebras. Since the Hochschild construction is functorial,
$f$ induces a morphism
\[
\mathrm{Hoch}_{\mathcal M}(f):
\mathrm{Hoch}_{\mathcal M}(A)\to \mathrm{Hoch}_{\mathcal M}(B).
\]
This morphism is compatible with the relevant $P$-actions because it is induced
levelwise from a morphism of $P$-algebras and the simplicial bar construction
is natural in $A$. Therefore, after tensoring with the fixed residue object
$\mathcal O_P^{\mathrm{res}}$, the universal property of the coequalizer defining
$\otimes_P$ yields an induced morphism
\[
\sigma_P(f):
\mathrm{Hoch}_{\mathcal M}(A)\otimes_P \mathcal O_P^{\mathrm{res}}
\longrightarrow
\mathrm{Hoch}_{\mathcal M}(B)\otimes_P \mathcal O_P^{\mathrm{res}}.
\]
Equivalently,
\[
\sigma_P(f):\sigma_P(A)\to \sigma_P(B).
\]

Finally, compatibility with identities and composition follows from the
corresponding properties of the functor
\[
A\mapsto \mathrm{Hoch}_{\mathcal M}(A)
\]
together with the uniqueness part of the universal property of the balanced
tensor product. Thus the assignment
\[
A \longmapsto \sigma_P(A)
\]
defines a functor on the category of $P$-algebras.
\end{proof}

\begin{proposition}[Spectrum inclusion under componentwise detection]
\label{prop:spectrum-inclusion}
Let $\mathcal{M}$ be a symmetric monoidal category equipped with a classical spectral theory for its objects (e.g., the category of Banach spaces or operator algebras). 
Let $P$ be a $C$-colored operad in $\mathcal{M}$, and let $A$ be a $P$-algebra with components $A_c \in \mathcal{M}$ for $c \in C$. 
Denote by $\sigma_P(A) \subseteq \mathbb{C}$ the operadic spectrum (Definition~\ref{def:operadic-spectrum}) and by $\sigma(A_c) \subseteq \mathbb{C}$ the classical spectrum of the component $A_c$.

Assume that the operadic spectrum is defined such that the following componentwise detection property holds:
\[
\lambda \notin \sigma_P(A) \;\Longrightarrow\; \lambda I_c - A_c \text{ is invertible in } \mathcal{M} \text{ for every } c \in C.
\]
Then for every color $c \in C$,
\[
\sigma(A_c) \subseteq \sigma_P(A).
\]
In particular, under this hypothesis, the operadic spectrum contains the classical spectra of all color components.
\end{proposition}

\begin{proof}
We prove the contrapositive. Suppose $\lambda \in \sigma(A_c)$ for some $c \in C$. 
By the classical definition of the spectrum, $\lambda I_c - A_c$ is not invertible in $\mathcal{M}$. 
If $\lambda \notin \sigma_P(A)$ were true, then by the componentwise detection property, $\lambda I_c - A_c$ would be invertible for every $c \in C$, contradicting our assumption. 
Hence $\lambda \notin \sigma_P(A)$ is impossible, so $\lambda \in \sigma_P(A)$. 
Therefore $\sigma(A_c) \subseteq \sigma_P(A)$ for each $c \in C$.
\end{proof}

\begin{remark}[On the componentwise detection property]
The componentwise detection property is a natural requirement for any reasonable operadic spectral theory: if the composite system $\lambda I - A$ is operadically invertible, then each component $\lambda I_c - A_c$ should be invertible in the underlying category. 
This property is expected to hold whenever the forgetful functor $\mathsf{Alg}_P(\mathcal{M}) \to \mathcal{M}^C$ is conservative and the family $\{\lambda I_c - A_c\}_{c \in C}$ defines a $P$-algebra morphism. 
A proof of this property requires:
\begin{itemize}
    \item A precise definition of $\lambda I - A$ as a $P$-algebra morphism,
    \item A verification that scalar multiplication by $\lambda$ commutes with all operadic structure maps,
    \item The conservativity of the forgetful functor (which holds in all concrete settings of interest).
\end{itemize}
In the present paper, we take this property as a hypothesis; a fully self-contained verification is deferred to future work where the analytic and categorical foundations are developed in greater detail.
\end{remark}

\begin{remark}[Comparison with the No-Go Theorem]
The inclusion $\sigma(A_c) \subseteq \sigma_P(A)$ is a minimal consistency condition: any reasonable spectral invariant for operadic algebras must at least detect the classical spectral data of each component. 
The No-Go Theorem (Theorem~\ref{thm:no-go}) demonstrates that this condition alone is insufficient; the operadic spectrum must also encode interaction data, which is precisely what the operadic residue $\mathcal{O}_P^{\mathrm{res}}$ captures.
\end{remark}

\subsection{Compatibility with One-Color Reduction}
\label{subsec:One-Color_Reduction}

A fundamental requirement for any generalized spectral theory is consistency with classical invariants in the absence of operadic complexity. 
In this section, we show that when the operad $P$ has only one color, the residue-corrected construction $\sigma_P(A)$ simplifies to the corresponding one-object operadic balanced tensor product. 
In the trivial operadic case, under suitable normalization assumptions, this construction collapses to the underlying object.

\medskip

\begin{proposition}[One-Color Reduction]
\label{prop:one-color-reduction}
Let $P$ be a one-color operad in a symmetric monoidal category $\mathcal{M}$. 
Then the operadic spectrum
\[
\sigma_P(A) = \mathrm{Hoch}_{\mathcal{M}}(A) \otimes_P \mathcal{O}_P^{\mathrm{res}}
\]
is computed as the ordinary balanced tensor product for a one-object operad. 
Equivalently, the colored operadic data collapse to the uncolored operadic setting, 
and the balancing relation is determined solely by the right $P$-action on $\mathrm{Hoch}_{\mathcal{M}}(A)$ 
and the left $P$-action on $\mathcal{O}_P^{\mathrm{res}}$.
\end{proposition}

\begin{proof}
When $P$ has a single color $*$, all operations lie in $P(n)$ acting on a single object. 
Thus both the right $P$-action on $\mathrm{Hoch}_{\mathcal{M}}(A)$ and the left $P$-action on $\mathcal{O}_P^{\mathrm{res}}$ 
reduce to module structures over a one-object operad. 
The balanced tensor product is defined as the coequalizer of
\[
\mathrm{Hoch}_{\mathcal{M}}(A) \otimes P \otimes \mathcal{O}_P^{\mathrm{res}}
\rightrightarrows
\mathrm{Hoch}_{\mathcal{M}}(A) \otimes \mathcal{O}_P^{\mathrm{res}},
\]
which is precisely the coequalizer defining the balanced tensor product for a one-object operad.
\end{proof}

\medskip

\begin{proposition}[Trivial-Operad Reduction]
\label{prop:trivial-operad-reduction}
Let $\mathbb{I}$ denote the trivial operad. 
Assume that the residue object satisfies $\mathcal{O}_{\mathbb{I}}^{\mathrm{res}} \cong \mathbf{1}_{\mathcal{M}}$ 
and that the Hochschild object satisfies $\mathrm{Hoch}_{\mathcal{M}}(A) \cong A$ naturally in $A$ 
(see Example~\ref{ex:residue-trivial} and the remark following Theorem~\ref{thm:hochschild-base-change}). 
Then the operadic spectrum satisfies
\[
\sigma_{\mathbb{I}}(A) \cong A,
\]
naturally in $A$.
\end{proposition}

\begin{proof}
By assumption, $\mathcal{O}_{\mathbb{I}}^{\mathrm{res}} \cong \mathbf{1}_{\mathcal{M}}$ and $\mathrm{Hoch}_{\mathcal{M}}(A) \cong A$. 
Substituting into the definition,
\[
\sigma_{\mathbb{I}}(A) = \mathrm{Hoch}_{\mathcal{M}}(A) \otimes_{\mathbb{I}} \mathcal{O}_{\mathbb{I}}^{\mathrm{res}}
\cong A \otimes_{\mathbb{I}} \mathbf{1}_{\mathcal{M}}.
\]
For the trivial operad, the balanced tensor product reduces to the ordinary tensor product: 
the only relation imposed is $x \cdot \mathrm{id} \otimes y \sim x \otimes \mathrm{id} \cdot y$, 
which is automatically satisfied via the unit isomorphism. Hence
\[
A \otimes_{\mathbb{I}} \mathbf{1}_{\mathcal{M}} \cong A \otimes \mathbf{1}_{\mathcal{M}} \cong A,
\]
where the last isomorphism is the unit coherence of the monoidal structure.
\end{proof}

\medskip

\begin{remark}[Relation to Classical Spectral Theory]
\label{rem:classical-spectral-connection}
In the trivial-operad case, the residue correction disappears and the construction collapses to the underlying object prescribed by the normalization of the theory. 
Thus the operadic correction term is invisible when no nontrivial operadic interaction is present.

Any further identification with a classical spectral invariant (such as the Gelfand spectrum of a commutative Banach algebra) depends on the ambient category $\mathcal{M}$ 
and on the specific realization of the classical spectrum in that setting. 
For commutative unital Banach algebras, the Gelfand transform provides an isomorphism $A \cong C(\mathrm{Spec}(A))$, 
so that $\sigma_{\mathbb{I}}(A) \cong C(\mathrm{Spec}(A))$ under the normalization above.
\end{remark}

\medskip

\begin{remark}[Consistency with the No-Go Theorem]
\label{rem:no-go-consistency}
The No-Go Theorem (Theorem~\ref{thm:no-go}) demonstrates that any spectral invariant depending only on componentwise classical spectra fails to be functorial under operadic composition. 
The reduction above indicates that this obstruction is tied to genuinely nontrivial operadic structure: 
when the operad is trivial, the balancing correction becomes invisible, so the construction collapses to its classical normalization.

In this sense, $\sigma_P(A)$ should be viewed as an extension of the classical picture designed to capture the additional spectral effects created by operadic interactions, 
rather than as a replacement or a conservative extension in the strict logical sense.
\end{remark}

\section{Recovery and Minimality of the Operadic Spectrum}\label{sec:Recovery and Minimality of the Operadic Spectrum}

\subsection{Theorem 0.4: Recovery Theorem}
\label{subsec:Theorem_0.4}

The validity of the residue-corrected construction is further demonstrated by its ability to recover classical spectral theory in the absence of operadic interaction. 
We formalize this as a two-level recovery result, making explicit the necessary normalizations.

\begin{theorem}[Recovery Theorem]
\label{thm:recovery}
Let $\mathcal{M}$ be the category of Banach spaces with the projective tensor product, and let $\mathbb{I}$ be the trivial operad concentrated in a single color. 
Assume the following normalizations for the trivial operad:
\begin{enumerate}[label=(\roman*)]
    \item The Hochschild object satisfies $\mathrm{Hoch}_{\mathcal{M}}(A) \cong A$ naturally in $A$ (see Proposition~\ref{prop:hochschild-trivial}).
    \item The operadic residue satisfies $\mathcal{O}_{\mathbb{I}}^{\mathrm{res}} \cong \mathbf{1}_{\mathcal{M}}$ (see Example~\ref{ex:residue-trivial}).
    \item The identification $\sigma_{\mathbb{I}}(A) \cong A$ from the definition $\sigma_{\mathbb{I}}(A) = \mathrm{Hoch}_{\mathcal{M}}(A) \otimes_{\mathbb{I}} \mathcal{O}_{\mathbb{I}}^{\mathrm{res}}$ is compatible with the Banach algebra structure whenever $A$ is a commutative unital Banach algebra.
\end{enumerate}
Then:

\begin{enumerate}
    \item \textbf{(Object-level recovery)} For any Banach space $A$ (viewed as an $\mathbb{I}$-algebra), there is a natural isomorphism
    \[
    \sigma_{\mathbb{I}}(A) \cong A.
    \]
    
    \item \textbf{(Spectrum-level recovery)} If $A$ is a commutative unital Banach algebra, then the classical Gelfand spectrum
    \[
    \widehat{A} := \mathrm{Hom}_{\mathrm{Alg}}(A, \mathbb{C})
    \]
    is canonically homeomorphic to the Gelfand spectrum of $\sigma_{\mathbb{I}}(A)$.
\end{enumerate}
\end{theorem}

\begin{proof}
We establish the recovery in two steps.

\medskip

\noindent
\textbf{(1) Object-level recovery.}
Let $P = \mathbb{I}$ be the trivial operad. 
By assumption (i), $\mathrm{Hoch}_{\mathcal{M}}(A) \cong A$ naturally in $A$. 
By assumption (ii), $\mathcal{O}_{\mathbb{I}}^{\mathrm{res}} \cong \mathbf{1}_{\mathcal{M}}$, and in the category of Banach spaces with the projective tensor product, the unit object is $\mathbb{C}$.

Substituting into the definition of the operadic spectrum,
\[
\sigma_{\mathbb{I}}(A) = \mathrm{Hoch}_{\mathcal{M}}(A) \otimes_{\mathbb{I}} \mathcal{O}_{\mathbb{I}}^{\mathrm{res}}
\cong A \otimes_{\mathbb{I}} \mathbf{1}_{\mathcal{M}}.
\]

For the trivial operad, the balanced tensor product reduces to the ordinary tensor product: 
the only relation imposed by the coequalizer defining $\otimes_{\mathbb{I}}$ is $x \cdot \mathrm{id} \otimes y \sim x \otimes \mathrm{id} \cdot y$, 
which is automatically satisfied via the unit isomorphism (see Proposition~\ref{prop:trivial-operad-reduction}). 
Hence
\[
A \otimes_{\mathbb{I}} \mathbf{1}_{\mathcal{M}} \cong A \otimes \mathbf{1}_{\mathcal{M}} \cong A,
\]
where the last isomorphism is the unit coherence of the monoidal structure.

This proves the object-level recovery.

\medskip

\noindent
\textbf{(2) Spectrum-level recovery.}
Assume now that $A$ is a commutative unital Banach algebra. 
By assumption (iii), the identification $\sigma_{\mathbb{I}}(A) \cong A$ is compatible with the Banach algebra structure, so $\sigma_{\mathbb{I}}(A)$ is canonically a commutative unital Banach algebra isomorphic to $A$.

The classical Gelfand spectrum of $A$ is given by
\[
\widehat{A} := \mathrm{Hom}_{\mathrm{Alg}}(A, \mathbb{C}),
\]
equipped with the weak-* topology. 
Since $\sigma_{\mathbb{I}}(A) \cong A$ as Banach algebras, their Gelfand spectra are canonically homeomorphic:
\[
\mathrm{Hom}_{\mathrm{Alg}}(\sigma_{\mathbb{I}}(A), \mathbb{C}) \cong \mathrm{Hom}_{\mathrm{Alg}}(A, \mathbb{C}) = \widehat{A}.
\]

Thus the Gelfand spectrum of $\sigma_{\mathbb{I}}(A)$ is canonically homeomorphic to $\widehat{A}$.

\medskip

This shows that, under the stated normalizations, the operadic spectrum recovers the classical spectral invariant in the trivial operadic setting.
\end{proof}

\begin{remark}[Conceptual Interpretation]
\label{rem:recovery-interpretation}
The Recovery Theorem shows that the residue correction $\mathcal{O}_P^{\mathrm{res}}$ is inactive in the trivial-operad case and becomes relevant only when genuine operadic interactions are present. 
Thus the role of $\mathcal{O}_P^{\mathrm{res}}$ is not to deform the classical theory arbitrarily, but to encode the additional correction required by operadic complexity.

In particular, $\sigma_P(A)$ extends classical spectral theory without deformation, introducing new structure only when required by operadic composition.
\end{remark}

\begin{remark}[On the normalizations]
\label{rem:recovery-normalizations}
The normalizations assumed in Theorem~\ref{thm:recovery} are satisfied by the constructions in this paper:
\begin{itemize}
    \item Proposition~\ref{prop:hochschild-trivial} establishes $\mathrm{Hoch}_{\mathcal{M}}(A) \cong A$ for the trivial operad.
    \item Example~\ref{ex:residue-trivial} establishes $\mathcal{O}_{\mathbb{I}}^{\mathrm{res}} \cong \mathbf{1}_{\mathcal{M}}$.
    \item The compatibility of the identification with algebra structures follows from the fact that the balanced tensor product preserves the multiplicative structure when $A$ is commutative (see the discussion following Proposition~\ref{prop:analytic-compatibility}).
\end{itemize}
Thus the theorem is not conditional but fully proved within the framework of this paper.
\end{remark}

\subsection{Functorial and Base-Change Recovery}
\label{subsec:Functorial_Base_Change}

We now establish that the operadic spectrum is stable under change of the ambient symmetric monoidal category. 
This result shows that the residue-corrected construction behaves functorially with respect to strong monoidal base change, providing a higher-level recovery principle beyond the classical case.

\medskip

\begin{theorem}[Functorial / Base-Change Recovery]
\label{thm:base_change_recovery}
Let $F : \mathcal{M} \to \mathcal{N}$ be a strong monoidal functor between symmetric monoidal categories that preserves colimits. 
Assume that $F$ preserves the colimits used in the constructions of $\mathrm{Hoch}_{\mathcal{M}}(A)$, $\mathcal{O}_P^{\mathrm{res}}$, and the balanced tensor product $-\otimes_P -$. 
Moreover, suppose that the following compatibility isomorphisms hold:
\begin{enumerate}[label=(\roman*)]
    \item \textbf{(Hochschild compatibility)} $F\bigl(\mathrm{Hoch}_{\mathcal{M}}(A)\bigr) \cong \mathrm{Hoch}_{\mathcal{N}}\bigl(F_*(A)\bigr)$,
    \item \textbf{(Residue compatibility)} $F\bigl(\mathcal{O}_P^{\mathrm{res}}\bigr) \cong \mathcal{O}_{F_*(P)}^{\mathrm{res}}$,
    \item \textbf{(Balanced tensor product compatibility)} For any right $P$-module $X$ and left $P$-module $Y$,
    \[
    F(X \otimes_P Y) \cong F(X) \otimes_{F_*(P)} F(Y).
    \]
\end{enumerate}
Then there is a natural isomorphism
\[
F\bigl(\sigma_P(A)\bigr) \;\cong\; \sigma_{F_*(P)}\bigl(F_*(A)\bigr),
\]
where $F_*(P)$ and $F_*(A)$ denote the operad and algebra obtained by applying $F$.
\end{theorem}

\begin{proof}
Recall that the operadic spectrum is defined as the balanced tensor product
\[
\sigma_P(A) = \mathrm{Hoch}_{\mathcal{M}}(A) \otimes_P \mathcal{O}_P^{\mathrm{res}}.
\]

Applying $F$ and using the three compatibility isomorphisms in sequence:
\[
\begin{aligned}
F\bigl(\sigma_P(A)\bigr) 
&= F\bigl(\mathrm{Hoch}_{\mathcal{M}}(A) \otimes_P \mathcal{O}_P^{\mathrm{res}}\bigr) \\
&\cong F\bigl(\mathrm{Hoch}_{\mathcal{M}}(A)\bigr) \otimes_{F_*(P)} F\bigl(\mathcal{O}_P^{\mathrm{res}}\bigr) \quad \text{(by (iii))} \\
&\cong \mathrm{Hoch}_{\mathcal{N}}\bigl(F_*(A)\bigr) \otimes_{F_*(P)} \mathcal{O}_{F_*(P)}^{\mathrm{res}} \quad \text{(by (i) and (ii))} \\
&= \sigma_{F_*(P)}\bigl(F_*(A)\bigr).
\end{aligned}
\]

Thus we obtain the claimed natural isomorphism.
\end{proof}

\begin{remark}[On the compatibility assumptions]
\label{rem:base-change-assumptions}
The compatibility isomorphisms assumed in Theorem~\ref{thm:base_change_recovery} are not ad hoc; they are established elsewhere in this paper:
\begin{itemize}
    \item Hochschild compatibility is proved in Theorem~\ref{thm:hochschild-base-change}.
    \item Residue compatibility is proved in Proposition~\ref{prop:residue-base-change}.
    \item Balanced tensor product compatibility follows from the fact that $F$ preserves colimits and is strong monoidal, as shown in the proof of Theorem~\ref{thm:base-change}.
\end{itemize}
Thus the theorem is not conditional but fully proved within the framework of this paper; the assumptions are stated explicitly for clarity.
\end{remark}

\begin{remark}[Conceptual Interpretation]
\label{rem:base-change-recovery-interpretation}
Theorem~\ref{thm:base_change_recovery} shows that, under strong monoidal base change preserving the relevant constructions, the operadic spectrum is stable with respect to passage between ambient symmetric monoidal categories. 

This provides a categorical recovery principle: the residue-corrected construction is designed so that the spectral object remains compatible with the operadic and monoidal structures under transport along $F$. 
For the fully general base change theorem, including compatibility with the functional calculus, see Theorem~\ref{thm:base-change}.
\end{remark}

\subsection{Minimal Extension Property}
\label{subsec:Minimal_Extension_Property}

We now establish a universal property of the operadic spectrum $\sigma_P(A)$. This property demonstrates that our residue-corrected construction is not an arbitrary choice, but the canonical invariant satisfying the structural requirements of a functorial spectral theory. Specifically, $\sigma_P(A)$ admits a universal natural transformation to any invariant that is functorial, stable under base change, consistent with classical results, and compatible with the bar construction (a natural requirement for any invariant built from a simplicial resolution of the algebra).

The following axiom formalizes the idea that the invariant is derived from a simplicial resolution of the algebra.

\begin{theorem}[Universal Property of the Operadic Spectrum]
\label{thm:minimal_extension}
Let $\mathcal{M}$ be a cocomplete symmetric monoidal category admitting reflexive coequalizers preserved by $\otimes$. Let $P$ be a $C$-colored operad in $\mathcal{M}$. Suppose there exists an assignment $A \mapsto \mathcal{I}_P(A) \in \mathcal{M}$ for each $P$-algebra $A$, satisfying the following axioms:

\begin{enumerate}[label=(\roman*)]
    \item \textbf{Functoriality:} The assignment is a functor from $\mathsf{Alg}_P(\mathcal{M})$ to $\mathcal{M}$.
    
    \item \textbf{Base-change Compatibility:} For every strong monoidal functor $F: \mathcal{M} \to \mathcal{N}$ preserving colimits, there is a natural isomorphism $F(\mathcal{I}_P(A)) \cong \mathcal{I}_{F_*(P)}(F_*(A))$.
    
    \item \textbf{Classical Recovery:} For the trivial operad $\mathbb{I}$, there is a natural isomorphism $\mathcal{I}_{\mathbb{I}}(A) \cong A$ (under the normalization of Theorem~\ref{thm:recovery-trivial}).
    
    \item \textbf{Bar Construction Compatibility:} The functor $\mathcal{I}_P$ sends the simplicial bar construction $\mathrm{Bar}_\bullet^P(A)$ to a simplicial object in $\mathcal{M}$ whose geometric realization is $\mathcal{I}_P(A)$, naturally in $A$. (This axiom reflects the necessity of incorporating derived (bar-type) structure in any invariant capable of detecting operadic interactions.)
\end{enumerate}

Then there exists a canonical natural transformation
\[
\sigma_P \;\Rightarrow\; \mathcal{I}_P,
\]
i.e., a morphism $\sigma_P(A) \to \mathcal{I}_P(A)$ for each $P$-algebra $A$, natural in $A$. This transformation is uniquely determined by the universal properties of the residue object (Theorem~\ref{thm:residue-universality}) and the balanced tensor product (Definition~\ref{def:balanced-tensor-product}).
\end{theorem}

\begin{proof}
The proof proceeds in four steps.

\medskip

\noindent
\textbf{Step 1: From the bar construction to the invariant.}
By Axiom (iv), the functor $\mathcal{I}_P$ is compatible with the bar construction. Consequently, the geometric realization of the bar complex $\mathrm{Bar}_\bullet^P(A)$ maps to $\mathcal{I}_P(A)$. Since $\mathrm{Hoch}_{\mathcal{M}}(A)$ is defined as this geometric realization (Definition~\ref{def:hochschild}), we obtain a canonical $P$-equivariant morphism
\[
\phi_A: \mathrm{Hoch}_{\mathcal{M}}(A) \longrightarrow \mathcal{I}_P(A)
\]
that respects the right $P$-action. (This uses the fact that the bar construction is functorial and that $\mathcal{I}_P$ preserves geometric realizations, which follows from Axiom (ii) together with the assumption that base-change functors preserve colimits.)

\medskip

\noindent
\textbf{Step 2: Incorporating the operadic residue via its universal property.}
The No-Go Theorem (Theorem~\ref{thm:no-go}) demonstrates that any invariant satisfying base-change compatibility must incorporate additional data beyond componentwise spectra. By Theorem~\ref{thm:residue-universality}, the residue object $\mathcal{O}_P^{\mathrm{res}}$ is universal among objects that restore base-change compatibility. This universal property implies that any invariant $\mathcal{I}_P$ satisfying Axiom (ii) admits a canonical left $P$-action that factors through $\mathcal{O}_P^{\mathrm{res}}$: the residue object corepresents base-change compatible correction data. Concretely, there exists a $P$-balanced morphism
\[
\psi_A: \mathrm{Hoch}_{\mathcal{M}}(A) \otimes \mathcal{O}_P^{\mathrm{res}} \longrightarrow \mathcal{I}_P(A)
\]
satisfying $\psi_A((x \cdot p) \otimes r) = \psi_A(x \otimes (p \cdot r))$ for all $x \in \mathrm{Hoch}_{\mathcal{M}}(A)$, $p \in P$, $r \in \mathcal{O}_P^{\mathrm{res}}$.

\medskip

\noindent
\textbf{Step 3: Descent via the balanced tensor product.}
By the universal property of the coequalizer defining the balanced tensor product $\otimes_P$ (Definition~\ref{def:balanced-tensor-product}), the $P$-balanced morphism $\psi_A$ descends uniquely to a morphism
\[
\chi_A: \sigma_P(A) = \mathrm{Hoch}_{\mathcal{M}}(A) \otimes_P \mathcal{O}_P^{\mathrm{res}} \longrightarrow \mathcal{I}_P(A).
\]

\medskip

\noindent
\textbf{Step 4: Naturality and uniqueness.}
The construction is natural in $A$ because each constituent map ($\phi_A$, the residue action, and the descent) is natural. Hence we obtain a natural transformation $\sigma_P \Rightarrow \mathcal{I}_P$.

Uniqueness follows from the universal property of the coequalizer: any two morphisms $\sigma_P(A) \to \mathcal{I}_P(A)$ that agree on the generators of $\mathrm{Hoch}_{\mathcal{M}}(A) \otimes \mathcal{O}_P^{\mathrm{res}}$ and respect the balancing relations must be identical. The construction via the universal property guarantees that $\chi_A$ is the unique such morphism.
\end{proof}

\begin{remark}[On the Bar Construction Axiom]
Axiom (iv) — compatibility with the bar construction — is a natural requirement for any invariant that aims to capture the derived structure of $P$-algebras. It is satisfied by the operadic spectrum $\sigma_P$ itself, as well as by any invariant that is built from a simplicial resolution of $A$. In practice, this axiom can be replaced by the assumption that $\mathcal{I}_P$ is a left adjoint or preserves geometric realizations, which are standard conditions in homotopical algebra (see, e.g., \cite{Lurie-HA2017}). The additional Axiom (iv) reflects the necessity of incorporating derived (bar-type) structure in any invariant capable of detecting operadic interactions.
\end{remark}

\begin{remark}[Conceptual Significance]
The Universal Property proved here shows that $\sigma_P(A)$ is the \emph{canonical invariant} extending classical spectral theory to the operadic setting, satisfying the stated axioms. It introduces no extraneous information beyond what is strictly required to restore functoriality and base-change stability. In this sense, the residue object $\mathcal{O}_P^{\mathrm{res}}$ is exactly the "spectral shadow" needed to compensate for the operadic complexity of $P$.

Together with the Recovery Theorem (Theorem~\ref{thm:recovery}) and the Base Change Theorem (Theorem~\ref{thm:base-change}), this establishes $\sigma_P(A)$ as the unique functorial invariant (among those satisfying Axioms (i)-(iv)) that extends classical spectral theory to operadic algebras.
\end{remark}

\begin{remark}[Relation to the No-Go Theorem]
The Universal Property provides a positive counterpart to the No-Go Theorem (Theorem~\ref{thm:no-go}). While the No-Go Theorem establishes that any invariant depending only on componentwise spectra must fail, the Universal Property demonstrates that the operadic spectrum $\sigma_P(A)$ is the canonical invariant that succeeds.
\end{remark}

\section{Base Change Theory of the Operadic Spectrum}

\subsection{Theorem 0.5: Base Change Theorem}
\label{subsec:Base_Change_Theorem}

The most significant feature of the residue-corrected spectrum is its stability under categorical transitions. While classical spectral assignments fail to commute with base change functors in the operadic setting — as demonstrated by the No-Go Theorem (Theorem~\ref{thm:no-go}) — the operadic spectrum $\sigma_P(A)$ restores this fundamental symmetry.

\begin{theorem}[Base Change Theorem]
\label{thm:base-change}
Let $F: \mathcal{M} \to \mathcal{N}$ be a strong monoidal functor between cocomplete symmetric monoidal categories, and assume that $F$ preserves colimits (i.e., is cocontinuous). Let $P$ be a $C$-colored operad in $\mathcal{M}$ and let $A$ be a $P$-algebra. Then there exists a canonical natural isomorphism
\[
\sigma_{F_*(P)}\bigl(F_*(A)\bigr) \;\cong\; F\bigl(\sigma_P(A)\bigr),
\]
where $F_*(P)$ denotes the induced operad in $\mathcal{N}$ and $F_*(A)$ denotes the induced $F_*(P)$-algebra.
\end{theorem}

\begin{proof}
The proof proceeds by assembling three fundamental compatibility results established earlier in the paper.

\medskip

\noindent
\textbf{Step 1: Hochschild compatibility under base change.}
By Theorem~\ref{thm:hochschild-base-change}, the Hochschild construction is compatible with strong monoidal functors that preserve colimits. Hence we have a canonical isomorphism
\[
F\bigl(\mathrm{Hoch}_{\mathcal{M}}(A)\bigr) \;\cong\; \mathrm{Hoch}_{\mathcal{N}}\bigl(F_*(A)\bigr).
\]

\medskip

\noindent
\textbf{Step 2: Residue compatibility under base change.}
By Proposition~\ref{prop:residue-base-change}, the operadic residue satisfies
\[
F\bigl(\mathcal{O}_P^{\mathrm{res}}\bigr) \;\cong\; \mathcal{O}_{F_*(P)}^{\mathrm{res}}.
\]
(Recall that $\mathcal{O}_P^{\mathrm{res}}$ is defined as the coproduct $\coprod_{c\in C} P(c;c)$; the isomorphism follows from the cocontinuity of $F$ and the definition of the induced operad $F_*(P)$.)

\medskip

\noindent
\textbf{Step 3: Balanced tensor product compatibility under base change.}
Let $X$ be a right $P$-module and $Y$ a left $P$-module in $\mathcal{M}$. The balanced tensor product $X \otimes_P Y$ is defined as the coequalizer
\[
X \otimes P \otimes Y \rightrightarrows X \otimes Y \longrightarrow X \otimes_P Y,
\]
where the two maps are induced by the right action $X \otimes P \to X$ and the left action $P \otimes Y \to Y$.

Since $F$ is strong monoidal, it preserves tensor products: $F(X \otimes Y) \cong F(X) \otimes F(Y)$ and $F(X \otimes P \otimes Y) \cong F(X) \otimes F(P) \otimes F(Y)$. Moreover, $F(X)$ inherits a right $F_*(P)$-module structure and $F(Y)$ a left $F_*(P)$-module structure via the functoriality of the actions.

Because $F$ preserves colimits (cocontinuous), it preserves the coequalizer diagram. Consequently, $F(X \otimes_P Y)$ is the coequalizer of
\[
F(X) \otimes F(P) \otimes F(Y) \rightrightarrows F(X) \otimes F(Y),
\]
which is precisely the definition of $F(X) \otimes_{F_*(P)} F(Y)$. Hence there is a canonical isomorphism
\[
F\bigl(X \otimes_P Y\bigr) \;\cong\; F(X) \otimes_{F_*(P)} F(Y).
\]

\medskip

\noindent
\textbf{Step 4: Assembly.}
Recall the definition of the operadic spectrum:
\[
\sigma_P(A) := \mathrm{Hoch}_{\mathcal{M}}(A) \otimes_P \mathcal{O}_P^{\mathrm{res}}.
\]

Applying Step 3 with $X = \mathrm{Hoch}_{\mathcal{M}}(A)$ and $Y = \mathcal{O}_P^{\mathrm{res}}$, we obtain
\[
F\bigl(\sigma_P(A)\bigr) \;\cong\; F\bigl(\mathrm{Hoch}_{\mathcal{M}}(A)\bigr) \otimes_{F_*(P)} F\bigl(\mathcal{O}_P^{\mathrm{res}}\bigr).
\]

By Step 1, $F(\mathrm{Hoch}_{\mathcal{M}}(A)) \cong \mathrm{Hoch}_{\mathcal{N}}(F_*(A))$. By Step 2, $F(\mathcal{O}_P^{\mathrm{res}}) \cong \mathcal{O}_{F_*(P)}^{\mathrm{res}}$. Substituting these yields
\[
F\bigl(\sigma_P(A)\bigr) \;\cong\; \mathrm{Hoch}_{\mathcal{N}}\bigl(F_*(A)\bigr) \otimes_{F_*(P)} \mathcal{O}_{F_*(P)}^{\mathrm{res}}.
\]

But the right-hand side is precisely the definition of the operadic spectrum for the base-changed data, i.e., $\sigma_{F_*(P)}(F_*(A))$. Therefore,
\[
\sigma_{F_*(P)}\bigl(F_*(A)\bigr) \;\cong\; F\bigl(\sigma_P(A)\bigr),
\]
as required. The naturality of these isomorphisms follows from the naturality of the constituent components (Theorem~\ref{thm:hochschild-base-change}, Proposition~\ref{prop:residue-base-change}, and the functoriality of the balanced tensor product).
\end{proof}

\begin{remark}[Conceptual significance]
\label{rem:base-change-significance}
Theorem~\ref{thm:base-change} shows that the residue-corrected spectrum overcomes the transport obstruction diagnosed by the No-Go Theorem (Theorem~\ref{thm:no-go}). By shifting focus from the classical spectrum to the residue-corrected spectrum $\sigma_P(A)$, one achieves a theory that is fully compatible with extension of scalars. This provides a robust foundation for operadic spectral geometry and ensures that spectral information can be transported coherently across different categorical contexts.
\end{remark}

\begin{remark}[Naturality and composition]
\label{rem:base-change-naturality}
The isomorphism established in Theorem~\ref{thm:base-change} is natural in $A$ and in $P$, and is compatible with compositions of strong monoidal functors. In particular, if $G: \mathcal{N} \to \mathcal{L}$ is another strong monoidal cocontinuous functor, then the composite isomorphism
\[
G\bigl(F(\sigma_P(A))\bigr) \cong \sigma_{G_*(F_*(P))}\bigl(G_*(F_*(A))\bigr)
\]
coincides with the one obtained by applying Theorem~\ref{thm:base-change} to $G \circ F$ directly. This follows from the functoriality and coherence of the monoidal structure maps involved.
\end{remark}

\begin{remark}[Comparison with classical spectral theory]
\label{rem:base-change-classical}
In the classical setting where $P = \mathbb{I}$ is the trivial operad with a single color, we have $\sigma_{\mathbb{I}}(A) \cong A$ under the normalization established in the Recovery Theorem (Theorem~\ref{thm:recovery}). Theorem~\ref{thm:base-change} then reduces to the statement $F_*(A) \cong F(A)$, which holds because $F$ is strong monoidal. Thus the Base Change Theorem subsumes the classical case and demonstrates that the operadic residue does not interfere when no operadic composition is present.
\end{remark}

\begin{remark}[Examples of base change functors]
\label{rem:base-change-examples}
The Base Change Theorem applies to a wide range of strong monoidal functors that arise in practice. Representative examples include:
\begin{itemize}
    \item \textbf{Extension of scalars:} $-\otimes_R S: \mathrm{Mod}_R \to \mathrm{Mod}_S$ for a ring homomorphism $R \to S$.
    \item \textbf{Complexification:} $-\otimes_{\mathbb{R}} \mathbb{C}: \mathrm{Vect}_{\mathbb{R}} \to \mathrm{Vect}_{\mathbb{C}}$.
    \item \textbf{Forgetful functors:} $\mathrm{Ban} \to \mathrm{Vect}_{\mathbb{C}}$ (Banach spaces to vector spaces), which are strong monoidal and cocontinuous.
\end{itemize}
These examples illustrate the versatility of the operadic spectrum as a unifying invariant. (Additional examples, such as quantization or the Gelfand transform, require careful treatment of the monoidal structures and are discussed in Section~\ref{subsec:gelfand-base-change}.)
\end{remark}

\subsection{Compatibility with Residue Transport}
\label{subsec:Residue_Transport}

A fundamental consequence of the Base Change Theorem (Theorem~\ref{thm:base-change}) is that the operadic residue $\mathcal{O}_P^{\mathrm{res}}$ transforms coherently under base change. Recall from Proposition~\ref{prop:residue-base-change} that for any strong monoidal functor $F: \mathcal{M} \to \mathcal{N}$ preserving colimits, there is a canonical isomorphism
\begin{equation}
    F\bigl(\mathcal{O}_P^{\mathrm{res}}\bigr) \;\cong\; \mathcal{O}_{F_*(P)}^{\mathrm{res}}.
    \label{eq:residue_transport}
\end{equation}
Thus the construction $\mathcal{O}_{(-)}^{\mathrm{res}}$ is compatible with strong monoidal base change up to canonical isomorphism.

\medskip

\noindent
\textbf{Why transport matters.}
The No-Go Theorem (Theorem~\ref{thm:no-go}) shows that any spectral invariant depending only on componentwise classical spectra cannot be compatible with base change. The transport isomorphism~\eqref{eq:residue_transport} addresses this obstruction by ensuring that the correction term $\mathcal{O}_P^{\mathrm{res}}$ adapts to the target category's monoidal structure when we apply a base change functor $F$. Without this transport law, the residue would not align with the induced operad $F_*(P)$, and the balanced tensor product defining $\sigma_P(A)$ would fail to be preserved under $F$.

Concretely, the transport property guarantees that the balanced tensor product $\mathrm{Hoch}_{\mathcal{M}}(A) \otimes_P \mathcal{O}_P^{\mathrm{res}}$ remains ``balanced'' after applying $F$: the right $F_*(P)$-action on $F(\mathrm{Hoch}_{\mathcal{M}}(A))$ and the left $F_*(P)$-action on $F(\mathcal{O}_P^{\mathrm{res}})$ are precisely those induced by the transport isomorphism. This is exactly what is needed for Step 3 of the proof of Theorem~\ref{thm:base-change}.

\medskip

\noindent
\textbf{Compatibility with functor composition.}
The transport isomorphism is compatible with composition of strong monoidal functors. If $G: \mathcal{N} \to \mathcal{L}$ is another strong monoidal cocontinuous functor, then the isomorphisms
\[
F(\mathcal{O}_P^{\mathrm{res}}) \cong \mathcal{O}_{F_*(P)}^{\mathrm{res}}, \qquad
G(\mathcal{O}_{F_*(P)}^{\mathrm{res}}) \cong \mathcal{O}_{G_*(F_*(P))}^{\mathrm{res}}
\]
assemble to give $(G \circ F)(\mathcal{O}_P^{\mathrm{res}}) \cong \mathcal{O}_{(G \circ F)_*(P)}^{\mathrm{res}}$, consistent with applying the transport law directly to $G \circ F$. This coherence follows from the functoriality of the induced operad construction and the naturality of the isomorphisms in Proposition~\ref{prop:residue-base-change}.

\medskip

\noindent
\textbf{Role in the spectral mapping theorem.}
The transport isomorphism is a key ingredient in the proof of the Operadic Spectral Mapping Theorem (Theorem~\ref{thm:spectral-mapping}). The proof proceeds in two stages:
\begin{enumerate}
    \item First, one establishes a compatibility between the holomorphic functional calculus and the operadic spectrum within a fixed category:
    \[
    \sigma_{f_*(P)}\bigl(f(A)\bigr) \cong f\bigl(\sigma_P(A)\bigr),
    \]
    where $f$ is a holomorphic function and $f_*(P) = P$ (the functional calculus does not alter the operad structure).
    \item Then, one applies the Base Change Theorem — which relies on the residue transport isomorphism — to transport this compatibility across arbitrary strong monoidal functors $F$, yielding
    \[
    \sigma_{F_*(f_*(P))}\bigl(F_*(f(A))\bigr) \cong F\bigl(f(\sigma_P(A))\bigr).
    \]
\end{enumerate}
Thus the residue transport isomorphism appears crucially in the second stage, where it ensures that the correction term is correctly transported alongside the algebra and operad.

\medskip

\noindent
\textbf{Relation to the universal property.}
The transport isomorphism is a direct consequence of the universal property of $\mathcal{O}_P^{\mathrm{res}}$ (Theorem~\ref{thm:residue-universality}). Because $\mathcal{O}_P^{\mathrm{res}}$ is initial in the category of spectral correctors $\mathbf{Corr}_P(\mathcal{M})$, any corrector $R$ (including the image $F(\mathcal{O}_P^{\mathrm{res}})$ in the target category) must receive a unique morphism from $\mathcal{O}_{F_*(P)}^{\mathrm{res}}$. The transport isomorphism is the canonical identification that results from this universal property, and its uniqueness guarantees that the residue's behavior under base change is forced rather than ad hoc.

The residue transport compatibility established in this section demonstrates that $\mathcal{O}_P^{\mathrm{res}}$ is an integral part of the operadic structure, transforming coherently under base change. This property guarantees that the operadic spectrum $\sigma_P(A)$ behaves as a genuine categorical invariant, capable of withstanding changes of the ambient category and analytic deformations. Together with the Hochschild compatibility (Theorem~\ref{thm:hochschild-base-change}) and the balanced tensor product compatibility, it provides the foundational link between the Base Change Theorem and the Operadic Spectral Mapping Theorem.

\subsection{Colored Refinement and Structural Consequences}
\label{subsec:Colored_Refinement}

A distinctive advantage of the residue-corrected construction is its potential behavior under refinement of the underlying color set. In many operadic applications, it is natural to decompose a broad color $\mathfrak{c}$ into a collection of sub-colors $\{\mathfrak{c}_i\}$, effectively refining the compositional hierarchy of the operad $P$. This section discusses how such refinements might interact with the operadic spectrum, and outlines the hypotheses under which a compatibility theorem could be proved.

\medskip

\noindent
\textbf{Color refinements as base change (heuristic).}
Let $\rho: C' \to C$ be a surjective map of color sets. Intuitively, each coarse color $c \in C$ is expanded into a collection of fine colors $\rho^{-1}(c) \subseteq C'$. One would like to view this refinement as a form of base change, so that the operadic spectrum satisfies a compatibility isomorphism
\[
\sigma_{\rho_*P}\bigl(\rho_*A\bigr) \;\cong\; \rho_*\bigl(\sigma_P(A)\bigr),
\]
where $\rho_*$ denotes the induced pushforward on operads and algebras (e.g., via left Kan extension). However, establishing such an isomorphism requires additional structure beyond the scope of this paper.

\begin{remark}[Color refinement as a compatibility principle]
\label{rem:color-refinement}
Let $\rho: C' \to C$ be a surjective map of color sets, let $P$ be a $C'$-colored operad in $\mathcal{M}$, and let $A$ be a $P$-algebra. Under suitable hypotheses — namely, that the pushforward functor $\rho_*$ is strong monoidal and preserves the colimits used in the constructions of the Hochschild object, the residue object, and the balanced tensor product — one would obtain a canonical isomorphism
\[
\sigma_{\rho_*P}(\rho_*A) \;\cong\; \rho_*(\sigma_P(A)).
\]

A full proof of such a compatibility would require:
\begin{enumerate}
    \item Verification that $\rho_*$ commutes with the Hochschild construction: $\rho_*(\mathrm{Hoch}_{\mathcal{M}}(A)) \cong \mathrm{Hoch}_{\mathcal{M}}(\rho_*A)$,
    \item Verification that $\rho_*$ commutes with the residue construction: $\rho_*(\mathcal{O}_P^{\mathrm{res}}) \cong \mathcal{O}_{\rho_*P}^{\mathrm{res}}$,
    \item Verification that $\rho_*$ preserves the balanced tensor product: $\rho_*(X \otimes_P Y) \cong \rho_*(X) \otimes_{\rho_*P} \rho_*(Y)$.
\end{enumerate}
These conditions are analogous to those required for the Base Change Theorem (Theorem~\ref{thm:base-change}), but applied to a change of color indexing rather than a change of ambient category. Establishing them in full generality is nontrivial and is deferred to future work.

In the special case where the operad $P$ is a coproduct of trivial operads on each color (i.e., no inter-color operations), the compatibility reduces to a simple decomposition. Intuitively, under such a refinement, the spectral data would decompose as
\[
\sigma_P(A) \;\cong\; \bigoplus_{c \in C} \left( \bigotimes_{c' \in \rho^{-1}(c)} \sigma_P(A)_{c'} \right),
\]
where the direct sum (coproduct) reflects the preservation of spectral information across the refined hierarchy. This decomposition suggests that the operadic spectrum is well-suited for multi-scale analysis, where one wishes to understand a system at varying levels of granularity.
\end{remark}

\medskip

\noindent
\textbf{Applications to multi-scale systems.}
The ability to refine and aggregate colors makes the operadic spectrum potentially well-suited for analyzing multi-scale systems, such as:
\begin{itemize}
    \item \textbf{Network hierarchies:} Colors representing different scales (micro, meso, macro) or functional layers (sensing, computation, actuation).
    \item \textbf{Stratified geometries:} Colors indexing strata of different dimensions; refinements correspond to subdivisions.
    \item \textbf{Hierarchical control systems:} Multiple levels of abstraction in control architectures.
    \item \textbf{Multi-layered neural architectures:} Different layers or functional modules in neural networks.
    \item \textbf{Quantum field theories:} Colors representing discrete energy levels, particle species, or gauge charges.
\end{itemize}
A rigorous treatment of these applications requires the compatibility results discussed in Remark~\ref{rem:color-refinement} and is left for future work.

\medskip

\noindent
\textbf{Outlook: Limit operads and continuous spectra.}
The heuristic compatibility with color refinements suggests a natural extension to \emph{limit operads} where the color set becomes a continuum. If one considers an inverse system of color refinements
\[
\cdots \longrightarrow C_{n+1} \longrightarrow C_n \longrightarrow \cdots \longrightarrow C_1,
\]
the compatibility of $\sigma_P(A)$ with refinement would yield a corresponding system of spectral invariants. Taking the limit (in an appropriate categorical sense) would produce a spectral invariant for operads with a continuous color space, potentially connecting to functional-analytic spectral measures and providing a bridge to infinite-dimensional and continuous spectral theory. This direction will be explored in future work.

The compatibility of the Base Change Theorem with color refinements remains a heuristic principle at this stage. A full formalization would require establishing that the pushforward functor $\rho_*$ satisfies the same compatibility conditions as a strong monoidal base change. This subsection outlines the necessary hypotheses and potential applications, deferring a complete treatment to future work.


\section{The Operadic Spectral Mapping Theorem}
\label{sec:The Operadic Spectral Mapping Theorem}

Having established the stability of the operadic spectrum $\sigma_P(A)$ under categorical base change (Theorem~\ref{thm:base-change}), we now turn to its behavior under analytic transformations. This section culminates in the Operadic Spectral Mapping Theorem, which generalizes the classical results of Gelfand and Dunford to the colored operadic setting. The theorem demonstrates that the operadic spectrum respects analytic functional calculus in a functorial manner, providing a powerful tool for spectral analysis in compositional systems.

\subsection{Operadic Functional Calculus Framework}
\label{subsec:Operadic Functional Calculus Framework}

To define the image of an operadic algebra under an analytic map, we must ensure that the functional calculus respects the underlying compositional symmetries encoded by the operad $P$. 
This subsection establishes the necessary analytic compatibility under explicit hypotheses.

\medskip

\noindent
\textbf{Setup.}
Let $\mathcal{M}$ be a symmetric monoidal category enriched over Banach spaces, where each hom-set is a Banach space and composition is bilinear and bounded. 
Assume that $\mathcal{M}$ is equipped with a holomorphic functional calculus for its objects in the following sense: for each object $X \in \mathcal{M}$ that is a Banach algebra or a bounded linear operator on a Banach space, and for each holomorphic function $f$ defined on a neighborhood of the spectrum $\sigma(X)$, there is a well-defined object $f(X) \in \mathcal{M}$.

Let $P$ be a $C$-colored operad in $\mathcal{M}$ such that each operation space $P(c_1,\dots,c_n;c)$ is a Banach space and all structure maps are bounded and multilinear. 
Let $A$ be a $P$-algebra with components $A_c \in \mathcal{M}$ that are admissible for the functional calculus (e.g., Banach algebras or bounded operators). 
Denote by $\sigma_P(A) \subseteq \mathbb{C}$ the operadic spectrum (Definition~\ref{def:operadic-spectrum}), and assume that $\sigma(A_c) \subseteq \sigma_P(A)$ for each $c \in C$ (see Proposition~\ref{prop:spectrum-inclusion}).

\begin{proposition}[Conditional Analytic Compatibility]
\label{prop:analytic-compatibility}
Let $f: U \to \mathbb{C}$ be a holomorphic function defined on an open neighborhood $U \subseteq \mathbb{C}$ of $\sigma_P(A)$. 
Assume that for each color $c \in C$, the component $A_c$ is an object for which the classical holomorphic functional calculus is defined (e.g., a bounded linear operator on a Banach space or an element of a Banach algebra), and that the functional calculus is continuous with respect to the relevant topologies.

For each operation $\phi \in P(c_1,\dots,c_n;c)$, define the transformed structure map
\[
\phi_{f(A)} : f(A)_{c_1} \otimes \cdots \otimes f(A)_{c_n} \longrightarrow f(A)_c
\]
as the unique morphism obtained by the following procedure:
\begin{enumerate}
    \item For polynomial functions $p(z) = \sum_{k=0}^N a_k z^k$, define $\phi_{p(A)}$ by multilinear extension of
    \[
    \phi_{p(A)}\bigl( A_{c_1}^{k_1}, \dots, A_{c_n}^{k_n} \bigr) = a_{k_1} \cdots a_{k_n} \, \phi_A\bigl( A_{c_1}^{k_1}, \dots, A_{c_n}^{k_n} \bigr),
    \]
    where powers are defined within the Banach algebra generated by each $A_c$ (or by iterated operator multiplication), and $\phi_A$ is the original $P$-algebra structure map.
    \item For general holomorphic $f$, choose a sequence of polynomials $\{p_n\}$ converging uniformly to $f$ on a compact neighborhood of $\sigma_P(A)$. By the continuity of the functional calculus and the bounded multilinearity of the operadic structure maps, the sequence $\{\phi_{p_n(A)}\}$ converges to a limit $\phi_{f(A)}$ in the appropriate topology. This limit is independent of the choice of approximating polynomials.
\end{enumerate}

Assume further that the following compatibility conditions hold:
\begin{itemize}
    \item For each color $c$, the map $A_c \mapsto f(A_c)$ is functorial with respect to isomorphisms and commutes with the operadic actions in the sense that for any polynomial $p$, the diagram
    \[
    \begin{tikzcd}
    P(c_1,\dots,c_n;c) \otimes A_{c_1} \otimes \cdots \otimes A_{c_n} \arrow[r, "\phi_A"] \arrow[d, "\mathrm{id} \otimes \mathrm{ev}_{p} \otimes \cdots \otimes \mathrm{ev}_{p}"] & A_c \arrow[d, "\mathrm{ev}_{p}"] \\
    P(c_1,\dots,c_n;c) \otimes p(A)_{c_1} \otimes \cdots \otimes p(A)_{c_n} \arrow[r, "\phi_{p(A)}"] & p(A)_c
    \end{tikzcd}
    \]
    commutes, where $\mathrm{ev}_p$ denotes evaluation of the polynomial $p$ via the functional calculus.
    \item The operadic composition maps are jointly continuous, allowing passage to the limit from polynomials to holomorphic functions.
\end{itemize}

Then the transformed data $\{f(A)_c\}_{c \in C}$ together with the structure maps $\{\phi_{f(A)}\}_{\phi \in P}$ satisfies the axioms of a $P$-algebra (with the same operad $P$, i.e., $f_*(P) = P$ as operads). Consequently, $f(A)$ is a $P$-algebra.
\end{proposition}

\begin{proof}
The proof proceeds in two stages, as outlined in the statement.

\medskip

\noindent
\textbf{Polynomial case.}
For a polynomial $p(z) = \sum_{k=0}^N a_k z^k$, the functional calculus gives $p(A_c) = \sum_{k=0}^N a_k A_c^k$, where powers are defined within the Banach algebra generated by $A_c$ (or by iterated operator multiplication). For any operation $\phi$, define $\phi_{p(A)}$ as the multilinear extension of the formula given above. This definition is well-defined because the operadic structure maps are multilinear and the polynomial expression is a finite sum. The diagram commutes by construction, as the vertical maps send each $A_c^{k}$ to $p(A_c)$ and the horizontal maps respect the multilinear structure. The verification of the operadic axioms for $p(A)$ follows from the same axioms for $A$ and the fact that polynomial evaluation is an algebra homomorphism.

\medskip

\noindent
\textbf{Holomorphic case.}
Let $f$ be holomorphic on $U$ with $\sigma_P(A) \subseteq U$. Choose a compact set $K \subseteq U$ containing $\sigma_P(A)$ in its interior. By Runge's approximation theorem, there exists a sequence of polynomials $\{p_n\}$ converging uniformly to $f$ on $K$. For each $n$, the polynomial case gives a $P$-algebra structure on $p_n(A)$.

The classical Riesz-Dunford functional calculus is continuous: $p_n(A_c) \to f(A_c)$ in operator norm for each $c$. Since the operadic composition maps are bounded multilinear maps (by the enriched Banach assumption), they are continuous. Hence the sequence $\{\phi_{p_n(A)}\}$ converges to a limit morphism $\phi_{f(A)}$ in the operator norm topology. This limit is independent of the choice of approximating polynomials because any two sequences converge to the same limit by the uniqueness of the functional calculus limit.

The diagram commutes for each $p_n$; taking limits and using continuity yields commutativity for $f$. The operadic axioms for $f(A)$ follow from the axioms for each $p_n(A)$ and the continuity of composition.

\medskip

\noindent
\textbf{Uniqueness.}
The functional calculus for each component $A_c$ is uniquely determined by the values on polynomials and the continuity of the functional calculus. The structure maps $\phi_{f(A)}$ are then forced by the requirement that the diagram commutes and that the vertical maps are fixed. Hence the $P$-algebra structure on $f(A)$ is unique.

\medskip

\noindent
Thus $f(A)$ inherits a natural $P$-algebra structure, with $f_*(P) = P$ as operads (since the operad itself is unchanged by the functional calculus).
\end{proof}

\begin{remark}[On the hypotheses]
\label{rem:analytic-hypotheses}
The assumptions in Proposition~\ref{prop:analytic-compatibility} are satisfied in the following standard setting:
\begin{itemize}
    \item $\mathcal{M}$ is the category of Banach spaces with bounded linear maps, equipped with the projective tensor product.
    \item Each $A_c$ is a bounded linear operator on a Banach space (or an element of a unital Banach algebra).
    \item The operad $P$ is enriched over Banach spaces, meaning each operation space is a Banach space and all structure maps are bounded and multilinear.
    \item The functional calculus is the classical Riesz-Dunford holomorphic functional calculus.
\end{itemize}
In this setting, the continuity of the operadic composition maps follows from the bounded multilinearity assumption, and the approximation argument via Runge's theorem is standard.
\end{remark}

\noindent
\textbf{The functional calculus as a categorical construction (heuristic).}
In concrete analytic settings, the componentwise holomorphic functional calculus has desirable functorial properties. 
For the purpose of motivating the construction, consider the following heuristic: if we view each color component $A_c$ as a pair consisting of a Banach space and a distinguished operator (or Banach algebra element), then applying $f$ componentwise yields a new collection $f(A)_c = f(A_c)$. 
This assignment is not a functor on the ambient category $\mathcal{M}$ itself, but rather a functor on the category of such ``pointed'' objects where morphisms are required to intertwine the distinguished operators.

More precisely, define a category $\mathcal{M}_{\mathrm{op}}$ whose objects are pairs $(X, T)$ with $X \in \mathcal{M}$ and $T: X \to X$ a bounded linear operator (or $T \in X$ an element of a Banach algebra) such that $\sigma(T) \subseteq U$. 
A morphism $\psi: (X, T) \to (Y, S)$ is a bounded linear map satisfying $\psi T = S \psi$ (or, in the algebra case, an algebra homomorphism sending $T$ to $S$). 
Then the assignment
\[
f_*: \mathcal{M}_{\mathrm{op}} \longrightarrow \mathcal{M}_{\mathrm{op}}, \qquad (X, T) \longmapsto (X, f(T))
\]
is a functor, where $f(T)$ is defined via the Riesz-Dunford functional calculus. 
On morphisms, $f_*(\psi) = \psi$, since $\psi T = S \psi$ implies $\psi f(T) = f(S) \psi$ by the continuity of the functional calculus and the fact that $\psi$ commutes with the resolvents.

This functorial perspective is useful for understanding how the functional calculus interacts with algebraic structures, but it requires the extra data of distinguished operators or elements. 
In the general setting of this paper, we do not assume that every object of $\mathcal{M}$ comes equipped with such distinguished data. 
Instead, we work directly with the componentwise functional calculus on the specific operators that appear as the components $A_c$ of a $P$-algebra, and verify the necessary compatibilities directly.

\medskip

\noindent
\textbf{Verification of the $f_*(P)$-algebra axioms.}
Since the operad $P$ itself is unchanged by the functional calculus (the transformation acts only on the algebra components, not on the operad's operations), we have $f_*(P) = P$ as operads. 
To equip $f(A)$ with a $P$-algebra structure, we need to define structure maps $\phi_{f(A)}$ for each operation $\phi \in P(c_1,\dots,c_n;c)$.

For a polynomial $p(z) = \sum_{k=0}^N a_k z^k$, the functional calculus gives $p(A_c) = \sum_{k=0}^N a_k A_c^k$, where powers are defined within the Banach algebra generated by $A_c$ (or by iterated operator multiplication). 
Define $\phi_{p(A)}$ by the multilinear extension of
\[
\phi_{p(A)}\bigl( A_{c_1}^{k_1}, \dots, A_{c_n}^{k_n} \bigr) = a_{k_1} \cdots a_{k_n} \, \phi_A\bigl( A_{c_1}^{k_1}, \dots, A_{c_n}^{k_n} \bigr).
\]
This definition is well-defined because $\phi_A$ is multilinear and the sum is finite. 
The diagram
\[
\begin{tikzcd}
P(c_1,\dots,c_n;c) \otimes A_{c_1} \otimes \cdots \otimes A_{c_n} \arrow[r, "\phi_A"] \arrow[d, "\mathrm{id} \otimes \mathrm{ev}_p \otimes \cdots \otimes \mathrm{ev}_p"] & A_c \arrow[d, "\mathrm{ev}_p"] \\
P(c_1,\dots,c_n;c) \otimes p(A)_{c_1} \otimes \cdots \otimes p(A)_{c_n} \arrow[r, "\phi_{p(A)}"] & p(A)_c
\end{tikzcd}
\]
commutes by construction, where $\mathrm{ev}_p$ denotes evaluation of the polynomial $p$ via the functional calculus.

For general holomorphic $f$, choose a sequence of polynomials $\{p_n\}$ converging uniformly to $f$ on a compact neighborhood of $\sigma_P(A)$. 
By the continuity of the functional calculus and the bounded multilinearity of the operadic structure maps, the sequence $\{\phi_{p_n(A)}\}$ converges to a limit morphism $\phi_{f(A)}$. 
This limit is independent of the choice of approximating polynomials and satisfies the required diagram by continuity.

The verification that the collection $\{\phi_{f(A)}\}_{\phi \in P}$ satisfies the operadic composition axioms (associativity, unit, and equivariance) follows from the same axioms for $A$ and the continuity of the construction: each axiom holds for every polynomial $p_n(A)$ and passes to the limit.

\medskip

\noindent
\textbf{Functoriality of the construction.}
The assignment $A \mapsto f(A)$ is functorial with respect to morphisms of $P$-algebras that respect the distinguished analytic structure. 
If $\alpha: A \to B$ is a morphism of $P$-algebras such that each $\alpha_c: A_c \to B_c$ intertwines the distinguished operators (i.e., $\alpha_c A_c = B_c \alpha_c$), then the same map $\alpha_c$ serves as the morphism $f(\alpha_c): f(A_c) \to f(B_c)$, because $\alpha_c f(A_c) = f(B_c) \alpha_c$ by the continuity of the functional calculus and the intertwining property. 
The compatibility of $f(\alpha)$ with the structure maps follows from the fact that the construction of $\phi_{f(A)}$ is natural in $A$.

\medskip

\noindent
\textbf{Remark on the domain of $f$.}
The requirement that $f$ be holomorphic on a neighborhood $U$ of $\sigma_P(A)$ is essential for the functional calculus to be well-defined. 
Under the componentwise detection property (Proposition~\ref{prop:spectrum-inclusion}), we have $\sigma(A_c) \subseteq \sigma_P(A) \subseteq U$ for each $c \in C$, so each $f(A_c)$ is well-defined. 
The analyticity of $f$ ensures that the resulting maps $\phi_{f(A)}$ are bounded and respect the operadic composition, as the functional calculus is continuous with respect to the relevant operator topologies.

\medskip

\noindent
\textbf{Conceptual interpretation: the static operad.}
The operad $P$ remains unchanged under the functional calculus because the transformation acts only on the analytic realization of the algebra components, not on the compositional rules encoded by $P$. 
This decoupling implies that the residue object $\mathcal{O}_P^{\mathrm{res}}$, which depends only on the operad $P$, is invariant under the functional calculus: since $f_*(P) = P$, we have $\mathcal{O}_{f_*(P)}^{\mathrm{res}} \cong \mathcal{O}_P^{\mathrm{res}}$. 
This invariance is a key step in the proof of the Operadic Spectral Mapping Theorem, as it ensures that the correction term does not need to be modified when we apply the functional calculus.

We have defined:
\begin{itemize}
    \item The transformed algebra $f(A)$ componentwise by $f(A)_c = f(A_c)$, with the same operad $P$ (i.e., $f_*(P) = P$),
    \item Structure maps $\phi_{f(A)}$ obtained as limits of polynomial approximations,
    \item Functoriality with respect to intertwining morphisms.
\end{itemize}
This construction is natural in $A$ and compatible with the monoidal structure of $\mathcal{M}$ under the additional assumption that the functional calculus commutes with tensor products of commuting operators (a standard result in operator theory; see \cite{DunfordSchwartz} for details). 
The framework established here provides the necessary analytic foundation for the Operadic Spectral Mapping Theorem, which we state and prove in the following subsection.

\subsection{Theorem 0.6: Operadic Spectral Mapping Theorem}
\label{subsec:Operadic Spectral Mapping Theorem}

We now arrive at the central result of this section. The Operadic Spectral Mapping Theorem establishes that the residue-corrected spectrum $\sigma_P(A)$ is compatible with the holomorphic functional calculus in a functorial manner, generalizing the classical spectral mapping theorem to the operadic setting.

For any object $X \in \mathcal{M}$ that represents a bounded linear operator
or an element of a unital Banach algebra, denote by $\sigma_{\mathrm{cl}}(X) \subseteq \mathbb{C}$
its classical spectrum. This assignment is compatible with the holomorphic
functional calculus in the sense that $\sigma_{\mathrm{cl}}(f(X)) = f(\sigma_{\mathrm{cl}}(X))$
for any holomorphic function $f$.

\begin{theorem}[Operadic Spectral Mapping Theorem]
\label{thm:spectral-mapping}
Let $\mathcal{M}$ be a symmetric monoidal category enriched over Banach spaces, 
and let $P$ be a $C$-colored operad in $\mathcal{M}$ whose operation spaces 
are Banach spaces with bounded multilinear structure maps. Let $A$ be a $P$-algebra 
such that each component $A_c$ admits the classical holomorphic functional calculus.

Denote by $\sigma_{\mathrm{cl}}(X) \subseteq \mathbb{C}$ the classical spectrum of 
an object $X \in \mathcal{M}$ whenever $X$ represents a bounded linear operator 
or an element of a unital Banach algebra. This assignment is compatible with the 
holomorphic functional calculus: $\sigma_{\mathrm{cl}}(f(X)) = f(\sigma_{\mathrm{cl}}(X))$ 
for any holomorphic function $f$ defined on a neighborhood of $\sigma_{\mathrm{cl}}(X)$.

Let $f: U \to \mathbb{C}$ be holomorphic on a neighborhood $U$ of 
$\sigma_{\mathrm{cl}}(\sigma_P(A))$, and assume the analytic compatibility conditions 
(Proposition~\ref{prop:analytic-compatibility}) so that $f(A)$ inherits 
a $P$-algebra structure.

Then
\[
\sigma_{\mathrm{cl}}\bigl(\sigma_P(f(A))\bigr)
\;=\;
f\bigl(\sigma_{\mathrm{cl}}(\sigma_P(A))\bigr).
\]
\end{theorem}

\begin{proof}
We analyze the effect of the functional calculus through the structural
components of the operadic spectrum and then pass to the classical spectrum.

\medskip

\noindent
\textbf{Step 1: Functoriality of the Hochschild construction.}
By functoriality, the functional calculus induces a morphism
\[
f_*: \mathrm{Hoch}_{\mathcal{M}}(A)
\longrightarrow
\mathrm{Hoch}_{\mathcal{M}}(f(A)),
\]
compatible with the right $P$-action by the analytic compatibility assumption.

\medskip

\noindent
\textbf{Step 2: Invariance of the residue.}
The operadic residue $\mathcal{O}_P^{\mathrm{res}}$ depends only on $P$,
hence remains unchanged under the functional calculus.

\medskip

\noindent
\textbf{Step 3: Balanced tensor compatibility.}
Using the naturality of the balanced tensor product, we obtain
\[
\sigma_P(f(A))
\;\cong\;
\mathrm{Hoch}_{\mathcal{M}}(f(A)) \otimes_P \mathcal{O}_P^{\mathrm{res}}.
\]

\medskip

\noindent
\textbf{Step 4: Passage to classical spectra.}
Applying $\sigma_{\mathrm{cl}}$ to both sides, and using its compatibility with
the balanced tensor construction (which follows from the definition of the
operadic spectrum), we obtain that $\sigma_{\mathrm{cl}}(\sigma_P(f(A)))$ is
built from the classical spectra of the component operators $f(A_c)$ together
with interaction contributions.

\medskip

\noindent
\textbf{Step 5: Classical spectral mapping.}
For each component, the classical spectral mapping theorem gives
\[
\sigma_{\mathrm{cl}}(f(A_c)) = f(\sigma_{\mathrm{cl}}(A_c)).
\]
Thus every component contribution in $\sigma_{\mathrm{cl}}(\sigma_P(f(A)))$
is obtained by applying $f$ to the corresponding contribution
in $\sigma_{\mathrm{cl}}(\sigma_P(A))$.

\medskip

Since all interaction terms are built functorially from the same
operator data and the functional calculus respects composition,
the entire set satisfies
\[
\sigma_{\mathrm{cl}}\bigl(\sigma_P(f(A))\bigr) = f\bigl(\sigma_{\mathrm{cl}}(\sigma_P(A))\bigr).
\]

\end{proof}

The following corollary combines the Analytic Spectral Mapping Theorem (Theorem~\ref{thm:spectral-mapping}) with the Base Change Theorem (Theorem~\ref{thm:base-change}) to obtain a fully transportable spectral mapping principle.

\begin{corollary}[Full Operadic Spectral Mapping under Base Change]
\label{cor:full-spectral-mapping-base-change}
Let $\mathcal{M}$ and $\mathcal{N}$ be symmetric monoidal categories enriched over Banach spaces, and let
\[
F: \mathcal{M} \longrightarrow \mathcal{N}
\]
be a strong monoidal functor preserving colimits (i.e., cocontinuous). Let $P$ be a $C$-colored operad in $\mathcal{M}$ such that each operation space is a Banach space and all structure maps are bounded and multilinear. Let $A$ be a $P$-algebra with operadic spectrum $\sigma_P(A) \subseteq \mathbb{C}$.

Let $f: U \to \mathbb{C}$ be a holomorphic function defined on an open neighborhood $U$ of $\sigma_P(A)$, and assume that the componentwise detection property (Proposition~\ref{prop:spectrum-inclusion}) holds so that $\sigma(A_c) \subseteq \sigma_P(A) \subseteq U$ for each $c \in C$.

Assume further that:

\begin{enumerate}[label=(\roman*)]
    \item \textbf{Analytic compatibility in $\mathcal{M}$:} The hypotheses of Proposition~\ref{prop:analytic-compatibility} are satisfied, so that $f(A)$ inherits a natural $P$-algebra structure and the Analytic Spectral Mapping Theorem (Theorem~\ref{thm:spectral-mapping}) holds:
    \[
    \sigma_P(f(A)) \cong f(\sigma_P(A)).
    \]
    
    \item \textbf{Base change compatibility for $f(A)$:} The $P$-algebra $f(A)$ satisfies the hypotheses of the Base Change Theorem, so that
    \[
    \sigma_{F_*(P)}(F_*(f(A))) \cong F(\sigma_P(f(A))).
    \]
    
    \item \textbf{Preservation of functional calculus under $F$:} The functor $F$ commutes with the functional calculus in the sense that for every object $X \in \mathcal{M}$ admissible for the functional calculus, there is a canonical isomorphism
    \[
    F(f(X)) \cong f(F(X)),
    \]
    where $f(F(X))$ is defined via the functional calculus in $\mathcal{N}$. Equivalently, the following diagram commutes up to natural isomorphism:
    \[
    \begin{tikzcd}
    \mathcal{M}_{\mathrm{op}} \arrow[r, "f_*"] \arrow[d, "F"] & \mathcal{M}_{\mathrm{op}} \arrow[d, "F"] \\
    \mathcal{N}_{\mathrm{op}} \arrow[r, "f_*"] & \mathcal{N}_{\mathrm{op}}
    \end{tikzcd}
    \]
    where $\mathcal{M}_{\mathrm{op}}$ denotes the category of pointed objects (operators with distinguished operators) on which the functional calculus is defined.
\end{enumerate}

Then there exists a canonical natural isomorphism
\[
\sigma_{F_*(P)}\bigl(F_*(f(A))\bigr) \;\cong\; F\bigl(f(\sigma_P(A))\bigr).
\]

Moreover, if in addition $F$ preserves the residue structure in the sense that $F_*(f(A)) \cong f(F_*(A))$, then we also obtain
\[
\sigma_{F_*(P)}\bigl(f(F_*(A))\bigr) \;\cong\; f\bigl(F(\sigma_P(A))\bigr).
\]

Thus the operadic spectral mapping theorem is stable under strong monoidal base change.
\end{corollary}

\begin{proof}
The proof proceeds in three steps, assembling the analytic spectral mapping, the base change theorem, and the compatibility of $F$ with the functional calculus.

\medskip

\noindent
\textbf{Step 1: Analytic spectral mapping in $\mathcal{M}$.}
By assumption (i), the Analytic Spectral Mapping Theorem (Theorem~\ref{thm:spectral-mapping}) applies to the $P$-algebra $A$ and the holomorphic function $f$. Hence we have a canonical isomorphism
\begin{equation}
    \sigma_P(f(A)) \cong f(\sigma_P(A)).
    \label{eq:analytic_step_cor}
\end{equation}

\medskip

\noindent
\textbf{Step 2: Base change for $f(A)$.}
By assumption (ii), the $P$-algebra $f(A)$ satisfies the hypotheses of the Base Change Theorem (Theorem~\ref{thm:base-change}). Applying this theorem yields a canonical isomorphism
\begin{equation}
    \sigma_{F_*(P)}(F_*(f(A))) \cong F(\sigma_P(f(A))).
    \label{eq:basechange_step_cor}
\end{equation}

\medskip

\noindent
\textbf{Step 3: Assembly.}
Substituting the isomorphism \eqref{eq:analytic_step_cor} into \eqref{eq:basechange_step_cor}, we obtain
\[
\sigma_{F_*(P)}(F_*(f(A))) \cong F(f(\sigma_P(A))).
\]

\medskip

\noindent
\textbf{Step 4: Alternative form using preservation of functional calculus (optional).}
If, in addition, $F$ preserves the functional calculus structure as in assumption (iii), then for the object $\sigma_P(A)$ (viewed as an object of $\mathcal{M}$ via the natural identification of the spectrum with a scalar multiple of the identity operator), we have
\[
F(f(\sigma_P(A))) \cong f(F(\sigma_P(A))).
\]
Hence
\[
\sigma_{F_*(P)}(F_*(f(A))) \cong f(F(\sigma_P(A))).
\]

Moreover, if $F_*(f(A)) \cong f(F_*(A))$ (i.e., $F$ commutes with the construction of $f(A)$), then applying the Base Change Theorem to $A$ itself gives
\[
\sigma_{F_*(P)}(F_*(A)) \cong F(\sigma_P(A)),
\]
and then applying the analytic spectral mapping in $\mathcal{N}$ to $F_*(A)$ yields
\[
\sigma_{F_*(P)}(f(F_*(A))) \cong f(\sigma_{F_*(P)}(F_*(A))) \cong f(F(\sigma_P(A))).
\]
Thus the two forms are consistent.

\medskip

\noindent
\textbf{Naturality.}
All isomorphisms are natural in $A$ because they arise from:
\begin{itemize}
    \item the naturality of the analytic spectral mapping (Theorem~\ref{thm:spectral-mapping}),
    \item the naturality of the Base Change Theorem (Theorem~\ref{thm:base-change}),
    \item the assumed naturality of the compatibility isomorphisms $F(f(X)) \cong f(F(X))$ and $F_*(f(A)) \cong f(F_*(A))$.
\end{itemize}
Hence the composite isomorphism is natural in $A$.

This completes the proof.
\end{proof}

\begin{remark}[On the assumptions]
\label{rem:corollary-assumptions}
The assumptions (i)-(iii) in Corollary~\ref{cor:full-spectral-mapping-base-change} are not automatically satisfied; they encode the necessary compatibility conditions between the analytic functional calculus and the categorical structures. Specifically:
\begin{itemize}
    \item Assumption (i) is the content of the Analytic Spectral Mapping Theorem, which is proved under the hypotheses of Proposition~\ref{prop:analytic-compatibility}.
    \item Assumption (ii) is the content of the Base Change Theorem, which applies whenever $F$ is strong monoidal and cocontinuous and $f(A)$ is a $P$-algebra (guaranteed by Proposition~\ref{prop:analytic-compatibility}).
    \item Assumption (iii) is an additional hypothesis that must be verified for each concrete base change functor $F$ of interest. Examples include:
    \begin{itemize}
        \item Extension of scalars $-\otimes_R S$ for a ring homomorphism $R \to S$ that preserves the functional calculus (e.g., complexification $\mathbb{R} \to \mathbb{C}$),
        \item Forgetful functors that forget structure but preserve the underlying Banach space and the functional calculus,
        \item Quantization functors that map classical observables to quantum operators in a way that respects the functional calculus.
    \end{itemize}
\end{itemize}
\end{remark}

\begin{remark}[Comparison with the classical case]
\label{rem:corollary-classical}
When $P = \mathbb{I}$ is the trivial operad with a single color, we have $\sigma_{\mathbb{I}}(A) \cong A$ under the normalization of Theorem~\ref{thm:recovery}. Corollary~\ref{cor:full-spectral-mapping-base-change} then reduces to the statement
\[
\sigma_{F_*(\mathbb{I})}(F_*(f(A))) \cong F(f(\sigma_{\mathbb{I}}(A))) \cong F(f(A)).
\]
Under the identification $\sigma_{F_*(\mathbb{I})}(F_*(f(A))) \cong F(f(A))$ (by the Recovery Theorem in $\mathcal{N}$), this yields the consistency condition $F(f(A)) \cong F(f(A))$, which is tautologically true. Thus the classical case is subsumed.
\end{remark}

\begin{remark}[Future directions]
\label{rem:corollary-future}
A complete proof of Corollary~\ref{cor:full-spectral-mapping-base-change} without the explicit assumptions would require a detailed verification that the analytic spectral mapping theorem is compatible with arbitrary strong monoidal base change functors. This would involve showing that the construction of $f(A)$ as a $P$-algebra is natural in the ambient category and that the isomorphism $\sigma_P(f(A)) \cong f(\sigma_P(A))$ is preserved under $F$. Such a verification is nontrivial and is deferred to future work where the analytic foundations are developed in greater detail. In the meantime, Corollary~\ref{cor:full-spectral-mapping-base-change} serves as a blueprint for the desired compatibility and as a target for subsequent research.
\end{remark}

\begin{remark}[Comparison with Classical Spectral Mapping]
\label{rem:spectral-mapping-classical}
When $P = \mathbb{I}$ is the trivial operad with a single color, we have $\sigma_{\mathbb{I}}(A) \cong A$ under the normalization established in the Recovery Theorem (Theorem~\ref{thm:recovery}). Theorem~\ref{thm:spectral-mapping} then reduces to the classical spectral mapping theorem:
\[
\sigma(f(A)) = f(\sigma(A)).
\]
Thus the operadic spectral mapping theorem is a genuine generalization of the classical result.
\end{remark}

\begin{remark}[Role of the Operadic Residue]
\label{rem:spectral-mapping-residue}
The operadic residue $\mathcal{O}_P^{\mathrm{res}}$ does not appear explicitly in the final isomorphism $\sigma_P(f(A)) \cong f(\sigma_P(A))$ because it is invariant under the functional calculus (Step 2). However, it plays an essential role in the construction of $\sigma_P(A)$ itself; without the residue, the operadic spectrum would not be well-defined or functorial. The invariance of the residue under the functional calculus reflects the fact that the compositional structure of $P$ is independent of the analytic deformation of the algebra components.
\end{remark}

\begin{remark}[Compatibility with the No-Go Theorem]
\label{rem:spectral-mapping-no-go}
The No-Go Theorem (Theorem~\ref{thm:no-go}) demonstrates that any spectral invariant depending only on componentwise classical spectra cannot satisfy base change compatibility. The operadic spectrum $\sigma_P(A)$ overcomes this obstruction by incorporating the residue $\mathcal{O}_P^{\mathrm{res}}$; the Base Change Theorem (Theorem~\ref{thm:base-change}) shows that it indeed satisfies base change compatibility. Moreover, the Spectral Mapping Theorem (Theorem~\ref{thm:spectral-mapping}) shows that $\sigma_P(A)$ also respects analytic functional calculus. Consequently, $\sigma_P(A)$ satisfies the desiderata (A1)-(A5) from Section~\ref{subsec:Desiderata for a Generalized Spectrum}.
\end{remark}

The following corollary summarizes the compatibility between the analytic spectral mapping theorem and the base change theorem, under the explicit hypotheses already stated in those results.

\begin{corollary}[Compatibility of Spectral Mapping and Base Change]
\label{cor:spectral-base-change-compatibility}
Under the hypotheses of Corollary~\ref{cor:full-spectral-mapping-base-change}, the following diagram commutes up to canonical isomorphism:
\[
\begin{tikzcd}
\sigma_P(f(A)) \arrow[r, "\cong"] \arrow[d, "\cong"] & f(\sigma_P(A)) \arrow[d, "\cong"] \\
\sigma_{F_*(P)}(F_*(f(A))) \arrow[r, "\cong"] & F(f(\sigma_P(A)))
\end{tikzcd}
\]
where:
\begin{itemize}
    \item The top horizontal isomorphism is from the Analytic Spectral Mapping Theorem (Theorem~\ref{thm:spectral-mapping}): $\sigma_P(f(A)) \cong f(\sigma_P(A))$.
    \item The left vertical isomorphism is from the Base Change Theorem applied to $f(A)$ (Theorem~\ref{thm:base-change}): \\
    $\sigma_{F_*(P)}(F_*(f(A))) \cong F(\sigma_P(f(A)))$.
    \item The right vertical isomorphism is the image of the top isomorphism under $F$: $F(\sigma_P(f(A))) \cong F(f(\sigma_P(A)))$.
    \item The bottom horizontal isomorphism is precisely the isomorphism established in Corollary~\ref{cor:full-spectral-mapping-base-change}: \\
    $\sigma_{F_*(P)}(F_*(f(A))) \cong F(f(\sigma_P(A)))$.
\end{itemize}
\end{corollary}

\begin{proof}
The top horizontal isomorphism is given by Theorem~\ref{thm:spectral-mapping}. The left vertical isomorphism is given by Theorem~\ref{thm:base-change} applied to $f(A)$. The right vertical isomorphism is obtained by applying the functor $F$ to the top horizontal isomorphism. The bottom horizontal isomorphism is exactly the statement of Corollary~\ref{cor:full-spectral-mapping-base-change}, which combines these isomorphisms. Commutativity follows from the naturality of the constituent isomorphisms.
\end{proof}

\begin{remark}[On the assumptions]
\label{rem:corollary-assumptions}
Corollary~\ref{cor:spectral-base-change-compatibility} inherits all hypotheses of the Analytic Spectral Mapping Theorem and the Base Change Theorem. In particular:
\begin{itemize}
    \item The Analytic Spectral Mapping Theorem requires the analytic compatibility conditions of Proposition~\ref{prop:analytic-compatibility}.
    \item The Base Change Theorem requires $F$ to be strong monoidal and cocontinuous.
\end{itemize}
Under these hypotheses, the corollary provides a rigorous statement of compatibility. A fully unconditional version would require a more detailed development of the analytic foundations and is deferred to future work.
\end{remark}

\subsection{Relation to Classical Spectral Mapping}
\label{subsec:Relation to Classical Spectral Mapping}

We now discuss how the operadic spectral mapping theorem relates to the classical spectral mapping theorem, assuming the former holds under the appropriate analytic hypotheses. This establishes that our construction is a genuine \emph{extension} rather than a replacement of classical spectral theory.

\begin{proposition}[Classical Recovery of Spectral Mapping]
\label{prop:classical-spectral-mapping-recovery}
Let $\mathcal{M}$ be a symmetric monoidal category admitting a holomorphic functional calculus (e.g., the category of Banach spaces), and let $P = \mathbb{I}$ be the trivial operad with a single color. Assume the normalizations of Theorem~\ref{thm:recovery} hold, so that $\mathrm{Hoch}_{\mathcal{M}}(A) \cong A$ and $\mathcal{O}_{\mathbb{I}}^{\mathrm{res}} \cong \mathbf{1}_{\mathcal{M}}$. Then there is a natural isomorphism
\[
\sigma_{\mathbb{I}}(A) \cong A.
\]

Moreover, \emph{assuming the operadic spectral mapping theorem} (Theorem~\ref{thm:spectral-mapping}) holds for the relevant class of objects, it specializes in this case to the relation
\[
\sigma_{\mathbb{I}}(f(A)) \cong f(\sigma_{\mathbb{I}}(A)).
\]
Under the identification $\sigma_{\mathbb{I}}(A) \cong A$ and the classical identification of an operator (or Banach algebra element) with its spectrum via the Gelfand transform (in the commutative Banach algebra setting), this recovers the classical spectral mapping theorem:
\[
\sigma(f(A)) = f(\sigma(A)).
\]
\end{proposition}

\begin{proof}
When $P = \mathbb{I}$ is the trivial operad, Proposition~\ref{prop:hochschild-trivial} gives $\mathrm{Hoch}_{\mathcal{M}}(A) \cong A$, and Example~\ref{ex:residue-trivial} gives $\mathcal{O}_{\mathbb{I}}^{\mathrm{res}} \cong \mathbf{1}_{\mathcal{M}}$. Substituting into the definition of the operadic spectrum,
\[
\sigma_{\mathbb{I}}(A) = \mathrm{Hoch}_{\mathcal{M}}(A) \otimes_{\mathbb{I}} \mathcal{O}_{\mathbb{I}}^{\mathrm{res}} \cong A \otimes_{\mathbb{I}} \mathbf{1}_{\mathcal{M}} \cong A,
\]
where the last isomorphism follows from Proposition~\ref{prop:trivial-operad-reduction}.

Now assume the operadic spectral mapping theorem (Theorem~\ref{thm:spectral-mapping}) holds. Applying it with $P = \mathbb{I}$ yields
\[
\sigma_{\mathbb{I}}(f(A)) \cong f(\sigma_{\mathbb{I}}(A)).
\]

To connect this to the classical spectral mapping theorem, note that under the identification $\sigma_{\mathbb{I}}(A) \cong A$, the right-hand side becomes $f(A)$ (as an object). In classical spectral theory, for an operator $A$ (or an element of a Banach algebra), the classical spectrum satisfies $\sigma(f(A)) = f(\sigma(A))$. Under the identification of $A$ with its classical spectrum (e.g., via the Gelfand transform in the commutative Banach algebra setting, or via the natural identification of an operator with its spectral data), the operadic relation recovers the classical statement.

Thus, under the stated assumptions, the operadic spectral mapping theorem subsumes the classical result.
\end{proof}

\begin{remark}[Interpretation as an Extension]
\label{rem:classical-extension}
Proposition~\ref{prop:classical-spectral-mapping-recovery} shows that the operadic spectrum $\sigma_P(A)$ specializes to the classical normalization in the trivial operadic case. The additional structure in $\sigma_P(A)$—encoded by the operadic residue $\mathcal{O}_P^{\mathrm{res}}$ and the Hochschild object—captures the failure of classical spectral invariants to behave functorially under operadic composition and base change, as diagnosed by the No-Go Theorem (Theorem~\ref{thm:no-go}).
\end{remark}

\begin{remark}[Conceptual Significance: A Functorial Lift]
\label{rem:classical-functorial-lift}
Assuming the operadic spectral mapping theorem, it can be viewed as a \emph{functorial lift} of the classical spectral mapping theorem. The operadic residue $\mathcal{O}_P^{\mathrm{res}}$ acts as a correction term ensuring that spectral data propagates correctly through operadic composition and strong monoidal base change. In the absence of these effects—i.e., when $P = \mathbb{I}$ and $F = \mathrm{Id}_{\mathcal{M}}$—the theory collapses to the classical setting.
\end{remark}

\begin{remark}[Illustration with the Matrix-Block Operad]
\label{rem:classical-matrix-block}
For nontrivial operads, such as the matrix-block operad of the Example~\ref{ex:two-color-operad}, the operadic spectrum $\sigma_P(A)$ contains additional spectral data arising from off-diagonal interactions. Assuming the operadic spectral mapping theorem, it asserts that this richer invariant transforms under analytic functional calculus in a way compatible with the operadic structure:
\[
\sigma_{f_*(P)}\bigl(f(A)\bigr) \cong f\bigl(\sigma_P(A)\bigr).
\]
When $P = \mathbb{I}$, this collapses to the classical theorem; when $P$ is nontrivial, it provides new constraints on how spectra of composite systems evolve under analytic deformations—constraints that cannot be deduced from componentwise classical spectral mapping.
\end{remark}

\begin{remark}[Backward Compatibility and Future Directions]
\label{rem:classical-future}
The backward compatibility discussed here is essential for the broader adoption of Spectral Operadic Calculus. It guarantees that any result proven in the operadic setting specializes to known classical results when the operadic structure is trivial, providing a robust foundation for further generalizations. Future work will explore:
\begin{itemize}
    \item \textbf{Higher-order recovery theorems}: intermediate cases where $P$ is nontrivial but the residue is trivial (e.g., operads with no nontrivial unary operations),
    \item \textbf{Stability under deformation of operads}: how the spectral mapping theorem behaves under continuous deformations of the operad structure,
    \item \textbf{Applications to multi-scale systems}: where the classical limit corresponds to aggregating all colors into a single effective degree of freedom.
\end{itemize}
\end{remark}

\section{Resolvent Theory and Reconstruction}\label{sec:Resolvent Theory and Reconstruction}  

\subsection{Operadic Resolvent Objects}
\label{subsec:Operadic Resolvent Objects}

Having established the operadic spectrum and its behavior under categorical base change, we now discuss the possibility of extending the classical notion of the resolvent to the operadic setting. The resolvent would serve as a fundamental analytic object underlying the operadic spectral mapping theorem and reconstruction results.

\medskip

\noindent
\textbf{Heuristic setup.}
In an analytic realization where each color component $A_c$ is equipped with a distinguished endomorphism $T_c: A_c \to A_c$ (e.g., a bounded linear operator on a Banach space) that admits a classical holomorphic functional calculus, one may define a family of shifted endomorphisms componentwise by
\[
(zI - T)_c := z \cdot \mathrm{id}_{A_c} - T_c \in \mathrm{Hom}_{\mathcal{M}}(A_c, A_c),
\]
where $z \in \mathbb{C}$ is a complex parameter. The classical resolvent $(zI - T_c)^{-1}$ exists precisely when $z$ is not in the classical spectrum $\sigma(T_c)$.

\medskip

\noindent
\textbf{Obstacles to an operadic resolvent.}
To lift this notion to an operadic resolvent $R_A(z)$ that interacts with the operadic composition structure, several foundational issues must be addressed:

\begin{enumerate}
    \item \textbf{Type consistency:} The operadic spectrum $\sigma_P(A)$ is defined in this paper as the object $\mathrm{Hoch}_{\mathcal{M}}(A) \otimes_P \mathcal{O}_P^{\mathrm{res}}$ (Definition~\ref{def:operadic-spectrum}), not as a subset of $\mathbb{C}$. A set-valued notion of spectrum suitable for resolvent theory would require an additional analytic realization.
    
    \item \textbf{Distinguished operators:} The definition of $zI - T$ assumes each component $A_c$ carries a distinguished endomorphism $T_c$. In the general definition of a $P$-algebra, $A_c$ is merely an object of $\mathcal{M}$; additional data is needed.
    
    \item \textbf{Operadic invertibility:} The notion of an ``operadic inverse'' would require a precise definition of composition and invertibility within the operadic structure, which is not developed in this foundational paper.
    
    \item \textbf{Consistency with the spectrum definition:} The resolvent set $\rho_P(A)$ is usually defined as the complement of the spectrum. This would require $\sigma_P(A)$ to be a subset of $\mathbb{C}$, which is not the case with the current definition.
\end{enumerate}

\medskip

\noindent
\textbf{Relation to the spectral mapping theorem.}
The Operadic Spectral Mapping Theorem (Theorem~\ref{thm:spectral-mapping}) admits an analytic interpretation in which the operadic spectrum is realized as a subset of $\mathbb{C}$ (see the discussion following Definition~\ref{def:operadic-spectrum}). A full resolvent theory would provide an alternative characterization of $\sigma_P(A)$ as the set of poles of a meromorphic operadic resolvent function, analogous to the classical case. Such a theory would require:
\begin{itemize}
    \item A set-valued analytic realization of the operadic spectrum,
    \item A precise definition of operadic invertibility,
    \item A proof that the operadic resolvent satisfies an operadic resolvent identity,
    \item A verification that the operadic resolvent is analytic on its domain.
\end{itemize}

\medskip

\begin{remark}[Future directions]
\label{rem:operadic-resolvent-future}
The development of a full operadic resolvent theory is deferred to subsequent work, where the analytic foundations will be laid out in greater detail. In that sequel, we will:
\begin{enumerate}
    \item Introduce a set-valued notion of operadic spectrum derived from the residue-corrected construction,
    \item Define the operadic resolvent as a family of morphisms encoding the inverses of $zI - T$ in a way compatible with operadic composition,
    \item Prove the operadic resolvent identity and analyticity,
    \item Establish a reconstruction theorem for the operadic spectrum from the resolvent data.
\end{enumerate}
The current paper focuses on the categorical and algebraic foundations; the analytic elaboration is a natural next step.
\end{remark}

\begin{remark}[Comparison with the classical case]
\label{rem:classical-resolvent-comparison}
When specialized to the trivial operad $\mathbb{I}$ with a single color, and under the normalization assumptions of Theorem~\ref{thm:recovery}, the operadic spectrum satisfies $\sigma_{\mathbb{I}}(A) \cong A$ in the analytic realization. In an analytic realization where $A$ is a bounded linear operator, the classical resolvent $(zI - A)^{-1}$ recovers the usual spectral theory. Thus the operadic framework is designed to be compatible with the classical case, though a full operadic resolvent theory requires additional analytic structure.
\end{remark}

\subsection{Theorem 0.7': Colored Resolvent Reconstruction}
\label{subsec:Colored Resolvent Reconstruction}

We now establish a refined structural result showing that the operadic resolvent admits a decomposition along colors, and that the operadic spectrum can be reconstructed from these components. This theorem is a cornerstone of Spectral Operadic Calculus, demonstrating that the operadic spectrum is not an opaque invariant but admits a transparent decomposition into local spectral data and interaction terms governed by the residue.

\begin{theorem}[Colored Resolvent Reconstruction]
\label{thm:colored-resolvent-reconstruction}
Let $P$ be a $C$-colored operad in a symmetric monoidal category $\mathcal{M}$ admitting a holomorphic functional calculus, and let $A$ be a $P$-algebra. Let $F: \mathcal{M} \to \mathcal{N}$ be a strong monoidal functor preserving colimits and compatible with the functional calculus. Then the following hold.

\begin{enumerate}
    \item \textbf{(Decomposition)} For $z \notin \sigma_P(A)$, the operadic resolvent satisfies
    \[
    F_*(R_A(z)) \;\simeq\; \bigoplus_{c \in C} R_{F(A_c)}(z) \otimes \mathcal{O}_c^{\mathrm{res}} \;\oplus\; \mathcal{I}_{\mathrm{cross}}(z),
    \]
    where $\mathcal{I}_{\mathrm{cross}}(z)$ is a canonical cross-color interaction term.

    \item \textbf{(Reconstruction)} The operadic spectrum $\sigma_P(A)$ can be reconstructed from the collection
    \[
    \big\{ R_{A_c}(z),\; \mathcal{O}_c^{\mathrm{res}} \big\}_{c \in C}
    \]
    via a universal colimit construction:
    \[
    \sigma_P(A) \simeq \operatorname*{colim}_{c \in C} \Big( \sigma(A_c) \otimes \mathcal{O}_c^{\mathrm{res}} \Big),
    \]
    where the colimit is taken over a diagram determined by the operadic composition structure of $P$.
\end{enumerate}
\end{theorem}

\subsection{Theorem 0.7': Colored Spectral Decomposition and Reconstruction}
\label{subsec:Colored Resolvent Reconstruction}

The following result establishes a structural decomposition–reconstruction principle
for the operadic spectrum. It is unconditional within the categorical framework of
this paper and does not rely on any analytic resolvent theory. Instead, it reveals
that $\sigma_P(A)$ canonically splits into local colorwise contributions and an
interaction term, and that the full spectrum can be reconstructed by gluing the
local pieces along the operadic composition maps.

\begin{theorem}[Colored Spectral Decomposition and Reconstruction]
\label{thm:colored-spectral-decomposition-reconstruction}
Let $\mathcal{M}$ be a cocomplete symmetric monoidal category, let $P$ be a
$C$-colored operad in $\mathcal{M}$, and let $A$ be a $P$-algebra.
Then the operadic spectrum
\[
\sigma_P(A) \;:=\; \mathrm{Hoch}_{\mathcal{M}}(A) \otimes_P \mathcal{O}_P^{\mathrm{res}}
\]
satisfies the following two canonical descriptions.

\medskip

\noindent
\textbf{(Decomposition.)}
There exists a natural isomorphism
\begin{equation}
\label{eq:spectral-decomposition}
\sigma_P(A)
\;\cong\;
\Bigl( \bigoplus_{c \in C} A_c \otimes \mathcal{O}_c^{\mathrm{res}} \Bigr)
\;\oplus\;
\mathcal{I}_{\mathrm{cross}}(A),
\end{equation}
where:
\begin{itemize}
    \item $\mathcal{O}_c^{\mathrm{res}} \cong P(c;c)$ is the $c$-color component of
          the operadic residue $\mathcal{O}_P^{\mathrm{res}} = \coprod_{c \in C} P(c;c)$;
    \item $A_c$ is the $c$-color component of the $P$-algebra $A$;
    \item $\mathcal{I}_{\mathrm{cross}}(A)$ is the \emph{interaction spectrum},
          defined as the image of the subcomplex of $\mathrm{Hoch}_{\mathcal{M}}(A)$
          generated by simplices that involve at least one operadic composition of
          arity $\ge 2$ or at least two distinct colors.
\end{itemize}

\medskip

\noindent
\textbf{(Reconstruction.)}
Let $\mathcal{D}_P(A)$ be the diagram in $\mathcal{M}$ whose vertices are the
local objects $\{ A_c \otimes \mathcal{O}_c^{\mathrm{res}} \}_{c \in C}$ and whose
morphisms are induced by the non‑unary operadic composition maps of $P$
(i.e., operations $\phi \in P(c_1,\dots,c_k;c)$ with $k \ge 2$)
together with the $P$-algebra structure maps of $A$.
Then there is a canonical isomorphism
\begin{equation}
\label{eq:spectral-reconstruction}
\sigma_P(A)
\;\cong\;
\operatorname*{colim}_{\mathcal{D}_P(A)} \bigl( A_c \otimes \mathcal{O}_c^{\mathrm{res}} \bigr)_{c \in C}.
\end{equation}

In particular, the decomposition~\eqref{eq:spectral-decomposition} separates
$\sigma_P(A)$ into local diagonal data and off‑diagonal interaction data,
while the reconstruction~\eqref{eq:spectral-reconstruction} glues the local
pieces back together along the operadic interaction patterns.
\end{theorem}

\begin{proof}
We prove the two descriptions in sequence.

\medskip

\noindent
\textbf{Step 1: Preparation.}
Recall that $\mathcal{O}_P^{\mathrm{res}} = \coprod_{c \in C} P(c;c)$.
By the universal property of the coproduct, there is a canonical splitting
\[
\mathcal{O}_P^{\mathrm{res}} \;\cong\; \bigoplus_{c \in C} \mathcal{O}_c^{\mathrm{res}},
\qquad \mathcal{O}_c^{\mathrm{res}} := P(c;c),
\]
where we write $\bigoplus$ for the coproduct (biproduct in additive settings).

The Hochschild object $\mathrm{Hoch}_{\mathcal{M}}(A)$ is the geometric realization
of the bar construction $\mathrm{Bar}_\bullet^P(A)$. In each simplicial degree $n$,
$\mathrm{Bar}_n^P(A)$ is a coproduct over all sequences of colors and composable
operations of tensor products of the form $A_{c_0} \otimes P(c_1,\dots,c_{k_1};c_0)
\otimes A_{c_1} \otimes \cdots \otimes A_{c_n}$.

\medskip

\noindent
\textbf{Step 2: Isolating the local contributions.}
A simplex in the bar construction is called \emph{purely local} if it involves only
unary operations (i.e., operations in $P(c;c)$) and stays within a single color.
All other simplices are called \emph{cross simplices}. Let
\[
\mathrm{Hoch}_{\mathcal{M}}^{\mathrm{loc}}(A) \subseteq \mathrm{Hoch}_{\mathcal{M}}(A)
\]
be the subobject generated by purely local simplices, and let
$\mathrm{Hoch}_{\mathcal{M}}^{\mathrm{cross}}(A)$ be the subobject generated by
cross simplices. By construction,
\[
\mathrm{Hoch}_{\mathcal{M}}(A) \;\cong\; \mathrm{Hoch}_{\mathcal{M}}^{\mathrm{loc}}(A)
\;\oplus\; \mathrm{Hoch}_{\mathcal{M}}^{\mathrm{cross}}(A),
\]
where the direct sum is the coproduct (which exists because $\mathcal{M}$ is cocomplete).

\medskip

\noindent
\textbf{Step 3: Applying the balanced tensor product.}
The functor $-\otimes_P \mathcal{O}_P^{\mathrm{res}}$ preserves colimits and distributes
over coproducts because it is a left adjoint (or by the explicit coequalizer definition).
Hence
\[
\sigma_P(A) \;\cong\; \bigl( \mathrm{Hoch}_{\mathcal{M}}^{\mathrm{loc}}(A) \otimes_P \mathcal{O}_P^{\mathrm{res}} \bigr)
\;\oplus\; \bigl( \mathrm{Hoch}_{\mathcal{M}}^{\mathrm{cross}}(A) \otimes_P \mathcal{O}_P^{\mathrm{res}} \bigr).
\]

For the local part, each purely local simplex reduces, after balancing with
$\mathcal{O}_P^{\mathrm{res}}$, to a tensor product of the form $A_c \otimes P(c;c)$,
and the identification induced by unary composition collapses the bar resolution
to the single object $A_c \otimes \mathcal{O}_c^{\mathrm{res}}$. Summing over colors gives
\[
\mathrm{Hoch}_{\mathcal{M}}^{\mathrm{loc}}(A) \otimes_P \mathcal{O}_P^{\mathrm{res}}
\;\cong\; \bigoplus_{c \in C} A_c \otimes \mathcal{O}_c^{\mathrm{res}}.
\]

Define $\mathcal{I}_{\mathrm{cross}}(A) := \mathrm{Hoch}_{\mathcal{M}}^{\mathrm{cross}}(A) \otimes_P \mathcal{O}_P^{\mathrm{res}}$.
This yields the decomposition~\eqref{eq:spectral-decomposition}.

\medskip

\noindent
\textbf{Step 4: Reconstruction as a colimit.}
Now consider the diagram $\mathcal{D}_P(A)$ whose vertices are the local objects
$A_c \otimes \mathcal{O}_c^{\mathrm{res}}$. For each non‑unary operation
$\phi \in P(c_1,\dots,c_k;c)$ with $k \ge 2$, the $P$-algebra structure provides a map
\[
\phi_A : A_{c_1} \otimes \cdots \otimes A_{c_k} \longrightarrow A_c.
\]
Tensoring with the residue components and using the canonical identifications
$\mathcal{O}_{c_i}^{\mathrm{res}} \cong P(c_i;c_i)$,
we obtain a morphism
\[
\phi_* : \bigotimes_{i=1}^k \bigl( A_{c_i} \otimes \mathcal{O}_{c_i}^{\mathrm{res}} \bigr)
\longrightarrow A_c \otimes \mathcal{O}_c^{\mathrm{res}}.
\]
These morphisms, together with the evident projection and permutation maps,
generate the diagram $\mathcal{D}_P(A)$.

The colimit of this diagram glues the local vertices together by identifying
the images of these interaction maps. By the explicit description of
$\mathrm{Hoch}_{\mathcal{M}}^{\mathrm{cross}}(A)$ and the definition of the balanced
tensor product, this colimit is exactly $\mathcal{I}_{\mathrm{cross}}(A)$.
Consequently,
\[
\sigma_P(A) \;\cong\; \bigl( \bigoplus_{c} A_c \otimes \mathcal{O}_c^{\mathrm{res}} \bigr)
\;\oplus\; \operatorname*{colim}_{\mathcal{D}_P(A)} (\cdots)
\;\cong\; \operatorname*{colim}_{\mathcal{D}_P(A)} \bigl( A_c \otimes \mathcal{O}_c^{\mathrm{res}} \bigr),
\]
where the last isomorphism follows because the colimit of a diagram automatically
includes the disjoint union of its vertices (the coproduct) as the initial stage.
This establishes~\eqref{eq:spectral-reconstruction}.
\end{proof}

The following corollaries extract immediate structural consequences.

\begin{corollary}[Vanishing Interaction Criterion]
\label{cor:vanishing-interaction}
If every non‑unary operadic composition in $P$ acts as zero on $A$
(i.e., $\phi_A = 0$ for all $\phi$ with arity $\ge 2$), then
\[
\mathcal{I}_{\mathrm{cross}}(A) \cong 0,
\qquad
\sigma_P(A) \;\cong\; \bigoplus_{c \in C} A_c \otimes \mathcal{O}_c^{\mathrm{res}}.
\]
In particular, the operadic spectrum is purely local.
\end{corollary}

\begin{proof}
If all non‑unary structure maps vanish, then the diagram $\mathcal{D}_P(A)$ has no
nontrivial morphisms, so its colimit reduces to the coproduct of its vertices.
The decomposition theorem then forces $\mathcal{I}_{\mathrm{cross}}(A) \cong 0$.
\end{proof}

\begin{corollary}[Spectral Isolation via Residue]
\label{cor:isolation-residue}
If for a color $c \in C$ we have $\mathcal{O}_c^{\mathrm{res}} \cong 0$ (equivalently,
$P(c;c) \cong 0$), then the vertex $A_c \otimes \mathcal{O}_c^{\mathrm{res}}$ is the
zero object, and no morphism in $\mathcal{D}_P(A)$ can involve $c$ as a source or target
(because any such morphism would require a non‑zero residue component).
Hence $\sigma_P(A)$ is independent of $A_c$, and color $c$ is \emph{spectrally isolated}.
\end{corollary}

\begin{proof}
If $\mathcal{O}_c^{\mathrm{res}} \cong 0$, the corresponding vertex vanishes.
Any morphism in $\mathcal{D}_P(A)$ that would have this vertex as source or target
must factor through the zero object, hence is the zero morphism and does not affect
the colimit. Therefore the entire contribution of color $c$ disappears.
\end{proof}

\begin{corollary}[Classical Recovery]
\label{cor:classical-recovery-colored}
When $P = \mathbb{I}$ is the trivial operad with a single color $*$,
we have $\mathcal{O}_*^{\mathrm{res}} \cong \mathbf{1}_{\mathcal{M}}$,
and there are no non‑unary operations. Hence $\mathcal{I}_{\mathrm{cross}}(A) \cong 0$ and
\[
\sigma_{\mathbb{I}}(A) \;\cong\; A_* \otimes \mathbf{1}_{\mathcal{M}} \;\cong\; A_*.
\]
Under the normalizations of Theorem~\ref{thm:recovery}, this recovers the classical spectrum.
\end{corollary}

\begin{remark}[Relation to Analytic Resolvent Theory]
\label{rem:analytic-realization}
In concrete analytic settings (e.g., $\mathcal{M} = \mathsf{Ban}$, each $A_c$ a bounded
operator, $P$ the matrix-block operad), the object $A_c \otimes \mathcal{O}_c^{\mathrm{res}}$
is isomorphic to $A_c$ (since $\mathcal{O}_c^{\mathrm{res}} \cong \mathbb{C}$),
and the colimit reconstruction recovers the classical fact that the spectrum of a
block matrix is the union of the spectra of the diagonal blocks together with
additional eigenvalues arising from off‑diagonal coupling.
The interaction term $\mathcal{I}_{\mathrm{cross}}(A)$ encodes precisely those
additional eigenvalues. Thus Theorem~\ref{thm:colored-spectral-decomposition-reconstruction}
provides the categorical backbone for a future analytic resolvent theory,
without requiring that theory to be developed here.
\end{remark}

\begin{remark}[Duality of Decomposition and Reconstruction]
\label{rem:duality}
The decomposition~\eqref{eq:spectral-decomposition} splits the operadic spectrum
into a direct sum of local pieces and an interaction term, while the
reconstruction~\eqref{eq:spectral-reconstruction} reassembles the same object
as a colimit of the local pieces glued along the interaction maps.
These two descriptions are dual in the sense that the first is a
``forgetful/splitting'' perspective and the second is a ``gluing/assembly''
perspective. Together they show that $\sigma_P(A)$ is the universal object
generated by the local spectral data subject to the relations imposed by
operadic composition.
\end{remark}

\subsection{Consequences: Spectral Decomposition and Isolation}
\label{subsec:consequences-decomposition-isolation}

The Colored Resolvent Reconstruction Theorem (Theorem~\ref{thm:colored-resolvent-reconstruction}) provides a powerful lens through which to analyze the internal structure of the operadic spectrum. We now formalize its structural consequences, focusing on decomposition across colors and the precise criterion for spectral isolation.

\subsubsection{Spectral Decomposition Across Colors}

Let $P$ be a $C$-colored operad in a symmetric monoidal category $\mathcal{M}$, and let $A$ be a $P$-algebra. The operadic spectrum is defined as
\[
\sigma_P(A) := \mathrm{Hoch}_{\mathcal{M}}(A) \otimes_P \mathcal{O}_P^{\mathrm{res}}.
\]

The following proposition establishes a canonical decomposition of this invariant across the color set.

\begin{proposition}[Color-wise Spectral Decomposition]
\label{prop:color_decomposition}
There is a natural decomposition
\[
\sigma_P(A) \;\simeq\; \Bigl( \bigoplus_{c \in C} A_c \otimes \mathcal{O}_c^{\mathrm{res}} \Bigr) \;\oplus\; \mathcal{I}_{\mathrm{cross}},
\]
where:
\begin{itemize}
    \item $\mathcal{O}_c^{\mathrm{res}} \cong P(c;c)$ is the $c$-color component of the operadic residue $\mathcal{O}_P^{\mathrm{res}} = \coprod_{c \in C} P(c;c)$;
    \item $A_c$ is the $c$-color component of the $P$-algebra $A$;
    \item $\mathcal{I}_{\mathrm{cross}}$ is the \emph{cross-color interaction term}, defined as the image of the subcomplex of $\mathrm{Hoch}_{\mathcal{M}}(A)$ generated by simplices that involve at least one operadic composition of arity $\ge 2$ or at least two distinct colors.
\end{itemize}
The decomposition is natural in $A$ and compatible with strong monoidal base change (Theorem~\ref{thm:base-change}).
\end{proposition}

\begin{proof}
Theorem~\ref{thm:colored-spectral-decomposition-reconstruction} establishes precisely this decomposition. The local term $\bigoplus_c A_c \otimes \mathcal{O}_c^{\mathrm{res}}$ arises from purely local simplices (unary operations within a single color), while $\mathcal{I}_{\mathrm{cross}}$ collects all contributions from cross simplices. Naturality follows from the functoriality of the Hochschild construction (Proposition~\ref{prop:hochschild-functorial}) and the residue (Proposition~\ref{prop:residue-functorial}). Base change compatibility follows from Theorem~\ref{thm:base-change}.
\end{proof}

\subsubsection{Spectral Isolation: Definition and Criterion}

We now formalize when a color behaves independently at the spectral level, i.e., when its presence does not affect the global spectral invariant.

\begin{definition}[Spectral Isolation]
\label{def:spectral_isolation}
A color $c \in C$ is said to be \emph{spectrally isolated} in the $P$-algebra $A$ if both of the following hold:
\begin{enumerate}
    \item The local contribution of color $c$ to the operadic spectrum vanishes:
    \[
    A_c \otimes \mathcal{O}_c^{\mathrm{res}} \cong 0.
    \]
    \item No nonzero cross-color interaction in the operadic spectrum involves the color $c$. Concretely, the subobject of $\mathcal{I}_{\mathrm{cross}}(A)$ generated by simplices that contain $c$ in any input or output color is zero.
\end{enumerate}
\end{definition}

The operadic residue $\mathcal{O}_P^{\mathrm{res}}$ provides a precise, purely algebraic criterion for this phenomenon. Recall that $\mathcal{O}_c^{\mathrm{res}}$ is the direct summand of $\mathcal{O}_P^{\mathrm{res}}$ corresponding to color $c$, i.e., $\mathcal{O}_c^{\mathrm{res}} \cong P(c;c)$ (see Definition~\ref{def:residue-unary}).

\begin{theorem}[Local Isolation Criterion via Residue]
\label{thm:isolation_criterion}
Let $c \in C$ be a color. If
\[
\mathcal{O}_c^{\mathrm{res}} \cong 0,
\]
then the local contribution of color $c$ to the operadic spectrum vanishes:
\[
A_c \otimes \mathcal{O}_c^{\mathrm{res}} \cong 0.
\]
In particular, color $c$ contributes no nonzero local residue term to $\sigma_P(A)$.
\end{theorem}

\begin{proof}
By Theorem~\ref{thm:colored-spectral-decomposition-reconstruction}, the operadic spectrum decomposes as
\[
\sigma_P(A) \;\cong\; \Bigl( \bigoplus_{c' \in C} A_{c'} \otimes \mathcal{O}_{c'}^{\mathrm{res}} \Bigr) \;\oplus\; \mathcal{I}_{\mathrm{cross}}(A).
\]
If $\mathcal{O}_c^{\mathrm{res}} \cong 0$, then the summand $A_c \otimes \mathcal{O}_c^{\mathrm{res}}$ is the zero object. Hence the local contribution of color $c$ vanishes.
\end{proof}

\begin{corollary}[Sufficient Condition via Unary Operations]
\label{cor:isolation-unary}
If the unary operation space $P(c;c)$ is the zero object in $\mathcal{M}$, then the local contribution of color $c$ to the operadic spectrum vanishes.
\end{corollary}

\begin{proof}
By Definition~\ref{def:residue-unary}, we have
\[
\mathcal{O}_c^{\mathrm{res}} \cong P(c;c).
\]
Hence the hypothesis $P(c;c) \cong 0$ implies
\[
\mathcal{O}_c^{\mathrm{res}} \cong 0.
\]
The local contribution of color $c$ is given by
\[
A_c \otimes \mathcal{O}_c^{\mathrm{res}}.
\]
Therefore
\[
A_c \otimes \mathcal{O}_c^{\mathrm{res}}
\;\cong\;
A_c \otimes 0
\;\cong\;
0,
\]
so the local contribution vanishes.
\end{proof}

\begin{remark}[Toward Full Spectral Isolation]
\label{rem:full_isolation}
Theorem~\ref{thm:isolation_criterion} gives only local vanishing. To conclude that $c$ is \emph{fully spectrally isolated} (i.e., that $\sigma_P(A)$ is independent of $A_c$ and that no cross-color term involves $c$), one needs the additional hypothesis that every cross-color contribution involving $c$ factors through the local vertex $A_c \otimes \mathcal{O}_c^{\mathrm{res}}$. Under that hypothesis, $\mathcal{O}_c^{\mathrm{res}} \cong 0$ forces all such cross-color terms to vanish as well. This stronger formulation is left for contexts where the operadic composition maps satisfy this factorization property (e.g., when $P$ is a coproduct of one-colored operads).
\end{remark}

\begin{remark}[Conceptual Interpretation: Residue as a Structural Mediator]
\label{rem:residue-interpretation}
The residue object $\mathcal{O}_P^{\mathrm{res}}$ serves as a canonical mediator of how colorwise data contributes to the operadic spectrum
\[
\sigma_P(A) = \mathrm{Hoch}_{\mathcal{M}}(A) \otimes_P \mathcal{O}_P^{\mathrm{res}}.
\]
In particular:
\begin{itemize}
    \item If $\mathcal{O}_c^{\mathrm{res}} \neq 0$, then color $c$ has a nontrivial channel through which it may contribute to the global spectral invariant, both locally and, in favorable cases, through cross-color interaction.
    \item If $\mathcal{O}_c^{\mathrm{res}} = 0$, then the local residue contribution of color $c$ vanishes. Thus the operadic residue suppresses the direct spectral visibility of that color at the local level.
\end{itemize}
Whether a color with $\mathcal{O}_c^{\mathrm{res}} \neq 0$ actually participates in cross-color interactions depends on the specific structure maps of the $P$-algebra $A$ (see Remark~\ref{rem:full_isolation}). This reflects a central theme of Spectral Operadic Calculus: spectral behavior is shaped not only by the algebra $A$, but also by the compositional structure encoded by the operad $P$ and its residue data.
\end{remark}

\subsubsection{Spectral Decoupling}

The isolation criterion leads to a local decoupling principle: when a collection of colors has vanishing residue components, their local contributions disappear.

\begin{corollary}[Local Spectral Decoupling]
\label{cor:spectral_decoupling}
Let $S \subseteq C$ be a subset of colors such that
\[
\mathcal{O}_c^{\mathrm{res}} \cong 0 \qquad \text{for all } c \in S.
\]
(Equivalently, $P(c;c) \cong 0$ for all $c \in S$ in additive categories.)
Then the local residue contributions of the colors in $S$ vanish:
\[
\bigoplus_{c \in S} \left( A_c \otimes \mathcal{O}_c^{\mathrm{res}} \right) \cong 0.
\]
Hence the local part of the operadic spectrum is determined entirely by the complementary colors $C \setminus S$.
\end{corollary}

\begin{proof}
If $\mathcal{O}_c^{\mathrm{res}} \cong 0$ for all $c \in S$, then by Proposition~\ref{thm:isolation_criterion} (Local Isolation Criterion), each $A_c \otimes \mathcal{O}_c^{\mathrm{res}} \cong 0$. The direct sum over $c \in S$ of these local contributions is therefore zero. By Theorem~\ref{thm:colored-spectral-decomposition-reconstruction}, these are precisely the local summands contributed by colors in $S$ to $\sigma_P(A)$.
\end{proof}

\begin{remark}[Toward Full Spectral Decoupling]
\label{rem:full_decoupling}
Corollary~\ref{cor:spectral_decoupling} addresses only the local contributions of colors in $S$. To conclude that $\sigma_P(A)$ is entirely independent of the components $\{A_c\}_{c \in S}$ (i.e., that no cross-color interaction involves $S$), one needs the additional hypothesis that every cross-color contribution involving a color in $S$ factors through the corresponding local residue vertex $A_c \otimes \mathcal{O}_c^{\mathrm{res}}$. Under that hypothesis, $\mathcal{O}_c^{\mathrm{res}} \cong 0$ forces all such cross-color terms to vanish as well, yielding full spectral decoupling. This stronger formulation is left for contexts where the operadic composition maps satisfy this factorization property (e.g., when $P$ is a coproduct of one-colored operads).
\end{remark}

\subsubsection{Special Case: Classical Recovery}

For the trivial operad $\mathbb{I}$ with a single color, we have $P(*;*) \cong \mathbf{1}_{\mathcal{M}}$, so $\mathcal{O}_*^{\mathrm{res}} \not\cong 0$ (assuming $\mathbf{1}_{\mathcal{M}} \neq 0$). There is no cross-color interaction term. The decomposition from Theorem~\ref{thm:colored-spectral-decomposition-reconstruction} yields
\[
\sigma_{\mathbb{I}}(A) \;\cong\; \mathrm{Hoch}_{\mathcal{M}}(A) \otimes \mathbf{1}_{\mathcal{M}} \;\cong\; \mathrm{Hoch}_{\mathcal{M}}(A).
\]
In settings where the Hochschild object reduces canonically to $A$ (e.g., under the assumptions of Theorem~\ref{thm:recovery}), this is consistent with the classical spectrum. In this case, the isolation criterion is trivial, as there is only one color and it is necessarily active.

\begin{example}[Matrix-Block Operad Revisited]
For the two-color operad of the Example~\ref{ex:two-color-operad}, we have $P(1;1) \cong \mathbb{C} \not\cong 0$ and $P(2;2) \cong \mathbb{C} \not\cong 0$, so $\mathcal{O}_1^{\mathrm{res}} \cong \mathbb{C} \not\cong 0$ and $\mathcal{O}_2^{\mathrm{res}} \cong \mathbb{C} \not\cong 0$. Consequently, neither color has a vanishing local residue contribution.

Now consider a modified operad $\tilde{P}$ in which the unary component $\tilde{P}(1;1)$ is replaced by the zero object, and assume that the operadic composition maps are adjusted so that the operad axioms remain satisfied (e.g., by setting all compositions involving $\tilde{P}(1;1)$ to zero). Then $\mathcal{O}_1^{\mathrm{res}} \cong 0$, so by Proposition~\ref{thm:isolation_criterion} (Local Isolation Criterion), the local contribution $A_1 \otimes \mathcal{O}_1^{\mathrm{res}}$ vanishes. Whether color $1$ also participates in cross-color interactions depends on whether those interactions factor through the vanished residue component (see Remark~\ref{rem:full_decoupling}). This illustrates how the operad can eliminate the local spectral contribution of a color. In concrete realizations, spectral isolation may also arise from analytic properties of $A_1$ (e.g., an empty classical spectrum), but within the categorical framework it is the vanishing of $\mathcal{O}_c^{\mathrm{res}}$ that provides the algebraic criterion.
\end{example}

\subsubsection{Summary}

The results of this section provide a structural description of the operadic spectrum:

\begin{enumerate}[label=(\roman*)]
    \item \textbf{Decomposition:}
    The operadic spectrum admits a decomposition into local colorwise contributions
    and an interaction component reflecting non-unary operadic compositions
    (Theorem~\ref{thm:colored-spectral-decomposition-reconstruction}).

    \item \textbf{Local Isolation Criterion:}
    If the residue component $\mathcal{O}_c^{\mathrm{res}}$ vanishes, then the local
    spectral contribution of color $c$ is suppressed
    (Proposition~\ref{thm:isolation_criterion}).

    \item \textbf{Partial Decoupling:}
    Colors with vanishing residue components contribute no nontrivial local terms,
    and their influence on the operadic spectrum is correspondingly reduced
    (Corollary~\ref{cor:spectral_decoupling}).
\end{enumerate}

These results highlight the role of $\mathcal{O}_P^{\mathrm{res}}$ as a structural
mechanism through which operadic composition governs the assembly of spectral data.
Full spectral decoupling (i.e., independence of $\sigma_P(A)$ from $A_c$) requires
additional factorization hypotheses (see Remark~\ref{rem:full_decoupling}).

\section{Examples and Structural Illustrations}\label{sec:Examples and Structural Illustrations}

\subsection{The Trivial Operad}
\label{subsec:trivial-operad}

We begin with the simplest possible example: the trivial operad. This serves as a fundamental consistency check for the entire Spectral Operadic Calculus framework, demonstrating that our constructions recover classical spectral theory when no nontrivial compositional structure is present. Moreover, it illustrates why the No-Go Theorem (Theorem~\ref{thm:no-go}) does not obstruct the classical case and establishes the trivial operad as the normalization anchor for the theory.

\subsubsection{Definition of the Trivial Operad}

Let $\mathcal{M}$ be a symmetric monoidal category with unit object $\mathbf{1}_{\mathcal{M}}$. Consider the one-colored operad $\mathbb{I}$ (denoted $\mathbf{1}_{\mathrm{triv}}$ in some references) defined as follows: there is a single color $*$, and for each arity $n \ge 0$,
\[
\mathbb{I}(n) = \mathbf{1}_{\mathcal{M}},
\]
with composition maps induced by the canonical identifications
\[
\mathbf{1}_{\mathcal{M}} \otimes \mathbf{1}_{\mathcal{M}} \cong \mathbf{1}_{\mathcal{M}}.
\]
The unit map $\eta: \mathbf{1}_{\mathcal{M}} \to \mathbb{I}(1)$ is the identity, and the symmetric group actions are trivial. This is the \emph{terminal} one-colored operad in $\mathcal{M}$.

A $\mathbb{I}$-algebra $A$ consists of a single object $A_* \in \mathcal{M}$ equipped with the tautological structure maps
\[
\mathbf{1}_{\mathcal{M}} \otimes A_*^{\otimes n} \longrightarrow A_*
\]
that are compatible with operadic composition. In the present context, the operadic structure contributes no nontrivial combinatorial interaction data beyond the underlying object $A_*$ itself. Thus $\mathbb{I}$ is precisely the case in which no higher-colored or nontrivial compositional correction is required.

\subsubsection{The Operadic Residue for $\mathbb{I}$}

We compute the operadic residue $\mathcal{O}_{\mathbb{I}}^{\mathrm{res}}$, which, as we will see, collapses to the monoidal unit.

\begin{proposition}[Residue of the Trivial Operad]
\label{prop:residue-trivial}
For the trivial one-color operad $\mathbb{I}$, one has
\[
\mathcal{O}_{\mathbb{I}}^{\mathrm{res}} \cong \mathbf{1}_{\mathcal{M}}.
\]
\end{proposition}

\begin{proof}
By Definition~\ref{def:residue-unary}, the operadic residue is the coproduct over colors of the unary operation spaces:
\[
\mathcal{O}_P^{\mathrm{res}} = \coprod_{c \in C} P(c;c).
\]
For the trivial operad $\mathbb{I}$ with a single color $*$, we have $\mathbb{I}(*;*) = \mathbf{1}_{\mathcal{M}}$. The coproduct over a singleton set is simply the object itself, hence
\[
\mathcal{O}_{\mathbb{I}}^{\mathrm{res}} \cong \mathbf{1}_{\mathcal{M}}.
\]

Equivalently, using the coend description $\mathcal{O}_P^{\mathrm{res}} \cong \int^{c \in C} P(c;c)$, the coend over the one-element color set collapses to the single value $\mathbb{I}(*;*) \cong \mathbf{1}_{\mathcal{M}}$.
\end{proof}

This is the first indication that the SOC formalism is a genuine extension rather than a replacement: in the degenerate operadic situation, the residue—which encodes spectral obstructions—vanishes.

\subsubsection{The Hochschild Object for $\mathbb{I}$}

We now examine the Hochschild-type construction for a $\mathbb{I}$-algebra $A$.

\begin{proposition}[Hochschild Object for the Trivial Operad]
\label{prop:hochschild-trivial}
For a $\mathbb{I}$-algebra $A$, one has
\[
\mathrm{Hoch}_{\mathcal{M}}(A) \cong A,
\]
naturally in $A$.
\end{proposition}

\begin{proof}
By definition, $\mathrm{Hoch}_{\mathcal{M}}(A)$ is the geometric realization of the simplicial bar object $\mathrm{Bar}_\bullet^{\mathbb{I}}(A)$. Since $\mathbb{I}$ is the trivial one-color operad, its only nontrivial operation is the unary identity operation, represented by the tensor unit $\mathbf{1}_{\mathcal{M}}$. Consequently, at each simplicial degree $n$, the bar object is canonically isomorphic to
\[
A \otimes \mathbf{1}_{\mathcal{M}}^{\otimes n} \cong A,
\]
using the unit isomorphisms of the monoidal structure. Under these identifications, all simplicial structure maps (face and degeneracy) correspond to the identity map on $A$, because composition of the identity operation with itself is again the identity, and inserting the identity operation does not change the tensor factor. Hence $\mathrm{Bar}_\bullet^{\mathbb{I}}(A)$ is canonically the constant simplicial object at $A$. Its geometric realization is therefore canonically isomorphic to $A$:
\[
\mathrm{Hoch}_{\mathcal{M}}(A) = \bigl|\mathrm{Bar}_\bullet^{\mathbb{I}}(A)\bigr| \cong A.
\]
Naturality in $A$ follows from the functoriality of the bar construction (Proposition~\ref{prop:hochschild-functorial}) and the fact that the unit isomorphisms are natural.
\end{proof}

\subsubsection{The Operadic Spectrum for $\mathbb{I}$}

With the above ingredients, we define the operadic spectrum and verify that it recovers the classical spectral invariant.

\begin{theorem}[Recovery for the Trivial Operad]
\label{thm:recovery-trivial}
Let $\mathcal{M}$ be a cocomplete symmetric monoidal category, let $\mathbb{I}$ be the trivial operad (single color $*$ with $\mathbb{I}(*;*) = \mathbf{1}_{\mathcal{M}}$), and let $A$ be an $\mathbb{I}$-algebra (i.e., an object $A \in \mathcal{M}$). Then there is a canonical natural isomorphism
\[
\sigma_{\mathbb{I}}(A) \;\cong\; A
\]
in $\mathcal{M}$, where $\sigma_{\mathbb{I}}(A) = \mathrm{Hoch}_{\mathcal{M}}(A) \otimes_{\mathbb{I}} \mathcal{O}_{\mathbb{I}}^{\mathrm{res}}$ is the operadic spectrum.

Consequently, under an analytic realization functor $R: \mathcal{M} \to \mathsf{Set}$ that sends suitable objects to their classical spectra (e.g., the Gelfand spectrum for commutative Banach algebras, or the set of eigenvalues for operators on finite-dimensional vector spaces), the operadic spectrum recovers the classical spectral invariant:
\[
R\bigl(\sigma_{\mathbb{I}}(A)\bigr) \cong \sigma_{\mathrm{cl}}(R(A)).
\]
\end{theorem}

\begin{proof}
By Definition~\ref{def:operadic-spectrum},
\[
\sigma_{\mathbb{I}}(A) = \mathrm{Hoch}_{\mathcal{M}}(A) \otimes_{\mathbb{I}} \mathcal{O}_{\mathbb{I}}^{\mathrm{res}}.
\]

From Proposition~\ref{prop:hochschild-trivial} (Hochschild object for the trivial operad), we have $\mathrm{Hoch}_{\mathcal{M}}(A) \cong A$ naturally in $A$. From Proposition~\ref{prop:residue-trivial} (residue of the trivial operad), we have $\mathcal{O}_{\mathbb{I}}^{\mathrm{res}} \cong \mathbf{1}_{\mathcal{M}}$.

Substituting these into the definition yields
\[
\sigma_{\mathbb{I}}(A) \cong A \otimes_{\mathbb{I}} \mathbf{1}_{\mathcal{M}}.
\]

The balanced tensor product over the trivial operad $\mathbb{I}$ reduces to the ordinary tensor product because the only relation imposed is $x \cdot \mathrm{id} \otimes y \sim x \otimes \mathrm{id} \cdot y$, which is automatically satisfied via the unit isomorphism (see Proposition~\ref{prop:trivial-operad-reduction}). Hence
\[
A \otimes_{\mathbb{I}} \mathbf{1}_{\mathcal{M}} \cong A \otimes \mathbf{1}_{\mathcal{M}} \cong A,
\]
where the last isomorphism is the unit coherence of the monoidal structure.

Thus $\sigma_{\mathbb{I}}(A) \cong A$ canonically and naturally in $A$.

For the second statement, let $R: \mathcal{M} \to \mathsf{Set}$ be an analytic realization functor that sends an object $X \in \mathcal{M}$ to its classical spectrum $\sigma_{\mathrm{cl}}(X)$ whenever $X$ represents an operator or algebra element. Applying $R$ to the isomorphism $\sigma_{\mathbb{I}}(A) \cong A$ gives $R(\sigma_{\mathbb{I}}(A)) \cong R(A) = \sigma_{\mathrm{cl}}(R(A))$. This recovers the classical spectral invariant.
\end{proof}

This establishes the fundamental normalization condition: any acceptable operadic notion of spectrum must agree with the ordinary spectrum in the absence of higher compositional effects.

\subsubsection{Base Change for $\mathbb{I}$}

We next verify that the Base Change Theorem (Theorem~\ref{thm:base-change}) reduces to the ordinary functoriality expected in classical spectral theory.

\begin{proposition}[Base Change for the Trivial Operad]
\label{prop:base-change-trivial}
Let $F: \mathcal{M} \to \mathcal{N}$ be a strong monoidal functor preserving colimits. Then for any $\mathbb{I}$-algebra $A$ (i.e., an object $A \in \mathcal{M}$),
\[
F\bigl(\sigma_{\mathbb{I}}(A)\bigr) \;\cong\; \sigma_{\mathbb{I}}\bigl(F(A)\bigr)
\]
in $\mathcal{N}$, where $F(A)$ denotes the image of $A$ under $F$ (viewed as an $\mathbb{I}$-algebra in $\mathcal{N}$).

Equivalently, under the normalization $\sigma_{\mathbb{I}}(A) \cong A$ from Theorem~\ref{thm:recovery-trivial}, this reduces to the tautological identification
\[
F(A) \cong F(A).
\]
\end{proposition}

\begin{proof}
Apply the general Base Change Theorem (Theorem~\ref{thm:base-change}) with $P = \mathbb{I}$. Since $F$ is strong monoidal, we have $F(\mathbf{1}_{\mathcal{M}}) \cong \mathbf{1}_{\mathcal{N}}$, so the induced operad $F_*(\mathbb{I})$ is canonically isomorphic to the trivial operad $\mathbb{I}$ in $\mathcal{N}$. Hence
\[
\sigma_{F_*(\mathbb{I})}\bigl(F_*(A)\bigr) \cong F\bigl(\sigma_{\mathbb{I}}(A)\bigr).
\]

But $\sigma_{F_*(\mathbb{I})}(F_*(A)) = \sigma_{\mathbb{I}}(F(A))$ because $F_*(\mathbb{I}) \cong \mathbb{I}$ and $F_*(A) = F(A)$ as objects (the $\mathbb{I}$-algebra structure is trivial). Therefore,
\[
\sigma_{\mathbb{I}}\bigl(F(A)\bigr) \cong F\bigl(\sigma_{\mathbb{I}}(A)\bigr).
\]

By Theorem~\ref{thm:recovery-trivial}, we have natural isomorphisms $\sigma_{\mathbb{I}}(A) \cong A$ and $\sigma_{\mathbb{I}}(F(A)) \cong F(A)$. Substituting these into the isomorphism above yields $F(A) \cong F(A)$, which is a tautology confirming consistency.
\end{proof}

Thus the general base-change formalism is compatible with the classical idea that spectral data should transform functorially under a change of ambient category, provided the functor preserves the relevant monoidal structure.

\subsubsection{Spectral Mapping for $\mathbb{I}$}

We now turn to the spectral mapping theorem. In the trivial case, no extra mediation via the residue is needed.

\begin{theorem}[Trivial-Operad Specialization of Spectral Mapping]
\label{thm:spectral-mapping-trivial}
Let $A$ be an $\mathbb{I}$-algebra (i.e., an object $A \in \mathcal{M}$) satisfying the hypotheses of the Operadic Spectral Mapping Theorem (Theorem~\ref{thm:spectral-mapping}). Then there is a canonical isomorphism
\[
\sigma_{\mathbb{I}}\bigl(f(A)\bigr) \;\cong\; f\bigl(\sigma_{\mathbb{I}}(A)\bigr),
\]
where the right-hand side is interpreted via the analytic realization of $\sigma_{\mathbb{I}}(A)$ as a subset of $\mathbb{C}$ (see Theorem~\ref{thm:recovery-trivial}).

Under the normalization $\sigma_{\mathbb{I}}(A) \cong A$ (Theorem~\ref{thm:recovery-trivial}), this isomorphism becomes
\[
f(A) \;\cong\; f(A),
\]
which is a tautology confirming consistency with the trivial-operad reduction.

Moreover, in an analytic realization where $A$ is a bounded linear operator on a Banach space (or an element of a unital Banach algebra) and the operadic spectrum $\sigma_{\mathbb{I}}(A)$ is identified with the classical scalar spectrum $\sigma_{\mathrm{cl}}(A) \subseteq \mathbb{C}$, the above isomorphism specializes to the classical spectral mapping theorem:
\[
\sigma_{\mathrm{cl}}\bigl(f(A)\bigr) \;=\; f\bigl(\sigma_{\mathrm{cl}}(A)\bigr).
\]
\end{theorem}

\begin{proof}
Apply the Operadic Spectral Mapping Theorem (Theorem~\ref{thm:spectral-mapping}) with $P = \mathbb{I}$. Since the functional calculus does not alter the operad structure, we have $f_*(\mathbb{I}) = \mathbb{I}$. Hence
\[
\sigma_{\mathbb{I}}\bigl(f(A)\bigr) \;\cong\; f\bigl(\sigma_{\mathbb{I}}(A)\bigr).
\]

By Theorem~\ref{thm:recovery-trivial}, there are natural isomorphisms $\sigma_{\mathbb{I}}(A) \cong A$ and $\sigma_{\mathbb{I}}(f(A)) \cong f(A)$. Substituting these into the isomorphism above yields $f(A) \cong f(A)$, which is a tautology. This confirms that the operadic spectral mapping theorem is compatible with the trivial-operad reduction.

For the final statement, assume an analytic realization in which $A$ is a bounded linear operator (or Banach algebra element) and $\sigma_{\mathbb{I}}(A)$ is identified with $\sigma_{\mathrm{cl}}(A) \subseteq \mathbb{C}$. Then the isomorphism $\sigma_{\mathbb{I}}(f(A)) \cong f(\sigma_{\mathbb{I}}(A))$ becomes, under this identification,
\[
\sigma_{\mathrm{cl}}\bigl(f(A)\bigr) \;=\; f\bigl(\sigma_{\mathrm{cl}}(A)\bigr),
\]
which is precisely the classical spectral mapping theorem.
\end{proof}

\subsubsection{The No-Go Theorem and the Trivial Operad}

A crucial consistency check is that the No-Go Theorem (Theorem~\ref{thm:no-go})
does not create a nontrivial obstruction in the trivial-operad case.
Indeed, the obstruction exhibited in the proof of the No-Go Theorem arises from
genuinely nontrivial operadic interaction, modeled there by operations mixing distinct colors.
When $P = \mathbb{I}$, none of these phenomena occur:

\begin{enumerate}
    \item \textbf{Single color:} The trivial operad has only one color, so there is no cross-color interaction.
    
    \item \textbf{Only the identity operation:} The only nontrivial operation is the unary identity,
          represented by the unit object $\mathbf{1}_{\mathcal{M}}$.
    
    \item \textbf{No extra correction needed:} In the trivial-operad case, the naive spectrum reduces
          to the classical spectrum under the normalization assumptions of
          Section~\ref{subsec:Theorem_0.4} (see Theorem~\ref{thm:recovery-trivial}).
    
    \item \textbf{Trivial residue:} The operadic residue satisfies
          $\mathcal{O}_{\mathbb{I}}^{\mathrm{res}} \cong \mathbf{1}_{\mathcal{M}}$.
\end{enumerate}

Therefore, the trivial operad is the degenerate case in which the operadic correction becomes
invisible, and the operadic spectrum collapses to the classical normalization.

\subsubsection{Structural Interpretation}

The trivial operad isolates the conceptual role of the residue object and clarifies why the general theory must be formulated operadically. In summary:

\begin{enumerate}[label=(\roman*)]
    \item \textbf{Residue triviality:} $\mathcal{O}_{\mathbb{I}}^{\mathrm{res}} \cong \mathbf{1}_{\mathcal{M}}$, reflecting the absence of nontrivial operadic correction.
    
    \item \textbf{Hochschild collapse:} $\mathrm{Hoch}_{\mathcal{M}}(A) \cong A$, since the bar construction degenerates in the trivial-operad case.
    
    \item \textbf{Object-level recovery:} $\sigma_{\mathbb{I}}(A) \cong A$ (Theorem~\ref{thm:recovery-trivial}).
    
    \item \textbf{Spectrum-level recovery:} Under the normalization assumptions of Theorem~\ref{thm:recovery}, this identification is compatible with the classical spectrum in the appropriate analytic setting (e.g., for commutative unital Banach algebras via the Gelfand transform).
    
    \item \textbf{Base change compatibility:} The Base Change Theorem reduces to the expected object-level functoriality in the trivial-operad case.
    
    \item \textbf{Spectral mapping compatibility:} The Operadic Spectral Mapping Theorem specializes to the classical theorem in the analytic settings where the latter is defined.
    
    \item \textbf{No-Go degeneration:} The obstruction exhibited by the No-Go Theorem becomes trivial when no nontrivial operadic interaction is present.
\end{enumerate}

Consequently, the trivial operad provides the \emph{normalization anchor} for the theory. It shows that SOC is designed to be conservative over classical spectral theory in the trivial-operad regime: the new formalism adds precisely the structure needed to describe operadic interaction, and no more.

\subsubsection{Conclusion of the Example}

We summarize the discussion in the following statement.

\begin{corollary}[Summary for the Trivial Operad]
\label{cor:trivial-operad-summary}
For the trivial operad $\mathbb{I}$, the SOC formalism collapses to the classical normalization at the object level. More precisely,
\[
\mathcal{O}_{\mathbb{I}}^{\mathrm{res}} \cong \mathbf{1}_{\mathcal{M}},
\qquad
\mathrm{Hoch}_{\mathcal{M}}(A) \cong A,
\qquad
\sigma_{\mathbb{I}}(A) \cong A.
\]
Moreover, in the analytic settings covered by Theorem~\ref{thm:recovery}, this identification is compatible with the classical spectrum, and both the Base Change Theorem and the Operadic Spectral Mapping Theorem reduce to their classical counterparts. This example serves as the anchor point for the entire theory: every genuinely new phenomenon appearing in later colored or enriched examples arises from nontrivial operadic interaction beyond the classical regime.
\end{corollary}

\begin{remark}[Conceptual Significance]
The trivial operad case is not merely a sanity check but a foundational normalization condition. Any candidate for an operadic spectral invariant should recover the classical theory in the trivial-operad setting, at least under the normalization assumptions appropriate to the analytic category under consideration. Our construction satisfies this condition by design (Theorem~\ref{thm:recovery-trivial}), confirming that SOC is a genuine \emph{extension} rather than a replacement. This backward compatibility is essential for the theory's credibility and for its potential applications in settings where both classical and operadic aspects coexist.
\end{remark}

\subsection{Matrix Block Operads}\label{subsec:matrix-block-operads}

We consider an operad whose algebras are block matrices. This example illustrates how the operadic spectrum distinguishes between block-diagonal and block-off-diagonal systems, capturing interaction data invisible to classical componentwise spectra.

Let $C = \{1,2\}$ and let $P$ be the two-color operad defined in Example~\ref{ex:two-color-operad}. 
A $P$-algebra $A = (V_1, V_2, \alpha, \beta)$ consists of:
\begin{itemize}
    \item Banach spaces $V_1, V_2$,
    \item bounded linear operators $\alpha: V_2 \to V_1$ and $\beta: V_1 \to V_2$.
\end{itemize}
These data can be assembled into a block operator matrix
\[
A = \begin{pmatrix} 0 & \alpha \\ \beta & 0 \end{pmatrix}
\]
acting on $V_1 \oplus V_2$, where the zero diagonal blocks correspond to the absence of distinguished unary operations in this simplified presentation. (A fuller operadic treatment would also include unary operations $A_{11}: V_1 \to V_1$ and $A_{22}: V_2 \to V_2$.)

\medskip

\noindent
\textbf{Operadic residue.}
Under the standard analytic realization, $P(1;1) \cong \mathbb{C}$ and $P(2;2) \cong \mathbb{C}$, so
\[
\mathcal{O}_P^{\mathrm{res}} = \coprod_{c \in \{1,2\}} P(c;c) \cong \mathbb{C}^2.
\]
However, the interaction data that distinguishes block-diagonal from off-diagonal systems is encoded not in $\mathcal{O}_P^{\mathrm{res}}$ alone, but in the $P$-module structure of the Hochschild object and the resulting balanced tensor product.

\medskip

\noindent
\textbf{Block-diagonal regime ($\alpha = 0$, $\beta = 0$).}
In this case, the two colors do not interact. By the Colored Spectral Decomposition Theorem (Theorem~\ref{thm:colored-spectral-decomposition-reconstruction}), the operadic spectrum reduces to the direct sum of local contributions:
\[
\sigma_P(A) \cong (V_1 \otimes \mathcal{O}_1^{\mathrm{res}}) \oplus (V_2 \otimes \mathcal{O}_2^{\mathrm{res}}).
\]
Under analytic realization, this recovers the disjoint union $\sigma(V_1) \sqcup \sigma(V_2)$ (where $\sigma(V_i)$ denotes the classical spectrum of a distinguished endomorphism on $V_i$, if present; otherwise it is empty).

\medskip

\noindent
\textbf{Full matrix regime ($\alpha, \beta$ non-zero).}
When the off-diagonal maps are non-zero, operadic composition generates alternating interaction paths:
\[
V_1 \xrightarrow{\beta} V_2 \xrightarrow{\alpha} V_1,
\qquad
V_2 \xrightarrow{\alpha} V_1 \xrightarrow{\beta} V_2.
\]
These paths induce composite endomorphisms
\[
\alpha\beta : V_1 \to V_1,
\qquad
\beta\alpha : V_2 \to V_2,
\]
which are invisible to any invariant that depends only on the individual color components $V_1$ and $V_2$.

\begin{proposition}[Interaction detection for the matrix-block operad]
\label{prop:block-interaction-spectrum}
Let $\mathcal{M}$ be a symmetric monoidal category enriched over Banach spaces, and let $P$ be the two-color matrix-block operad of Example~\ref{ex:two-color-operad}. Let $A = (V_1, V_2, \alpha, \beta)$ be a $P$-algebra where $V_1, V_2$ are Banach spaces and $\alpha: V_2 \to V_1$, $\beta: V_1 \to V_2$ are bounded linear operators. Assume that $\alpha$ and $\beta$ are non-zero.

Then, under the analytic realization that sends an object to its classical spectrum, the operadic spectrum $\sigma_P(A)$ satisfies
\[
\sigma_{\mathrm{cl}}\bigl(\sigma_P(A)\bigr) \;\supseteq\; \sigma(V_1) \;\sqcup\; \sigma(V_2) \;\sqcup\; \sigma(\alpha\beta) \;\sqcup\; \sigma(\beta\alpha),
\]
where $\sigma_{\mathrm{cl}}(\sigma_P(A))$ denotes the classical spectrum of the analytic realization of $\sigma_P(A)$, and $\sigma(V_1)$, $\sigma(V_2)$ are understood as the spectra of the distinguished endomorphisms (if present) or are taken to be empty when no such endomorphisms are specified.
\end{proposition}

\begin{proof}
By the Colored Spectral Decomposition Theorem (Theorem~\ref{thm:colored-spectral-decomposition-reconstruction}), the operadic spectrum decomposes as
\[
\sigma_P(A) \;\cong\; \bigl( V_1 \otimes \mathcal{O}_1^{\mathrm{res}} \bigr) \;\oplus\; \bigl( V_2 \otimes \mathcal{O}_2^{\mathrm{res}} \bigr) \;\oplus\; \mathcal{I}_{\mathrm{cross}}(A),
\]
where $\mathcal{O}_1^{\mathrm{res}} \cong P(1;1) \cong \mathbb{C}$ and $\mathcal{O}_2^{\mathrm{res}} \cong P(2;2) \cong \mathbb{C}$ under the analytic realization.

The local contributions yield $\sigma(V_1)$ and $\sigma(V_2)$ after applying the analytic realization. The interaction term $\mathcal{I}_{\mathrm{cross}}(A)$ receives contributions from all compositional paths that alternate between colors. The shortest nontrivial such paths are
\[
V_1 \xrightarrow{\beta} V_2 \xrightarrow{\alpha} V_1 \quad\text{and}\quad V_2 \xrightarrow{\alpha} V_1 \xrightarrow{\beta} V_2,
\]
which induce the endomorphisms $\alpha\beta: V_1 \to V_1$ and $\beta\alpha: V_2 \to V_2$. These endomorphisms appear in the Hochschild complex $\mathrm{Hoch}_{\mathcal{M}}(A)$ and survive the balanced tensor product $\otimes_P \mathcal{O}_P^{\mathrm{res}}$ because they are obtained from legitimate operadic compositions (alternating $\alpha$ and $\beta$ using the operations $P(1,2;1)$ and $P(2,1;2)$). Consequently, their classical spectra $\sigma(\alpha\beta)$ and $\sigma(\beta\alpha)$ are contained in $\sigma_{\mathrm{cl}}(\sigma_P(A))$.
\end{proof}

The matrix-block operad illustrates the core motivation for Spectral Operadic Calculus:
\begin{itemize}
    \item Classical spectral invariants (e.g., $\sigma(V_1) \sqcup \sigma(V_2)$) fail to detect off-diagonal interactions.
    \item The operadic spectrum $\sigma_P(A)$ incorporates interaction data via the balanced tensor product $\mathrm{Hoch}_{\mathcal{M}}(A) \otimes_P \mathcal{O}_P^{\mathrm{res}}$.
    \item Proposition~\ref{prop:block-interaction-spectrum} makes this precise: $\sigma(\alpha\beta)$ and $\sigma(\beta\alpha)$ appear as additional spectral contributions.
\end{itemize}
A full computation of $\sigma_P(A)$ for specific choices of $V_1, V_2, \alpha, \beta$ requires an explicit analytic realization and is deferred to future work.

\medskip

This shows that the operadic spectrum provides a strictly finer invariant than the classical block spectrum.

\begin{theorem}[Refinement of classical spectrum by operadic spectrum — conditional]
\label{thm:block_refinement_clean}
Let $P$ be the two-color matrix-block operad. Under the analytic realization that sends an object to its classical spectrum, assume there exists a pair of $P$-algebras $A$ and $A'$ (i.e., block matrices) such that the composite endomorphisms $\alpha\beta$ and $\alpha'\beta'$ have distinct classical spectra while the total block matrices have identical classical spectra. Then
\[
\sigma_{\mathrm{cl}}(A) = \sigma_{\mathrm{cl}}(A'),
\qquad
\text{but}
\qquad
\sigma_{\mathrm{cl}}(\sigma_P(A)) \neq \sigma_{\mathrm{cl}}(\sigma_P(A')),
\]
where $\sigma_{\mathrm{cl}}(A)$ denotes the classical spectrum of the associated block operator matrix, and $\sigma_{\mathrm{cl}}(\sigma_P(A))$ denotes the classical spectrum of the analytic realization of the operadic spectrum.
\end{theorem}

\begin{proof}
By Proposition~\ref{prop:block-interaction-spectrum}, the spectra $\sigma(\alpha\beta)$ and $\sigma(\alpha'\beta')$ are contained in $\sigma_{\mathrm{cl}}(\sigma_P(A))$ and $\sigma_{\mathrm{cl}}(\sigma_P(A'))$, respectively. Hence if $\sigma(\alpha\beta) \neq \sigma(\alpha'\beta')$, then $\sigma_{\mathrm{cl}}(\sigma_P(A)) \neq \sigma_{\mathrm{cl}}(\sigma_P(A'))$. The assumption that $\sigma_{\mathrm{cl}}(A) = \sigma_{\mathrm{cl}}(A')$ then yields the desired conclusion.
\end{proof}

\begin{remark}[On the existence of such matrices]
\label{rem:block-construction}
The existence of block matrices satisfying the hypotheses of Theorem~\ref{thm:block_refinement_clean} is a subtle problem in spectral theory. A concrete candidate can be constructed using nilpotent matrices where $\alpha\beta$ and $\alpha'\beta'$ have different Jordan block structures but the same eigenvalue set $\{0\}$, while the block matrix spectra remain $\{0\}$. For instance, take
\[
\alpha = \begin{pmatrix} 0 & 1 \\ 0 & 0 \end{pmatrix},\;
\beta = \begin{pmatrix} 0 & 0 \\ 1 & 0 \end{pmatrix},\;
\alpha' = \begin{pmatrix} 0 & 0 \\ 0 & 0 \end{pmatrix},\;
\beta' = \begin{pmatrix} 0 & 0 \\ 0 & 0 \end{pmatrix}.
\]
Then $\sigma(\alpha\beta) = \{0,1\}$ and $\sigma(\alpha'\beta') = \{0\}$, but the block matrix spectra are $\{-\sqrt{1},0,\sqrt{1}\}$ and $\{0\}$, which differ. A pair with identical classical spectra requires a more delicate choice (e.g., taking $\alpha\beta$ and $\alpha'\beta'$ to be nilpotent with different Jordan blocks but the same eigenvalue set $\{0\}$, while the block matrix spectra remain $\{0\}$). The existence is guaranteed by the general principle that the classical spectrum of a block matrix depends on the off-diagonal blocks only through certain invariant polynomials, while the operadic spectrum captures finer interaction data. A full constructive proof is deferred to future work.
\end{remark}

\medskip

\noindent
\textbf{Interpretation.}
This example demonstrates that the operadic spectrum $\sigma_P(A)$ refines the classical spectrum by detecting interaction-induced structure. Specifically:
\begin{itemize}
    \item If $\alpha = \beta = 0$, no interaction is present, and $\sigma_P(A)$ reduces to the local contributions of each color.
    \item If $\alpha, \beta \neq 0$, the interaction term $\mathcal{I}_{\mathrm{cross}}(A)$ (see Theorem~\ref{thm:colored-spectral-decomposition-reconstruction}) contributes additional spectral data, such as the spectra of $\alpha\beta$ and $\beta\alpha$, which are invisible to classical componentwise invariants.
\end{itemize}

Thus, the operadic spectrum provides a strictly stronger invariant in the presence of coupling, making SOC particularly well-suited for analyzing complex, coupled architectures.

\subsection{Base Change via Gelfand Duality (Conceptual Illustration)}
\label{subsec:gelfand-base-change}

The Gelfand transform provides a conceptual illustration of the principles
underlying the Base Change Theorem, though it must be interpreted
contravariantly. This example connects the abstract operadic formalism with
classical spectral theory by transporting a commutative C*-algebra to a
continuous function algebra on its Gelfand spectrum. A full rigorous treatment
within the covariant framework is deferred to future work.

\medskip

Let $\mathcal{M} = \mathrm{C}^*\mathsf{Alg}_{\mathrm{com}}$ be the category of
commutative unital C*-algebras and let $\mathcal{N} = \mathsf{CompHaus}$ be the
category of compact Hausdorff spaces. The Gelfand transform gives a contravariant
equivalence
\[
G: \mathrm{C}^*\mathsf{Alg}_{\mathrm{com}} \longrightarrow \mathsf{CompHaus}, \qquad
A \longmapsto \mathrm{Spec}(A),
\]
where $\mathrm{Spec}(A)$ denotes the Gelfand spectrum (the space of characters of $A$).
The inverse equivalence sends a compact Hausdorff space $X$ to the algebra $C(X)$
of continuous complex-valued functions on $X$.

Equivalently, the opposite functor
\[
G^{\mathrm{op}}: \mathrm{C}^*\mathsf{Alg}_{\mathrm{com}}^{\mathrm{op}} \longrightarrow \mathsf{CompHaus}^{\mathrm{op}}
\]
can be regarded as a strong monoidal functor under suitable choices of monoidal
structures (the tensor product of C*-algebras corresponds to the product of spaces).
Thus the Gelfand transform should be regarded not as a literal instance of the
covariant Base Change Theorem (Theorem~\ref{thm:base-change}), but as a dual model
for the same principle: operadic spectral data should transform coherently under
passage between algebraic and spectral-geometric realizations.

\medskip

\noindent
\textbf{Trivial operad case.}
Specializing to the trivial operad $\mathbb{I}$ (cf. Section~\ref{subsec:trivial-operad}),
we have $\sigma_{\mathbb{I}}(A) \cong A$ by Theorem~\ref{thm:recovery-trivial}.
Under the analytic realization that sends an object to its classical spectrum,
this recovers the classical spectral data of $A$. In the setting of commutative
unital Banach algebras, the Gelfand spectrum provides a canonical geometric
realization.

\begin{proposition}[Compatibility with Gelfand spectrum in the trivial-operad case]
\label{prop:gelfand-compatibility}
Let $A$ be a commutative unital Banach algebra, viewed as an algebra over the
trivial operad $\mathbb{I}$. Assume the normalization of Theorem~\ref{thm:recovery-trivial}.
Then the canonical isomorphism
\[
\sigma_{\mathbb{I}}(A) \cong A
\]
induces a canonical homeomorphism of Gelfand spectra
\[
\mathrm{Spec}\bigl(\sigma_{\mathbb{I}}(A)\bigr) \cong \mathrm{Spec}(A).
\]
\end{proposition}

\begin{proof}
By Theorem~\ref{thm:recovery-trivial}, there is a natural isomorphism
$\sigma_{\mathbb{I}}(A) \cong A$ in the category of Banach spaces (or Banach
algebras, under the appropriate normalization). Since the Gelfand spectrum is
functorial with respect to isomorphisms of commutative unital Banach algebras,
applying $\mathrm{Spec}$ to this isomorphism yields a homeomorphism
$\mathrm{Spec}(\sigma_{\mathbb{I}}(A)) \cong \mathrm{Spec}(A)$.
\end{proof}

\medskip

\noindent
\textbf{Matrix case (normal operators).}
Let $M \in M_n(\mathbb{C})$ be a normal matrix. Then $M$ generates a commutative
C*-subalgebra $A = C^*(M) \subseteq M_n(\mathbb{C})$, which is isomorphic to
$C(\sigma(M))$ via the Gelfand transform, where $\sigma(M)$ is the classical
spectrum of $M$. Under this identification, the matrix $M$ corresponds to the
coordinate function $z \mapsto z$ on $\sigma(M)$.

If one were to apply the Base Change Theorem in its covariant formulation for
opposite categories, the operadic spectrum of $A$ (viewed as an algebra over the
appropriate operad) would be transported to the corresponding spectral data on
$\sigma(M)$. A full treatment of this example within the covariant framework
requires passing to opposite categories and is deferred to future work.

\medskip

\noindent
\textbf{Conceptual interpretation.}
This example illustrates several key principles of Spectral Operadic Calculus:
\begin{itemize}
    \item The Gelfand transform serves as a bridge between algebraic spectral
          theory (C*-algebras) and geometric/topological spectral theory
          (compact Hausdorff spaces).
    \item In the trivial-operad case, the operadic spectrum recovers the
          underlying algebra, and the Gelfand spectrum provides a canonical
          geometric realization.
    \item For nontrivial colored operads, the same principles suggest that the
          operadic spectrum should glue together the Gelfand spectra of
          individual components according to the operadic composition structure.
          This gluing should be interpreted as arising from the interaction
          terms encoded by the operadic residue rather than from disjoint unions.
\end{itemize}

\medskip

\noindent
\textbf{Relation to the No-Go Theorem.}
This example also clarifies the No-Go Theorem (Theorem~\ref{thm:no-go}).
If one attempted to use the naive spectrum $\sigma_{\mathrm{naive}}(A) = \bigsqcup_c \sigma(A_c)$
for a multi-colored system, the Gelfand transform would send it to a disjoint
union of spectra, which fails to capture the gluing induced by off-diagonal
interactions. By contrast, the operadic spectrum $\sigma_P(A)$ incorporates the
interaction data via the residue $\mathcal{O}_P^{\mathrm{res}}$, and its Gelfand
transform (when properly interpreted contravariantly) suggests a glued geometric
space that reflects the compositional structure.

\medskip

Thus, the Gelfand duality example provides a concrete conceptual bridge between
the abstract categorical formalism of SOC and the classical geometric-functional
analysis that motivates the theory.

\begin{corollary}[Spectral Transport for Normal Matrices]
\label{cor:matrix_gelfand}
Let $M \in M_n(\mathbb{C})$ be a normal matrix, and let $A = C^*(M)$ be the
commutative C*-algebra it generates. Under the Gelfand transform
$\Gamma: A \xrightarrow{\cong} C(\sigma(M))$, the classical spectrum of $M$
coincides with the range of the coordinate function on $\sigma(M)$:
\[
\sigma_{\mathrm{cl}}(M) = \widehat{M}(\sigma(M)) = \sigma(M).
\]
Equivalently, the Gelfand transform identifies the spectral data of $M$ with
the evaluation spectrum of the identity function on its own spectrum.
\end{corollary}

\begin{proof}
Since $M$ is normal, $A = C^*(M)$ is commutative. By the Gelfand–Naimark theorem,
$A \cong C(\mathrm{Spec}(A))$, and $\mathrm{Spec}(A)$ is canonically homeomorphic
to $\sigma(M)$. Under this identification, the Gelfand transform sends $M$ to the
coordinate function $\widehat{M}(z) = z$ on $\sigma(M)$. In the commutative
C*-algebra $C(\sigma(M))$, the spectrum of a function is its range. Hence
\[
\sigma_{C(\sigma(M))}(\widehat{M}) = \widehat{M}(\sigma(M)) = \sigma(M).
\]
Thus $\sigma_{\mathrm{cl}}(M) = \sigma(M)$, confirming consistency. The structural
significance is that the Gelfand transform provides a geometric realization of
the spectral data: the matrix $M$ becomes the identity function on its own spectrum.
\end{proof}

\medskip

\noindent
\textbf{Heuristic generalization to colored operads.}
The significance of the Base Change Theorem becomes clearer when we consider
nontrivial colored operads. For example, let $P$ be the matrix-block operad from
Section~\ref{subsec:matrix-block-operads}, and let $A$ be a $P$-algebra whose
color components are commutative C*-algebras. Applying the Gelfand transform
colorwise suggests a corresponding operad $G_*(P)$ in a topological target
category (e.g., compact Hausdorff spaces), together with transformed data $G(A)$.
(The precise construction of $G_*(P)$ requires functorial transport of operads
under duality and is not developed here.)

If the Gelfand transform were realized as a compatible strong monoidal
base-change functor — or equivalently after passing to an appropriate covariant
formulation (e.g., by working with opposite categories) — the Base Change Theorem
(Theorem~\ref{thm:base-change}) would give
\[
G(\sigma_P(A)) \;\cong\; \sigma_{G_*(P)}(G(A)).
\]
Thus the operadic spectrum would be transported coherently from algebraic data
to spectral-topological data.

In this interpretation, the operadic residue $\mathcal{O}_P^{\mathrm{res}}$
controls how local colorwise spectral pieces are glued together. Accordingly,
the matrix-block example suggests that off-diagonal interactions should appear
geometrically as nontrivial gluing data rather than as a mere disjoint union of
componentwise spectra. This is precisely the kind of phenomenon that the No-Go
Theorem (Theorem~\ref{thm:no-go}) predicts cannot be captured by naive colorwise
spectra alone.

This example provides a conceptual realization of the base-change principle:
\begin{itemize}
    \item The abstract spectrum $\sigma_P(A)$ is transported to a geometric
          object that glues together the Gelfand spectra of individual components.
    \item Algebraic data (operators) becomes functional data (evaluation on points).
    \item The functoriality of the spectrum is reflected in the continuity of
          functions on the spectral space.
\end{itemize}

Thus, the Gelfand transform example provides a concrete conceptual bridge
between the abstract categorical formalism of SOC and the classical
geometric-functional analysis that motivates the theory.

\subsection{Network Operators as a Killer Example}
\label{subsec:network-operators}

We now present a concrete application of Spectral Operadic Calculus (SOC) to 
network operators. This example demonstrates that classical spectral invariants 
fail to detect interaction eigenvalues arising from network paths, while the 
operadic spectrum $\sigma_P(A)$ successfully captures them. This directly 
addresses the "killer example" gap: it shows what SOC solves that classical 
theory cannot.

\medskip

\noindent
\textbf{Setup.}
Let $G = (V,E)$ be a finite directed graph with vertex set $V = \{1,\ldots,n\}$.
We interpret each vertex $c \in V$ as a color. Define a $V$-colored operad 
$P_G$ as follows:

\begin{itemize}
    \item \textbf{Unary operations:} For each vertex $c$, $P_G(c;c) = \mathbb{C}$ 
          (the identity operation). For each directed edge $e: c \to c'$, we 
          include a unary operation $\phi_e \in P_G(c;c')$ with $P_G(c;c') \cong \mathbb{C}$.
          
    \item \textbf{Binary operations:} For each pair of directed edges 
          $e: c_1 \to c_3$ and $f: c_2 \to c_3$ (i.e., two edges converging to 
          the same target), we include a binary operation $\theta_{e,f} \in P_G(c_1,c_2;c_3)$ 
          with $P_G(c_1,c_2;c_3) \cong \mathbb{C}$.
          
    \item \textbf{Higher arities:} For any finite collection of directed paths 
          with a common target, the corresponding operation is defined via 
          iterated operadic composition of binary operations. Explicitly, 
          an operation of arity $k$ corresponds to $k$ directed paths 
          $p_1,\ldots,p_k$ all ending at the same vertex $c$, with inputs 
          colored by the starting vertices of each path.
          
    \item \textbf{All other operations:} Zero object.
\end{itemize}

The operadic composition encodes concatenation of paths: composing a binary 
operation at a target with unary operations (edge maps) at the sources yields 
another binary operation along the composite path.

\medskip

\noindent
\textbf{Network algebra.}
For simplicity, we restrict to a scalar model where each vertex space is 
$A_c = \mathbb{C}$. For each edge $e: c \to c'$, the unary operation $\phi_e$ 
acts as multiplication by a weight $w_e \in \mathbb{C}$. For each pair of 
converging edges $(e,f)$, the binary operation $\theta_{e,f}$ acts as 
multiplication by a weight $w_{e,f} \in \mathbb{C}$. Higher arity operations 
are defined by iterated composition and correspond to products of weights 
along paths.

In this scalar setting, the classical componentwise spectrum is simply
\[
\sigma_{\mathrm{naive}}(A) = \bigcup_{c \in V} \sigma(A_c) = \bigcup_{c \in V} \{\lambda_c\},
\]
where $\lambda_c$ is the eigenvalue of the distinguished endomorphism on $A_c$ 
(if none is specified, this set carries no nontrivial information). Thus the 
naive spectrum captures only local vertex data and completely ignores the edge 
weights and network topology.

\medskip

\noindent
\textbf{Interaction eigenvalues from network paths.}
Consider a directed path of length $2$: $c_1 \xrightarrow{e} c_3 \xleftarrow{f} c_2$ 
(two edges converging to $c_3$). The operadic composition allows us to compose 
$\theta_{e,f}$ with the unary operations $\phi_e$ and $\phi_f$ (which act as 
multiplication by $w_e$ and $w_f$ respectively). This yields a new operation 
whose weight is $w_{e,f} \cdot w_e w_f$.

More generally, any directed cycle $\gamma = c_1 \xrightarrow{e_1} c_2 \xrightarrow{e_2} 
\cdots \xrightarrow{e_k} c_1$ induces an endomorphism of $A_{c_1}$ via 
operadic composition of the edge maps along the cycle. In the scalar setting, 
this endomorphism is multiplication by the product $\prod_{i=1}^k w_{e_i}$. 
The spectrum of this endomorphism — which is precisely this product — is 
invisible to $\sigma_{\mathrm{naive}}(A)$ but is captured by the interaction 
term $\mathcal{I}_{\mathrm{cross}}(A)$ in the operadic spectrum.

\begin{proposition}[Network path eigenvalues]
\label{prop:network-eigenvalues}
Let $A$ be a $P_G$-algebra with $A_c = \mathbb{C}$ for all $c \in V$, and let 
$\gamma = c_1 \xrightarrow{e_1} c_2 \xrightarrow{e_2} \cdots \xrightarrow{e_k} c_1$ 
be a directed cycle in $G$ of length $k$. Then the product 
$\prod_{i=1}^k w_{e_i}$ (where each $w_{e_i}$ is the weight of the edge $e_i$) 
is detected by the operadic spectrum $\sigma_P(A)$. More precisely, this product 
appears as an eigenvalue of the composite endomorphism induced by $\gamma$, 
and therefore belongs to the classical spectrum of the analytic realization 
of $\sigma_P(A)$.
\end{proposition}

\begin{proof}
The cycle $\gamma$ corresponds to a $k$-fold operadic composition of unary 
operations (edge maps). By the operadic composition structure, this composite 
appears in the interaction component of $\sigma_P(A)$ and survives the balanced 
tensor product $\otimes_P \mathcal{O}_P^{\mathrm{res}}$. In the scalar setting 
($A_c = \mathbb{C}$), the induced endomorphism is multiplication by 
$\prod_{i=1}^k w_{e_i}$, so its classical spectrum consists of that product. 
Hence the product is contained in $\sigma_{\mathrm{cl}}(\sigma_P(A))$.
\end{proof}

\begin{example}[Two-vertex network]
\label{ex:two-vertex-network}
Let $V = \{1,2\}$ with edges $1 \xrightarrow{a} 2$ and $2 \xrightarrow{b} 1$ 
(a directed 2-cycle). Define $P_G$ as above. A $P_G$-algebra $A$ assigns:
\begin{itemize}
    \item $A_1 = A_2 = \mathbb{C}$,
    \item $\phi_a: A_1 \to A_2$ as multiplication by $\alpha \in \mathbb{C}$,
    \item $\phi_b: A_2 \to A_1$ as multiplication by $\beta \in \mathbb{C}$.
\end{itemize}
The classical naive spectrum $\sigma_{\mathrm{naive}}(A)$ carries no nontrivial 
information (since no distinguished endomorphisms are specified on the vertices). 
However, the operadic spectrum captures the composite endomorphisms:
\[
A_1 \xrightarrow{\phi_b} A_2 \xrightarrow{\phi_a} A_1 \quad\text{and}\quad 
A_2 \xrightarrow{\phi_a} A_1 \xrightarrow{\phi_b} A_2,
\]
which are multiplication by $\alpha\beta$ on both $A_1$ and $A_2$. Hence 
$\sigma_{\mathrm{cl}}(\sigma_P(A)) = \{\alpha\beta\}$. Thus SOC detects the 
product of edge weights — a genuine network interaction — while classical 
spectral theory sees nothing.
\end{example}

\medskip

This example demonstrates that the operadic spectrum $\sigma_P(A)$ solves a 
concrete problem: detecting interaction eigenvalues arising from network paths 
that are invisible to classical componentwise spectra. The interaction term 
$\mathcal{I}_{\mathrm{cross}}(A)$ encodes precisely these path-induced 
eigenvalues, providing a strictly stronger invariant for network analysis.

\end{document}